\documentclass[11pt]{amsart}
\usepackage[margin=1in,marginparwidth=0.8in, marginparsep=0.1in]{geometry}
\usepackage[toc,page]{appendix}
\usepackage{tikz-cd}
\usepackage{graphicx}
\usepackage{amssymb}
\usepackage{epstopdf}
\usepackage{mathabx}
\usepackage{amsthm}
\usepackage{amsfonts}
\usepackage{amscd}
\usepackage{hyperref}

\numberwithin{equation}{section}

\newcommand{\C}{\mathbb{C}}     
\newcommand{\Q}{\mathbb{Q}}     
\newcommand{\R}{\mathbb{R}}     
\newcommand{\Z}{\mathbb{Z}}      
\newcommand{\A}{\mathcal{A}}    
\newcommand{\M}{\mathcal{M}}    

\newcommand{\tr}{\operatorname{tr}}
                             
\DeclareMathOperator{\Tr}{Tr}           
\DeclareMathOperator{\Aug}{Aug}        
\DeclareMathOperator{\Crit}{Crit}        
\DeclareMathOperator{\im}{im}        
\DeclareMathOperator{\codim}{codim}        
\DeclareMathOperator{\Spec}{Spec}        
\DeclareMathOperator{\RL}{\R\times \Lambda_\mathit{K}}        
\DeclareMathOperator{\Md}{\mathit{M}_\mathit{K}^\delta}        
\DeclareMathOperator{\Ld}{\mathit{L}_\mathit{K}^\delta}        
\DeclareMathOperator{\LK}{\Lambda_\mathit{K}}        
\DeclareMathOperator{\An}{An}        
\DeclareMathOperator{\ev}{ev}
\DeclareMathOperator{\coker}{coker}
\DeclareMathOperator{\Span}{Span}

\newcommand{\MMa}{\mathcal{M}_{\mathrm{an}}(c)}  
\newcommand{\MMb}{\mathcal{M}(c;L)}  
\newcommand{\MMc}{\mathcal{M}_{\mathrm{an}}(c)\times_{\ev_{\partial}}\mathcal{M}(L)} 
\newcommand{\MMd}{\mathcal{M}(c)\times_{\ev_{\partial}}\mathcal{M}_{\mathrm{an}}(\R^3,L)} 
\newcommand{\MMe}{\mathcal{M}_{\mathrm{an}}(c,\boldsymbol{a})}
\newcommand{\MMf}{\mathcal{M}(c,\xi)}
\newcommand{\MMg}{\mathcal{M}(c,\eta)}
\newcommand{\MMh}{\mathcal{M}(c,\sigma)}
\newcommand{\MMi}{\mathcal{M}(c)\times_{\ev_{\partial}}\mathcal{M}_{\mathrm{an}}(\R^3,L)} 
\newcommand{\MMj}{\mathcal{M}^{\infty}(c,\sigma)}
\newcommand{\MMk}{\mathcal{M}(c,\boldsymbol{a},\sigma)} 
\newcommand{\cdts}{\cup \qquad \cdots} 
\newcommand{\MMl}{\mathcal{M}(c,V)}
\newcommand{\MMm}{\mathcal{M}(c,L)}
\newcommand{\MMn}{\mathcal{M}^{\infty}(c, \partial^{\infty}V)}
\newcommand{\MMo}{\mathcal{M}(c,\gamma)}
\newcommand{\MMp}{\mathcal{M}(c,\boldsymbol{a},V)}
\newcommand{\MMq}{\mathcal{M}(c,\sigma_\tau)}
\newcommand{\MMr}{\mathcal{M}(c,\gamma)}
\newcommand{\MMs}{\mathcal{M}^{\infty}(c,\tau)}
\newcommand{\MMt}{\mathcal{M}(c,\boldsymbol{a},\sigma_\tau)}
\newcommand{\MMu}{\mathcal{M}(c,\xi_{i},\eta_{j})}
\newcommand{\MMv}{\bigcup_{k} \mathcal{M}(c,\xi_{i},\xi_{k})\times\mathcal{M}(\xi_{k},\eta_{j})}
\newcommand{\MMw}{\bigcup_{k} \mathcal{M}(\xi_{i},\eta_{k})\times \mathcal{M}(c,\eta_{k},\eta_{j})}
\newcommand{\MMx}{\mathcal{M}(c)\times_{\ev_{\partial}} \mathcal{M}(\xi_{i},\eta_{j})}
\newcommand{\MMy}{\mathcal{M}(c,\boldsymbol{a},\xi_{i},\eta_{j})}
\newcommand{\MMz}{\mathcal{M}(c,\xi_i,\eta_j,\tilde\sigma)}
\newcommand{\MMaa}{\mathcal{M}(c)\times_{\ev_{\partial}} \mathcal{M}(\xi_{i},\eta_{j})}

\newcommand{\MMac}{\mathcal{M}^{\infty}(c,\xi_i,\eta_j,\tilde\sigma)}
\newcommand{\MMad}{\mathcal{M}(c,\boldsymbol{a},\xi_i,\eta_j,\tilde\sigma)}
\newcommand{\MMae}{\mathcal{M}(c,\xi_{i},\xi_{j})}
\newcommand{\MMaf}{\mathcal{M}(c,\eta_{i},\eta_{j})}
\newcommand{\MMag}{\mathcal{M}(\xi_{i},\eta_{j})}

\newcommand{\NNa}{\mathcal{N}(c,L_K\cup \R^3)}  
\newcommand{\NNb}{\mathcal{N}(c,\boldsymbol{a},L_K\cup \R^3)}  
\newcommand{\NNc}{\mathcal{N}(c,K)}

\newcommand{\bdryc}{\Theta^\lambda(c)\in C_*^\lambda} 
\newcommand{\bdrya}{\Theta^\lambda(a)\in C_{*-1}^\lambda}

\theoremstyle{plain} \newtheorem{theorem}{Theorem}[section]            
\theoremstyle{plain} \newtheorem{proposition}[theorem]{Proposition}       
\theoremstyle{plain} \newtheorem{lemma}[theorem]{Lemma}               
\theoremstyle{plain} \newtheorem{corollary}[theorem]{Corollary}              
\theoremstyle{plain} \newtheorem{definition}[theorem]{Definition}             
\theoremstyle{plain} \newtheorem{conjecture}[theorem]{Conjecture}           
\theoremstyle{remark} \newtheorem{remark}[theorem]{Remark}              
\theoremstyle{remark} \newtheorem{example}[theorem]{Example}           
\theoremstyle{remark} \newtheorem*{claim*}{Claim}
\theoremstyle{remark}             

\newlength{\oldfb}
\newcommand{\boxx}[1]
{\setlength{\oldfb}{\fboxrule}\setlength{\fboxrule}{2pt}\framebox{\parbox{\dimexpr\linewidth-2\fboxsep-2\fboxrule}{#1}}\setlength{\fboxrule}{\oldfb}\\} 

\begin{document}

\title{Augmentations, annuli, and Alexander polynomials}
\date{\today}
\author{Lu\'is Diogo}
\address{Department of mathematics, Uppsala University, Box 480, 751 06 Uppsala, Sweden}
\email{luis.diogo@math.uu.se}
\author{Tobias Ekholm}
\address{Department of mathematics, Uppsala University, Box 480, 751 06 Uppsala, Sweden and
Institut Mittag-Leffler, Aurav 17, 182 60 Djursholm, Sweden}
\email{tobias.ekholm@math.uu.se}

\begin{abstract}
The augmentation variety of a knot is the locus, in the 3-dimensional coefficient space of the knot contact homology dg-algebra, where the algebra admits a unital chain map to the complex numbers. We explain how to express the Alexander polynomial of a knot in terms of the augmentation variety: it is the exponential of the integral of a ratio of two partial derivatives. The expression is derived from a description of the Alexander polynomial as a count of Floer strips and holomorphic annuli, in the cotangent bundle of Euclidean 3-space, stretching between a Lagrangian with the topology of the knot complement and the zero-section, and from a description of the boundary of the moduli space of such annuli with one positive puncture.  
\end{abstract}
\maketitle
\tableofcontents

\section{Introduction}\label{Sec Intro}
Let $K\subset \R^3$ be a knot and let $\Delta_{K}(\mu)$ denote its Alexander polynomial. We study the relation between $\Delta_{K}(\mu)$ and holomorphic curve invariants of Lagrangian submanifolds of $T^{\ast}\R^{3}$ naturally associated to $K$. Very briefly, we show how the augmentation variety of a knot \cite{NgFramed} (which is related to holomorphic disk counts for the knot conormal after conifold transition by \cite{AENV}), determines the Alexander polynomial of the knot (which is related to counts of flow loops in the knot complement by \cite{HutchingsLee}, and to counts of holomorphic annuli in $T^*\R^3$, see Lemma \ref{annuli loops}). 

We first introduce the holomorphic curve invariants we need and then state our main result and discuss its ramifications.
Let $ST^{\ast}\R^3$ denote the unit co-sphere bundle of $\R^3$. The restriction of the Liouville form $p\,dq$ to $ST^{\ast}\R^3$ is a contact form, and the submanifold of unit covectors over $K$ that over each point $p\in K$ annihilate the tangent vectors to $K$ at $p$ is a Legendrian torus $\Lambda_{K}\subset ST^{\ast}\R^3$. 

We describe two Lagrangian fillings $L_{K}$ and $M_{K}$ in $T^{\ast}\R^3$ of $\Lambda_{K}$, that are naturally associated to $K$. The first, $L_{K}$, is the Lagrangian conormal of $K$ that consists of all (not necessarily unit length) co-vectors along $K$ that annihilate the vectors tangent to $K$. The conormal $L_{K}$ is diffeomorphic to a solid torus, $L_{K}\approx S^{1}\times\R^{2}$. The second, $M_{K}$, is obtained from $L_{K}$ and the zero-section $\R^3$: $L_{K}$ intersects $\R^3$ cleanly along $K$ and the Lagrangian $M_{K}$ is obtained by Lagrange surgery on this clean intersection. The Lagrangian $M_{K}$ is diffeomorphic to the knot complement $M_{K}\approx \R^3\setminus K$. 

The \emph{Chekanov--Eliashberg algebra $\mathcal{A}(\Lambda_{K})$} of $\Lambda_{K}$ is a dg-algebra freely generated by the Reeb chords of $\Lambda_{K}$ with coefficients in $\C[H_{2}(ST^{\ast}\R^3,\Lambda_{K})]$, the group algebra of the second relative homology group of $\Lambda_{K}$. For the contact form $p\,dq$, Reeb chords of $\Lambda_K$ correspond to geodesics in $\R^3$ that are binormal to $K$, and the grading of a chord equals the Morse grading of the corresponding binormal geodesic (which is $\ge 0$). 
Pick a basis $x, p, t$ for $H_{2}(ST^{\ast}\R^3,\Lambda_{K})$, where $x$ maps to the longitude of the knot under the connecting homomorphism, $p$ maps to the meridian, and $t$ is a generator of $H_{2}(ST^{\ast}\R^3)$. Here we use the natural identification of $\Lambda_{K}$ with the boundary of a tubular neighborhood of $K$. This choice of basis induces an identification of $\C[H_{2}(ST^{\ast}\R^3,\Lambda_{K})]$ with $\C[e^x, e^p, e^t]$.  
The differential in the dg-algebra counts rigid holomorphic disks in the symplectization $\R\times ST^{\ast}\R^3$ with boundary in $\R\times\Lambda_{K}$, with one positive and several negative punctures at Reeb chords. The homology of $\mathcal{A}(\Lambda_{K})$ is called the \emph{knot contact homology of $K$}.

An \emph{augmentation of $\mathcal{A}(\Lambda_{K})$} is a dg-algebra map $\epsilon\colon \mathcal{A}(\Lambda_{K})\to\C$. Given such a chain map, the dg-algebra differential induces a differential on the chain complex $\ker(\epsilon)/\ker(\epsilon)^{2}$, whose homology is called the $\epsilon$-linearized homology of $\mathcal{A}(\Lambda_{K})$. Denoting the collection of Reeb chords of degree 0 by $\boldsymbol{a}$, the full augmentation variety $\tilde V_K$ is the set of values $(e^{x},e^{p},e^{t},\boldsymbol{a})\in (\C^{\ast})^{3}\times \C^{|\boldsymbol{a}|}$ associated to such an augmentation. Let $\pi\colon  (\C^{\ast})^{3}\times \C^{|\boldsymbol{a}|} \to  (\C^{\ast})^{3}$ be the projection. The top-dimensional stratum $V_{K}$ of the Zariski-closure of $\pi(\tilde V_K) \subset (\C^{\ast})^{3}$ is the augmentation variety. 

It is well-known that for every knot $K$ the conormal Lagrangian $L_K$ is an exact Lagrangian filling of $\Lambda_K$, and that it induces a family of augmentations of $\mathcal{A}(\Lambda_{K})$ with $e^p=e^t=1$ and arbitrary $e^x\in \C^*$. Denote by $\boldsymbol{\epsilon_0}\in \tilde V_K$ the point corresponding to the augmentation in this family with $e^x = 1$. 
We show the following result about the full augmentation variety near this point. 

\begin{theorem}\label{t:linearizedhomology} For any knot $K$, the full augmentation variety $\tilde V_{K}$ is smooth and 2-dimensional in a neighborhood $\tilde V^{0}_{K}$ (with respect to the standard metric topology on $(\C^{\ast})^{3}\times \C^{|\boldsymbol{a}|}$)  of 
$\boldsymbol{\epsilon_0}$. Also, $d_{\boldsymbol{\epsilon_0}}\pi$ gives an isomorphism between $T_{\boldsymbol{\epsilon_0}}\tilde V_{K}^{0}$ and $\Span(\partial_x,\partial_p)$. 
For every augmentation in a neighborhood $\tilde U^{0}_{K}\subset \tilde V^{0}_{K}$ of $\boldsymbol{\epsilon_0}$, the corresponding linearized contact homology has rank one in degrees one and two, and rank zero otherwise. 
\end{theorem}

Theorem \ref{t:linearizedhomology} is proved in Section \ref{LCH VK}. It implies that we can parametrize the augmentations in $\tilde U^{0}_{K}$ by a holomorphic function of the variables $(e^x,e^p)$. The theorem also implies that the augmentation variety $V_{K} \subset (\C^*)^3$ contains a smooth 2-dimensional surface $V_K^0=\pi(\tilde V_K^0)$ through the point $(1,1,1)$, parametrized holomorphically by $(e^x,e^p)$. In particular, $V_K$ is of codimension at most one, and is cut out by a single polynomial, called the augmentation polynomial and denoted by $\Aug_K(e^x,e^p,e^t)$. 

Parametrizing augmentations in $\tilde U^{0}_{K}$ by $(e^x,e^p)$ as above, the image of the linearized contact homology differential acting on a Reeb chord of degree $k$ is a linear combination of chords of degree $k-1$, with coefficients given by holomorphic funtions of $(e^x,e^p)$. Applying Gauss elimination, we then find a family $y(e^{x},e^{p})$ of degree $1$ cycles generating the degree $1$ linearized contact homology, and depending holomorphically on $(e^x,e^p)$.  
We show in Section \ref{sec:proof of main thm} that if $F_K(e^{x},e^{p},e^{t})$ is given by the count of augmented disks with positive punctures in $y(e^{x},e^{p})$, see Proposition \ref{prp:about F}, then $dF_K(1,1,1)\ne 0$ and $d\Aug_{K}(e^{x},e^{p},e^{t})=h(e^{x},e^{p},e^{t})\, dF_K(e^{x},e^{p},e^{t})$ for some holomorphic function $h$. Our main result can then be stated as follows.

\begin{theorem} \label{T:Alex Aug}
	Let $K\subset \R^3$ be a knot, $\Aug_K(e^{x},e^{p},e^{t})$ its augmentation polynomial, $\Delta_K(e^{p})$ its Alexander polynomial, and $F_{K}(e^{x},e^{p},e^{t})$ the count of disks described above. 
	Then,
	\begin{align} \label{Alex from Aug 1}
		\Delta_K(e^{p}) &=(1-e^{p}) \exp \left.\left(\int - \frac{\partial_t F_K}{\partial_x F_K} \right|_{(x,t) = (0,0)} \, dp \right).
	\end{align}
	In particular, if $\partial_x \Aug_K |_{(x,t) = (0,0)} \neq 0$, then
	\begin{align} \label{Alex from Aug 2}
	\Delta_K(e^{p}) &=(1-e^{p}) \exp \left.\left(\int - \frac{\partial_t \Aug_K}{\partial_x \Aug_K} \right|_{(x,t) = (0,0)} \, dp \right).
	\end{align}
\end{theorem}

%
%
Theorem \ref{T:Alex Aug}\footnote{See Remark \ref{NgAlexander}  for a reference to earlier work by Ng showing that one can obtain the Alexander polynomial from knot contact homology.} is proved in Section \ref{sec:main proof}. 

\begin{remark} \label{rmk:antiderivatives}
The integrals in \eqref{Alex from Aug 1} and  \eqref{Alex from Aug 2} should be interpreted as taking an anti-derivative in the variable $p$. More explicitly, they are functions given by 
$$
\int_0^p - \frac{(\partial_t F_K)(1,e^{p'},1)}{(\partial_x F_K)(1,e^{p'},1)}\, dp' + C,\quad\text{ and }\quad
\int_0^p - \frac{(\partial_t \Aug_K)(1,e^{p'},1)}{(\partial_x \Aug_K)(1,e^{p'},1)}\, dp' + C,
$$
respectively, for a suitable constant $C$. One can impose a normalization of the Alexander polynomial (for example $|\Delta_K(1)|=1$) by a suitable choice of $C$. (The indeterminate constant $C$ has its origin in the proof of Theorem \ref{Alex F general}, where an annulus counting function $\An_{M_K}(e^p)$ is expressed as an indefinite integral $\int \partial_p \An_{M_K}(e^p) dp$.)
\end{remark}

The proof of Theorem \ref{T:Alex Aug} uses several cobordisms of moduli spaces of holomorphic curves. The starting point for the argument is Theorem \ref{T:Alex zeta tau} below. 
It interprets a dynamical formula for the Alexander polynomial of a knot $K\subset \R^3$, consisting of counts of gradient flow lines and loops in the knot complement, in terms of holomorphic annuli and strips in $T^*\R^3$. See \cite{MilnorDuality,FriedHomological,HutchingsLee} for the dynamical formula, and Section \ref{ssec:MorseAlex} for an alternative Morse-theoretic proof. Let $\Md$ denote a shift of the exact Lagrangian knot complement along a generic closed 1-form representing a generator of the first cohomology $H^{1}(M_{K})$. The following result is proved in Section \ref{ssec:annuli} using flow tree techniques from \cite{EkholmFlowTrees}.
\begin{theorem} \label{T:Alex zeta tau}
Let $K\subset \R^3$ be a knot. For any almost complex structure on $T^{\ast}\R^{3}$ agreeing with the standard almost complex structure along the zero-section $\R^{3}$ and for all sufficiently small shifts of $M_{K}$ to $M_{K}^{\delta}$, the Alexander polynomial of $K$ can be written as
\begin{equation} \label{Alex zeta tau}
\Delta_K(e^{p}) = (1-e^{p}) \zeta_{\rm{an}}(e^p) \cdot \tau_{\rm{str}}(e^p), 
\end{equation}
where $\zeta_{\rm{an}}(e^p)$ is the exponential of a generating function of holomorphic annuli in $T^*\R^3$ stretching from $\R^3$ to $\Md$, and $\tau_{\rm{str}}(e^p)$ counts holomorphic strips between $\R^3$ and $\Md$.
\end{theorem}

The right-hand side of \eqref{Alex from Aug 2} in Theorem \ref{T:Alex Aug} makes sense also at points in the augmentation variety $V_K$, where $(x,t)\ne (0,0)$. This gives a $Q$-deformation of the Alexander polynomial ($Q=e^{t}$), which connects our work to topological string and physics invariants in low-dimensional topology. More precisely, this observation is the starting point for a geometric treatment and $Q$-deformation (there called $a$-deformation emphasizing the role of $Q=a^{2}$ in colored HOMFLY polynomials) of the so called $\widehat{Z}$-invariant of 3-manifolds, originating from a 3d-3d correspondence (see \cite{GPV,GPPV}) for knot complements, as described in \cite[Section 4]{EGGKPS}. We elaborate on this in Section \ref{sec:SFT-stretch}. 

For fibered knots, we prove in Proposition \ref{prp:limitfibered} that the quotient by $1-e^p$ of this 
$Q$-deformation of the Alexander polynomial, where $Q=e^{t}$, can be written as $\tau_K = \exp\left(\partial_{t} U_{K}^{0}(p,t)\right)$, where $U_{K}^{0}$ is a certain Gromov--Witten disk potential for the knot complement Lagrangian $M_{K}$. In Conjecture \ref{conj1}, we propose that a generalization of Proposition \ref{prp:limitfibered} holds also for non-fibered knots.  
\begin{remark}
The logarithm of $\tau_K$ can be interpreted as the dual coordinate $s=s(x,t)$ or $s=s(p,t)$ of the deformation coordinate $t$ in an extended (holomorphic) Lagrangian augmentation variety in $(\C^{\ast})^{4}$. This extended augmentation variety is Lagrangian with respect to the symplectic form $dx\wedge dp + dt\wedge ds$, and given by the generating function $U_K^{0}(p,t)$.	
\end{remark}

The paper is organized as follows. In Section \ref{sec alex+grad}, we recall a formula for the Alexander polynomial of a knot $K\subset S^3$ in terms of gradient flow loops and gradient flow lines in the knot complement, and give a Morse-theoretic proof of that formula. In Section \ref{sec:flow lines and holo curves}, we study the Lagrangians $L_K$ and $M_K$ and their non-exact deformations, analyze holomorphic disks and annuli with boundaries in these Lagrangians and in $\R^3$, and prove Theorem \ref{T:Alex zeta tau}. In Section \ref{choices}, we discuss coefficients in knot contact homology and the augmentation variety of a knot. In Section \ref{sec:main proof}, we study several cobordisms of moduli spaces of holomorphic curves. In Section \ref{sec:LCH computations}, we compute linearized knot contact homology groups and use the results of the previous section to prove Theorems \ref{t:linearizedhomology} and Theorem \ref{T:Alex Aug}. We also relate the dependence of choices in that result with well-known properties of the Alexander polynomial. In Section \ref{sec Ex}, we present some examples to illustrate our results. In Section \ref{sec:SFT-stretch}, we discuss how Theorem \ref{T:Alex Aug} relates to physical invariants from the point of view of Gromov--Witten disk potentials and Floer homology torsion. As was mentioned, Section \ref{sec:SFT-stretch} is the only part of the paper that relies on abstract perturbations for holomorphic curves. 

\subsection*{Acknowledgements} 
LD thanks Paul Seidel for a suggestion that led to the start of this project. We thank Lenny Ng for helpful comments. LD was supported by the Knut and Alice Wallenberg Foundation. TE was supported by the Knut and Alice Wallenberg Foundation KAW 2020.0307 and the Swedish Research Council VR 2020--04535. 
The authors thank anonymous referees for useful comments and suggestions.

\section{The Alexander polynomial and gradient flows}\label{sec alex+grad} 
In this section, we discuss how to express the Alexander polynomial in terms of counts of gradient flow lines and gradient flow loops of $S^{1}$-valued Morse functions. The results are well-known.
We state them in Section \ref{ssec:dynAlex} and give a direct Morse-theoretic derivation in Section \ref{ssec:MorseAlex}. 

\begin{remark}[About the ambiguity of the Alexander polynomial] \label{rmk: Alexander ambiguity}
Let $K \subset S^3$ be a knot. The Alexander polynomial $\Delta_K(\mu) \in \Z[\mu^{\pm}]$ is well-defined up to multiplication by a unit $\pm \mu^k$, for some $k\in \Z$, and so is its dynamical formulation in \eqref{alex} below (see for instance \cite[Example 1.4 and Theorem 1.7]{HutchingsLee}). Therefore, the formulas for $\Delta_K(\mu)$ in Theorems \ref{T:Alex Aug} and \ref{T:Alex zeta tau} make sense only up to the same ambiguity. In particular, in Theorem \ref{T:Alex Aug} the antiderivatives are specified up to additive constants of the form $k \pi i$, for $k\in \Z$. In Section \ref{sec: independence of choices} we explain how the choices involved in the definition of knot contact homology (for instance, of capping disks for the longitude and meridian of $K$) are compatible with the ambiguity by a factor $\pm\mu^k$ in the definition of $\Delta_K$.

To simplify the exposition, we will use different coefficients in this section. In Subsection \ref{ssec:dynAlex}, we use formal Laurent series $\Q((\mu))$ (which is a field), and in Subsection \ref{ssec:MorseAlex} we use the Laurent polynomial ring $\Q[\mu^{\pm}]$ (which is a principal ideal domain). 
%
\end{remark}

\subsection{Dynamical definition}\label{ssec:dynAlex}
Let $K\subset S^3$ be a knot and $\nu(K)$ be a small tubular neighborhood of $K$. Choose coordinates $(r,\phi,\theta)\in (0,\epsilon)\times S^1\times S^1$ for $\nu(K)\setminus K$ such that curves of constant $(r,\theta)$ are meridians, curves of constant $(r,\phi)$ are longitudes and $K$ is the limit as $r\to 0$ (and $K$ can be thought of as being parametrized by $\theta$). Let $Y_{K}=S^{3}\setminus K$ and $\bar Y_{K}=S^3\setminus\nu(K)$. Then $H_{1}(Y_{K};\Z)\cong\Z$. In what follows, identify $S^1$ with $\mathbb R/\mathbb Z$. 


\begin{definition} \label{def: admissible}
Consider a pair $(f,g)$, where $f\colon S^3\setminus K\to S^1$ is a smooth function and $g$ is a Riemannian metric on $S^3$. Let $\nabla_g f$ be the $g$-gradient vector field of $f$ on $S^3\setminus K$. Call the pair $(f,g)$ {\em admissible} if the following holds:
\begin{enumerate}
\item $f(r,\phi,\theta) = \phi - C r^2$, for some $C>0$, in $\nu(K)\setminus K$; \label{def: f near boundary}
\item $f_*: H_1(S^3\setminus K;\Z) \to H_1(S^1;\Z)$ is an isomorphism; \label{def:H_*}
\item $f$ is Morse function, in the sense that all its critical points are of Morse-type; \label{def:Morse}
\item $f$ has no critical points of indices 0 and 3; as a consequence, for every $\theta_0\in S^1$, the preimage $f^{-1}(\theta_0)$ is connected; \label{def: no 0 or 3}
\item $\nabla_g f$ is everywhere transverse to the boundary $\partial \bar Y_K = \partial \nu(K)$, and points out of $\bar Y_{K}$ along $\partial \bar Y_K$; \label{def: point out}
\item for the vector field $-\nabla_g f$, the stable and unstable manifolds for every pair of critical points of $f$ intersect transversely, and every gradient flow loop is non-degenerate. \label{def:hyperbolic}
\end{enumerate}
\end{definition}

\begin{remark}
If $g$ has the standard form $dr^2 + r^2 d\phi^2 + d\theta^2$ near $\partial \bar Y_K$, then \eqref{def: f near boundary} implies \eqref{def: point out} (note that $r$ decreases as one approaches $\partial \bar Y_K$ within $\bar Y_K$). Also, \eqref{def: f near boundary} implies that $f$ naturally defines a function on the zero-surgery of $S^3$ along $K$. 
\end{remark}


We next produce an admissible pair $(f,g)$. 
Conditions \eqref{def: f near boundary}, \eqref{def:H_*}, \eqref{def:Morse} and \eqref{def: point out} in Definition \ref{def: admissible} are easily arranged. By the Kupka--Smale theorem, see e.g.~\cite[Theorem 3.2.6]{MeloPalis}, condition \eqref{def:hyperbolic} is satisfied for generic $(f,g)$. The fact that stable and unstable manifolds intersect transversely implies in particular that gradient flow lines between critical points of consecutive indices are isolated. The fact that gradient flow loops are non-degenerate (they are called {\em hyperbolic} in \cite{MeloPalis}) means the following for every flow loop $\gamma$, if $T>0$ is the period of $\gamma$, then $1-(d\varphi^T)_x \colon T_x S^3 / T_x\gamma \to T_x S^3 / T_x\gamma$ has no eigenvalue of modulus 1 (and in particular is invertible). One also gets that, for every $T_0>0$, $\gamma$ has a neighborhood such that the only flow loops of period larger than $T_0$ that are contained in this neighborhood are multiples of $\gamma$ ~\cite[Lemma 3.2.2]{MeloPalis}. 
%

This leaves \eqref{def: no 0 or 3} which can be achieved by an argument used when canceling $0$ and $1$ handles. More precisley we have the following result. 
\begin{lemma} \label{l:critical points of deg 1 and 2}
Suppose that the pair $(f,g)$ satisfies \eqref{def: f near boundary}, \eqref{def:H_*}, \eqref{def:Morse}, \eqref{def: point out} and \ref{def:hyperbolic} in Definition \ref{def: admissible}. Then, there exists a function $\hat{f}\colon S^3\setminus K\to S^1$, obtained from $f$ by canceling $1$-handles (and $2$-handles) also satisfies \eqref{def: no 0 or 3}, and the pair $(\hat{f},g)$ is admissible. 
\end{lemma}

\begin{proof}
We modify $f$ in the complement of $\nu(K)$.
Let $\theta_0\in S^1$ be a regular value of $f$. Add two copies of $F = f^{-1}(\theta_0)$ to $(S^3\setminus K)\setminus F$ to produce a cobordism $W$ from $F$ to $F$. Then, $f$ induces a Morse function $f_W: W \to [\theta_0,\theta_0 + 1]$ with finitely many critical points. Reorder the critical points of $f_W$ so that all the critical values are distinct, and $\theta_0 + i/4 < f_W(x) < \theta_0 + (i+1)/4$ if $x$ is a critical point of Morse index $i\in \{0,1,2,3\}$. By condition \eqref{def: f near boundary}, this reordering can be done without changing $f_W$ in the region of $W$ that projects to $\nu(K)$

Let $x$ be a critical point of index 1 with descending flow lines connecting different connected components of $f_W^{-1}(\theta_0+1/4)$. Note that each such connected component corresponds to either a connected component of $f_W^{-1}(\theta_0)$ or a critical point of index 0 of $f_W$. We can modify $f_W$ by decreasing the value $f_W(x)$, and eventually either cancel $x$ with a critical point of index 0 or merge two connected components of $f_W^{-1}(\theta_0)$. Afterwards, reorder again the critical points as in the previous paragraph and repeat the same process with all critical points of index 1 with descending flow lines connecting different connected components of $f_W^{-1}(\theta_0+1/4)$. Then, we apply an analogous process to all the critical points of index 2 with ascending flow lines connecting different connected components of $f_W^{-1}(\theta_0+3/4)$. Since $S^3\setminus K$ is connected, property \eqref{def:H_*} implies that by the end of this process the function $f_W\colon W \to [\theta_0,\theta_0 + 1]$ gives a function $\hat{f}\colon S^3\setminus K\to S^1$ with property \eqref{def: no 0 or 3}. Observe again that all the function modifications can be done away from $\nu(K)$. 
\end{proof}

From this point on, let $(f,g)$ be an admissible pair (whose existence is guaranteed by Lemma \ref{l:critical points of deg 1 and 2}).
Let $\mathcal O$ denote the discrete set of flow loops of $\nabla_g f$ (for which $\dot\gamma = - (\nabla_g f)(\gamma)$). Consider the zeta function 
\begin{equation} \label{zeta_loops}
\zeta_{\rm loop}(\mu) = \exp \left(\sum_{\gamma \in \mathcal O} \frac{\sigma(\gamma)}{m(\gamma)} \mu^{d(\gamma)}\right)\in \Q[[\mu]].
\end{equation}
Here, if $\gamma$ has period $T$ and $x\in \gamma$, then $\sigma(\gamma)\in \{\pm 1\}$ is the sign of the determinant of the linear map $1-(d\varphi^T)_x \colon T_x S^3 / T_x\gamma \to T_x S^3 / T_x\gamma$, $m(\gamma)$ is the largest integer such that $\gamma$ factors through an $m$-fold cover $S^1 \to S^1$, and $[\gamma] = d(\gamma) p$ on $H_1(S^3\setminus K)$.

The function $f$ has an associated Novikov complex, that can be defined as follows. 
Consider the maximal Abelian cover $\widetilde Y_{K}$ of $\bar Y_{K}$, with projection map $\pi \colon \widetilde Y_{K} \to \bar Y_{K}$. Denote by $\mu$ the positive generator of the deck transformation group $\Z$.    
A lift $\tilde f\colon \widetilde Y_{K}\to\R$ of $f$ is a Morse function. Consider the $\Q$-vector space that is the (infinite) direct product of copies of $\Q$ indexed by the critical points of $\tilde f$. This vector space can be graded by the Morse index. Elements of the direct product can be thought of as infinite sums of critical points and we let
$C_*(\tilde f;\Q)$ 
be a subspace of half-infinite elements. 
More precisely,
$$
\sum_{\begin{smallmatrix}x\in \Crit(\tilde f),\\ a_x\in \Q\end{smallmatrix}} \!\! a_x \cdot x \ \in  \ C_*(\tilde f;\Q)
$$
if, for any $x$ such that $a_x\neq 0$, 
$$
\#\{ k>0  \, | \, a_{\mu^{-k} \cdot x} \neq 0 \} <\infty.
$$

The action of $\Z$ on $\widetilde Y_{K}$ by deck transformations endows $C_*(\tilde f;\Q)$ with a free action of the ring of formal Laurent series $\Q((\mu))$, consisting of sums with finitely many negative powers of $\mu$ and possibly infinitely many positive powers. Here, $\mu$ is to be thought of as the exponential of the positive generator of $\Z$ 
(or, equivalently, of the generator $p\in H_1(Y_K)$ above). Since $\Q$ is a field, $\Q((\mu))$ is also a field, isomorphic to the field of fractions of the formal power series ring $\Q[[\mu]]$. 
The complex $C_*(\tilde f;\Q)$ is a finite dimensional vector space over $\Q((\mu))$, and a basis can be obtained by fixing, for every $x\in \Crit(f)$, a lift $\tilde x\in \Crit(\tilde f)$ such that $\pi(\tilde x) = x$. Recall that we assume that all these critical points have index 1 or 2. 
The {\em Morse torsion} is defined as 
\begin{equation} \label{tau flow lines}
\tau_{\rm Morse}(\mu):=\det M(\mu),
\end{equation} 
where $M(\mu)$ is the matrix representing the Morse differential of $\tilde f$ in the chosen basis:
$$
\begin{CD}
\bigoplus_{x \in \Crit_2(f)} \Q((\mu)) \langle \tilde x \rangle @>{d_2}>> \bigoplus_{y \in \Crit_1(f)} \Q((\mu)) \langle \tilde y \rangle.
\end{CD}
$$
If $f$ has no critical points, then we say that $\tau_{\rm Morse}$ = 1. 

Results of \cite{MilnorDuality,FriedHomological,HutchingsLee} imply that the Alexander polynomial of $K$ can be written as the product  
\begin{equation}\label{alex}
\Delta_K(\mu)= (1-\mu) \, \zeta_{\rm loop}(\mu) \cdot \tau_{\rm Morse}(\mu).
\end{equation}
Note that a different choice of lifts $\tilde x$ would result in multiplying \eqref{tau flow lines} and \eqref{alex} by a power of $\mu$, which is compatible with the fact that the Alexander polynomial is defined up to multiplication by such powers.
In Section \ref{ssec:MorseAlex}, we give a Morse-theoretic derivation of formula \eqref{alex}.

\begin{remark} \label{rmk:fibered K}
If $K$ is a fibered knot, we can think of the Morse function $f$ as coming from the fibration $\Sigma \to S^3\setminus K \to S^1$, which endows the knot complement with an open book decomposition with monodromy given by a map $\phi\colon \Sigma \to \Sigma$. We can arrange for $\phi$ to be a symplectomorphism with respect to an area form on $\Sigma$. In that case, $\mathcal O$ is the collection of generators of the fixed point Floer homology of $\phi$ (see \cite{SpanoCategorification}). 

For $K$ fibered, the function $\zeta_{\rm loop}$ agrees with a certain Gromov--Taubes invariant of $S^1\times S_0(K)$, where $S_0(K)$ is 0-surgery on $K\subset S^3$. This invariant counts holomorphic tori in a specified homology class in $H_2(S^1\times S_0(K))$. For more details, see \cite[Lemma 2.11]{HutchingsNotes} and \cite{SpanoGromov}, where one can find explicit formulas for the contribution of a simple periodic orbit $\gamma$ and its covers to this Gromov--Taubes invariant, depending on whether $\gamma$ is elliptic or (positive or negative) hyperbolic. See Remark \ref{simple annuli} for a related point concerning holomorphic annuli. For a general knot $K$, see \cite{MarkTorsion} for an identification of the product $\zeta_{\rm loop} \cdot \tau_{\rm Morse}$ with a Seiberg--Witten invariant of $S_0(K)$.  
\end{remark}

\subsection{Morse theory calculation}\label{ssec:MorseAlex}
In this section, we present a Morse theory proof of Equation \eqref{alex} for admissible $(f,g)$. The formula itself holds under more general conditions, since its right side is invariant under a large class of (generic) deformations, see \cite{HutchingsReidemeister}. We point out that the rest of the paper is formally independent of this section, which is included to emphasize the direct connection between flow loops/lines and holomorphic annuli/disks.  
We use the notation as in Section \ref{ssec:dynAlex}. 

The idea of the argument is to break all flow loops and flow lines of $\nabla_g f$ in the previous section into sequences of short flow lines, by suitably perturbing the function $f\colon \bar Y_K \to S^1$. Lifting the perturbed function to $\widetilde Y_{K}\to\R$ (and adding further small perturbations, as explained below) gives a Morse chain model for the homology of the Abelian cover $\widetilde Y_K$, as a $\Q[\mu^{\pm}]$-module. 
We point out that, if we did not begin by perturbing $f$ to break flow loops, then the Morse complex of the lifted function would in general not compute the homology of the Abelian cover $\widetilde Y_K$, consider e.g., the identity map on $S^{1}$.
In this setting, the Alexander polynomial can be thought of as the degree 1 contribution to the Reidemeister torsion of $H_{\ast}(\widetilde Y_{K};\Q)$.

To be more specific, recall that $\widetilde Y_{K}$ is a maximal Abelian cover of $\bar Y_{K}$, which endows $H_{\ast}(\widetilde Y_{K};\Q)$ with the structure of a finitely generated $\Q[\mu^{\pm}]$-module. We will be interested in the degree $1$ part. Since $\Q$ is a field, $\Q[\mu^{\pm}]$ is a principal ideal domain. The module $H_{1}(\widetilde Y_{K};\Q)$ is torsion and the Alexander polynomial $\Delta_{K}(\mu)\in \Q[\mu]$ generates the corresponding order ideal in $\Q[\mu^{\pm}]$. In other words,
\begin{equation} \label{define Delta}
H_{1}(\widetilde Y_{K};\Q)\cong \Q[\mu^{\pm}]/(\Delta_{K}(\mu)).
\end{equation}


Recall that $f\colon \bar Y_{K}\to S^{1}$ is an $S^{1}$-valued Morse function 
with critical points of indices $1$ and $2$ only, and that $\widetilde f\colon \widetilde Y_{K}\to \R$ is a lift of $f$. Let $\theta_0$ be a regular value of $f$ and recall that $F=f^{-1}(\theta_0)$ is connected. 
Let $\phi\colon \bar Y_{K}\to S^{1}$ be a Bott--Morse perturbation of $f$ with a canceling pair of Bott manifolds (maxima 
and minima) on nearby preimages $F_\pm = f^{-1}(\theta_0\pm \epsilon)$ for small $\epsilon>0$. Let $\widetilde\phi\colon \widetilde Y_K\to \R$ denote the lift of this function. 

Think of $F$ as a two-dimensional handlebody and pick a Morse function $h_F \colon F\to \mathbb R$ such that $h_F$ has a single critical point $m$ in degree 0, and such that every index 1 critical point of $h_{F}$ has a pair of canceling gradient flow lines connecting it to $m$. Let $C_{\ast}(h_F)$ be the Morse complex of $h_F$ (with trivial differential by our assumptions), let $\widecheck{C}_{\ast}(F) = C_{\ast}(h_F) \otimes_{\Q} \Q[\mu^\pm]$, and let $\widehat{C}_{\ast}(F)$ be a copy of $\widecheck{C}_{\ast}(F)$ with the degrees of all the elements shifted up by 1. We can use the function $h_F$ to morsify $\widetilde\phi$, by which we mean perturbing $\widetilde\phi$ near the lifts of $F_\pm$ to get a Morse function $\widehat\phi\colon\widetilde Y_{K}\to\R$. The Morse complex of $\widehat \phi$ is:
\begin{equation} \label{eq: chain complex} 
C_{\ast}(\widetilde Y_{K};\mu) \ = \ \widecheck{C}_{\ast}(F)\oplus \widehat{C}_{\ast}(F)\oplus C_{1}(Y_{K})\oplus C_{2}(Y_{K}),
\end{equation}
where all summands on the right are finitely generated free $\Q[\mu^{\pm}]$-modules, and where $C_j(Y_K)$, $j=1, 2$ are generated by the critical points of $\widetilde f$ of degree $j$. See Figure \ref{graph_f_fig} for a schematic depiction. 
\begin{figure}
  \begin{center}
    \def\svgwidth{.9\textwidth}
    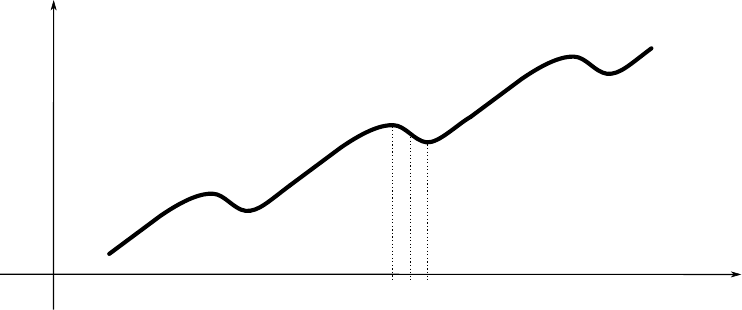
  \end{center}
  \caption{Schematic picture of the graph of $\widehat\phi$. The arrows indicate where the summands of $C_{\ast}(\widetilde Y_{K};\mu)$ are supported (when restricted to a fundamental domain for the action of $\Z$ on $\widetilde Y_K$ by deck transformations).}
  \label{graph_f_fig}
\end{figure} 
The differential in the complex $C_{\ast}(\widetilde Y_{K};\mu)$ is $\Q[\mu^{\pm}]$-linear and can be written as 
$$
\begin{CD} 
\widehat C_{1}(F) \oplus C_2(Y_K) @>\partial_2>> \widecheck C_{1}(F) \oplus \widehat C_{0}(F) \oplus C_1(Y_K) @>\partial_1>> \widecheck C_0(F)
\end{CD}
$$
where $\partial_1|_{\widecheck C_{1}(F) \oplus C_1(Y_K)} \equiv 0$ and $\partial_1|_{\widehat C_{0}(F)} \colon \widehat C_{0}(F) \to \widecheck C_{0}(F)$ is the map $1-\mu$. Hence, $H_{0}(\widetilde{Y}_{K};\mu) \cong \Q[\mu^{\pm}]/(1-\mu) \cong \Q$.
Furthermore, no element in the image of $\partial_2$ has a non-zero contribution in the summand $\widehat C_{0}(F)$, and $H_{1}(\widetilde{Y}_{K},\mu)$ is the cokernel of the $\mu$-module homomorphism
\begin{equation}\label{eq:presentation}
D = \partial_2 \colon \widehat{C}_{1}(F)\oplus C_{2}(Y_{K})\to \widecheck{C}_{1}(F)\oplus C_{1}(Y_{K}).
\end{equation}


The proof that the Morse homology of a smooth manifold is isomorphic to the homology of the manifold (see for instance \cite[Section 2.2]{HutchingsLee} for an isomorphism with singular homology or \cite[Section 4.9]{AudinDamian} for an isomorphism with cellular homology) can be adapted to show that the homology $H_*(\widetilde Y_K;\mu)$ is isomorphic to $H_{K}(\widetilde Y_{K};\Q)$ as $\Q[\mu^{\pm}]$-modules. By \eqref{define Delta}, we can use $H_1(\widetilde Y_K;\mu)$ to compute the Alexander polynomial.

\begin{lemma} \label{Delta det A}
	Pick bases for the left- and right-hand sides of \eqref{eq:presentation} as $\Q[\mu^{\pm}]$-modules, and let $A(\mu)$ be the corresponding matrix representation of $D$. Then $\Delta_{K}(\mu)=\det(A(\mu))$. 
\end{lemma}

\begin{proof}
The $\Q[\mu^{\pm}]$-module $H_1(\widetilde Y_K;\mu) \cong \coker (D)$ and the lemma follows from the classification theorem for finitely generated torsion modules over principal ideal domains as follows. The matrix $A(\mu)$ is a matrix of relations for the $\Q[\mu^{\pm}]$-module $H_1(\widetilde Y_K;\mu)$. The diagonal entries of its Smith normal form are the invariant factors. Their product is on the one hand equal to the determinant of $A(\mu)$, and on the other, by \eqref{define Delta}, equal to $\Delta_{K}(\mu)$. (See e.g., \cite[Theorem 7.7 in Section III.7]{LangAlgebra} for more details.) 
\end{proof}

Figure \ref{graph_f_fig} gives a schematic description of the components of $C_*(\widetilde Y_K;\mu)$ when restricted to a fundamental domain for the action of $\Z$ on $\widetilde Y_K$ by deck transformations. Pick bases as in Lemma \ref{Delta det A}, corresponding to the critical points of the Morse function $\widehat \phi$ that are contained in that fundamental domain. The differential on the Morse complex is then represented as a $\Q[\mu^{\pm}]$-module map by a matrix of the form
\[ 
A(\mu)=\left( 
\begin{matrix}
1-\mu\psi_{F} & -\mu\psi_{c} \\
\eta & d_0
\end{matrix}
\right).
\]
Here, $d_{0}$ denotes the count of flow lines from $C_{2}({Y}_{K})$ to $C_{1}({Y}_{K})$, $\psi_{F}$ denotes the count from $\widehat{C}_{1}(F)$ to $\mu\widecheck{C}_{1}(F)$, $\eta$ denotes the count from $\widehat{C}_{1}(F)$ to $C_{1}(Y_{K})$ and $\psi_{c}$ the denotes the count from $C_{2}(Y_{K})$ to $\mu\widecheck{C}_{1}(F)$. The contributions with $\mu$ correspond to flow lines that leave the fundamental domain that had been specified. By Lemma \ref{Delta det A}, we have 
\begin{equation} \label{eq:alex D}
\Delta_{K}(\mu)=\det A(\mu).
\end{equation}
%

In order to get a geometric interpretation of $\det A(\mu)$, we will express it as a product of two determinants. At this point, we will use the inclusion $\Q[\mu^{\pm}] \subset \Q((\mu))$ into the field of formal Laurent series (with powers of $\mu$ bounded below). 

\begin{lemma}\label{l: determinant filter}
The following holds in $\Q((\mu))$:
\begin{equation} \label{eq:determinants}
\det A(\mu) = \det(1-\mu\psi_{F})\det\left(d_0+\mu\eta \left(\sum_{n\ge 0}\mu^{n}\psi_{F}^{n}\right)\psi_{c}\right).
\end{equation}
\end{lemma}
\begin{proof}
In each successive step below, subtract from the second line the product of the bottom left entry by the first line: 
\begin{alignat*}{2}\notag
\det\left( 
\begin{matrix}
1-\mu\psi_{F} & -\mu\psi_{c} \\
\eta & d_0
\end{matrix}
\right)
\ &= \
\det\left(
\begin{matrix}
1-\mu\psi_{F} & -\mu\psi_{c} \\
\mu\eta\psi_{F} & d_0+\mu\eta\psi_{c}
\end{matrix}
\right)\\\notag
\ &= \
\det\left(
\begin{matrix}
1-\mu\psi_{F} & -\mu\psi_{c} \\
\mu^{2}\eta\psi_{F}^{2} & d_0+\mu\eta\psi_{c}+
\mu^{2}\eta\psi_{F}\psi_{c}
\end{matrix}
\right) \\\notag
 &\;\;\vdots  \\
\ &= \
\det\left(
\begin{matrix}
1-\mu\psi_{F} & -\mu\psi_{c} \\
\mu^N\eta\psi_{F}^N & d_0+\mu\eta
\left(\sum_{n= 0}^N\mu^{n}\psi_{F}^{n}\right)\psi_{c}
\end{matrix}
\right)
\end{alignat*}
for every $N>0$. Using the fact that $\det\left(
\begin{matrix}
A & B \\
C & D
\end{matrix}
\right) = \det(A) \det(D - C A^{-1}B)$ for block matrices, the determinant of the last matrix above is 
\begin{alignat*}{2}\notag
\det(1-\mu\psi_{F}) &\det\left(d_0+\mu\eta
\left(\sum_{n= 0}^N\mu^{n}\psi_{F}^{n}\right)\psi_{c} + \mu^N\eta\psi_{F}^N (1-\mu\psi_{F})^{-1} \mu\psi_{c} \right)
 \\
&= \
\det(1-\mu\psi_{F}) \det\left(d_0+\mu\eta
\left(\sum_{n= 0}^N\mu^{n}\psi_{F}^{n}\right)\psi_{c} + \sum_{k\geq 0}\mu^{N+k+1}\eta\psi_{F}^{N+k}\psi_{c} \right).
\end{alignat*}
In this last expression, consider the second factor in the product of determinants. It is the determinant of a sum of three terms. The first two terms give the small powers of $\mu$ in the determinant, whereas the high powers of $\mu$ come from the last summand. Therefore, by taking $N$ arbitrarily large we get the formula \eqref{eq:determinants}. 
\end{proof}

We next show that Lemmas \ref{Delta det A} and \ref{l: determinant filter} imply formula \eqref{alex}: 

\begin{lemma}
The factors in \eqref{eq:determinants} satisfy the following: 	
\begin{align*}
\det(1-\mu\psi_{F}) \ &= \ (1-\mu)\zeta_{\rm loop}(\mu),\\	
\det\left(d_0+\mu\eta \left(\sum_{n\ge 0}\mu^{n}\psi_{F}^{n}\right)\psi_{c}\right) \ &= \ \tau_{\rm Morse}(\mu),
\end{align*}	
where $\zeta_{\rm loop}(\mu)$ and $\tau_{\rm Morse}(\mu)$ are as in \eqref{zeta_loops} and \eqref{tau flow lines}, respectively.  
\end{lemma}

\begin{proof}
	We will use a standard Morse gluing result \cite[Proposition 3.13.(a)]{HutchingsReidemeister} which shows how the creation/cancellation of a pair of critical points affects spaces of flow lines and flow loops (similar to the First Cancellation Theorem, see \cite[Theorem 5.4]{MilnorHcobordism}, especially Figure 5.2).

	Consider the first determinant: 
	\begin{align} \label{det tr}
	\det(1-\mu \psi_{F}) &=\det\exp(\log(1-\mu \psi_{F}))=\exp(\tr \log(1-\mu \psi_{F}))\\\notag
	&=\exp\left(- \sum_{n\geq 1} \frac{1}{n}\mu^{n}\tr(\psi_{F}^{n})\right).
	\end{align}
We relate \eqref{det tr} with the zeta function $\zeta_{\rm loop}(\mu)$ in \eqref{zeta_loops}, counting flow loops for an admissible pair $(f,g)$. The function $\widehat \phi\colon \widetilde{Y}_K \to \mathbb R$ is a perturbation of the lift $\widetilde f \colon \widetilde{Y}_K \to \mathbb R$, where the  perturbation creates pairs of canceling critical points, with corresponding  new summands in the Morse complex referred to as $\widehat C_*(F)$ and $\widecheck C_*(F)$ in Figure \ref{graph_f_fig}. 
Denote the parameter of this perturbation by $\epsilon>0$, so that $\widehat \phi \to \widetilde f$ when $\epsilon\to 0$. 

Fix an integer $N\ge 1$, then for sufficiently small $\epsilon>0$ and every integer $1\le d \leq N$, we get the following correspondence between flow loops $\gamma$ of $f$ of degree $d(\gamma) = d$ and flow lines of $\widehat \phi$. A flow loop $\gamma$ corresponds either to a diagonal entry in $\psi_F^d$, thought of as $d$ flow lines contributing to $\psi_F$ together with $d$ short canceling flow lines from $\widehat C_*(F)$ to $\widecheck C_*(F)$, or to an iterate of the $d=1$ flow loop of $f$ that corresponds to the unique flow line of $\widehat \phi$ from $\widehat C_0(F)$ to $\mu\widecheck C_0(F)$. The total contribution of the latter (and its iterates) to $\zeta_{\rm loop}$ is
\[
\exp\left(\sum_{n\geq 1} \frac{\mu^n}{n}\right) = \frac{1}{1-\mu}.
\]
The first equation follows.


Consider the second determinant:
	\begin{equation} \label{det poof}
	\det\left(d_0 + \mu\eta\left(1+\mu \psi_{F} +\mu^{2}\psi_{F}^{2}+\dots\right)\psi_{c}\right).
	\end{equation}
We relate \eqref{det poof} with $\tau_{\rm Morse}(\mu)$, the determinant of the Morse--Novikov complex of $f$ in \eqref{tau flow lines}. As above, we need to compare rigid flow lines of $\widetilde f$ and $\widehat \phi$. 

Consider $M(\mu)$ in \eqref{tau flow lines}. The $d_0$ term in \eqref{det poof} accounts for the `short' flow lines in the Novikov differential, that do not pick up any power of $\mu$. The term $\eta \psi_{F}^{n}\psi_{c}$ can be interpreted as the flow lines obtained by starting with the unstable manifold of an index 2 critical point ($\psi_{c}$), canceling $n$ intermediate Bott maxima and minima ($\psi_{F}^{n}$), and then intersecting with the stable manifold of an index 1 critical point ($\eta$). This gives the `long' flow lines in the Novikov differential. The lemma follows.
%
\end{proof}

\section{Flow loops, flow lines, and holomorphic curves} \label{sec:flow lines and holo curves}
In this section we describe in more detail how a knot $K$ has an associated conormal Lagrangian $L_K\subset T^{\ast}\R^{3}$ diffeomorphic to $S^{1}\times\R^{2}$ and a knot complement Lagrangian $M_K\subset T^{\ast}\R^{3}$ diffeomorphic to $\R^{3}\setminus K$, both with natural non-exact deformations. We study holomorphic disks in $T^{\ast}\R^{3}$ with boundary on $L_{K}$ or $M_{K}$, and annuli and strips with one boundary component on $\R^3$ and the other on $L_{K}$ or $M_{K}$. For $M_K$, we establish a 1-1 correspondence between holomorphic annuli and disks, and Morse flow loops and flow lines in $\R^3 \setminus K$, respectively. The proof uses \cite{EkholmFlowTrees}, but avoids many of the complications there since the Lagrangians here admit fronts without singularities. 

\subsection{Models of $L_{K}$ and $M_{K}$} \label{models}
Let $K\subset \R^{3}$ be an oriented knot. We will consider exact and non-exact Lagrangians associated to $K$ in $T^*\R^3$, as in \cite[Section 6]{AENV}. Let $L_{K}\subset T^{\ast}\R^{3}$ denote its Lagrangian conormal:
\[ 
L_{K}=\{(q,p)\in T^{\ast}\R^{3}\colon q\in K,\; p|_{TK}=0 \}.
\]
Consider a small tubular neighborhood $\nu(K) \subset \R^3$ of $K$, with projection $\pi\colon \nu(K)\to K$. Fix a diffeomorphism $\alpha\colon K\to S^1$ (where we think of $S^1\cong \mathbb R/\mathbb Z$ as having length 1) and let $\eta_L\coloneq \pi^*d\alpha$. 
Given $\delta > 0$, write $\Ld$ for the Lagrangian that is obtained by applying to $L_K$ the non-exact symplectomorphism of $T^*(\nu(K))$ that is given by
\[ 
(q,p)\mapsto (q,p +\delta \, \eta_L(q)).
\]  
Note that $L_{K}\cap\R^{3}=K$ and $L_{K}^{\delta}\cap\R^{3}=\varnothing$. We will sometimes write $L_K^0$ instead of $L_K$.

Let $M_{K}$ denote the knot complement Lagrangian that is obtained by applying Lagrange surgery along the intersection $L_{K}\cap \R^{3}=K$, which cuts out $\epsilon$-neighborhoods of the knot in the two Lagrangians and glues them together.\footnote{There are two ways of performing Lagrange surgery, which do not give locally Hamiltonian isotopic Lagrangians. Nevertheless, the two versions of $M_K$ induce the same augmentations (which could be seen from the relation between $\epsilon_{L_K}$ and $\epsilon_{M_K}$ mentioned in the proof of Lemma \ref{epsilon LK v epsilon MK} below), so we can use either one for the purposes of this paper.} 
Then $M_{K}$ is diffeomorphic to $\R^{3}\setminus K$.  See \cite{MakWu} for a description of Lagrange surgery on a clean intersection, and for the fact that if the intersection locus is connected, then the result of surgery on two exact Lagrangians can also be made exact. 


Let $(f,g)$ be an admissible pair as in Definition \ref{def: admissible}, where $g$ is a Riemannian metric on $S^3$ and $f\colon S^3\setminus K \to S^1$ is an $S^1$-valued Morse function given by 
\begin{equation} \label{eq:f}
f(r,\phi,\theta) = \phi - C r^2
\end{equation}
in $\nu(K)\setminus K$, for some neighborhood $\nu(K)$ of $K$. Let $M_{K}$ be the surgery exact Lagrangian, and assume that the surgery region is contained in $T^*(\nu(K))$. This implies that $f$ is given by \eqref{eq:f} at the points in $M_K\cap \R^3$ that are contained in the surgery region. Observe that $M_K\setminus \R^3$ can be identified with $L_K\setminus K$ and parametrized by $S^1\times S^1 \times \R_{>0}$, with one of the $S^1$-factors corresponding to the variable $\phi$ in \eqref{eq:f}. One can then extend the restriction $f|_{M_K\cap \R^3}$ to a function $h\colon M_K\to S^1$, such that $\nabla h$ is bounded (with respect to the metric induced by $g$ on $T^*\R^3$) and $h$ has no critical points in $M_K\setminus \R^3$. Let $\nu(M_K)\subset T^*\R^3$ be a tubular neighborhood of $M_K$ with projection map $\pi\colon \nu(M_K)\to M_K$, and define $H\colon \nu(M_K)\to S^1$ as $H = h\circ \pi$. Let $X_H$ be the symplectic vector field associated to $H$ and define $\Md$ as the time-1 flow of $\delta X_H$, for $\delta\geq 0$ small enough (so that $\Md \subset \nu(M_K)$). 


Given the properties of $f$, we can assume that for $\delta>0$ the intersection of $\Md$ with the zero-section is transverse, and corresponds to the critical points of $f$ (all of which have Morse indices 1 or 2). We sometimes write $M_K^0$ instead of $M_K$.

\begin{remark}
The non-exact Lagrangian shifts $\Ld$ and $\Md$ do not have cylindrical ends, but they are asymptotic to a cylinder on $\Lambda_K$. For a Lagrangian filling with cylindrical ends, $\R$-translation invariant moduli spaces of holomorphic disks at infinity (in the symplectization of the contact boundary) form boundary components of corresponding moduli spaces in the filling. In particular such families can be continued as solution spaces of disks in the filling. To see why the analogous fact for asymptotically cylindrical fillings is also true, note that the non-exact perturbation is of fixed size whereas the symplectic form in the symplectization grows exponentially. Therefore, in sufficiently large disk bundles there is an arbitrarily small change of the almost complex structure (for instance, given by conjugation by a map taking a Lagrangian to its $\delta$-shift) making the original $\R$-invariant disks holomorphic with the non-exact boundary condition.
\end{remark}

\subsection{Holomorphic disks}
We consider holomorphic disks in $T^{\ast}\R^{3}$ with boundary on the Lagrangians $L_K^\delta$ and $M_K^\delta$ with  punctures asymptotic to Reeb chords. We note that punctured holomorphic curves with only one positive puncture are injective near that puncture and therefore regular for generic almost complex structures. Also, recall that the energy of a holomorphic curve with compact domain $\Sigma$ and Lagrangian boundary $u\colon (\Sigma,\partial \Sigma)\to (T^*\R^3,L)$ is $E(u) = \int_{\Sigma} u^*\omega$. If $\partial\Sigma$ has punctures, then the definition of energy of $u$ must be adjusted and one uses the Hofer energy in the non-compact ends of the target, as in \cite{BEHWZ}.

Our next result shows that for sufficiently small $\delta>0$ the corresponding non-exact shifts $\Ld$ and $\Md$ of $L_{K}$ and $M_{K}$ do not affect moduli spaces of once punctured holomorphic disks. Hence small non-exact shifts can also be used to compute the augmentations associated to $L_K$ and $M_K$. 

 \begin{lemma} \label{LK MK augment}	
Let $J$ be an almost complex structure on $T^{\ast}\R^{3}$ that is standard near the zero section and let $L^{\delta}=\Ld$ or $L^{\delta}=\Md$. For any $E>0$, there is $\delta_0 > 0$ and a neighborhood $\mathcal{J}$ of $J$ such that for any $0<\delta\leq \delta_0$ the following holds for complex structures in $\mathcal{J}$.
\begin{itemize}
\item[$(i)$] There are no non-constant closed holomorphic disks (without punctures) of energy less than or equal to $E$, with boundary on $L^{\delta}$.
\item[$(ii)$] Assume that all moduli spaces of disks with one positive puncture on $L^{0}$ are transversely cut out. If $a$ is a Reeb chord of $\Lambda_{K}$ of degree 0, then there is a natural 1-1 correspondence 
between the holomorphic disks of energy less than or equal to $E$ with one puncture asymptotic to $a$ and with boundary on $L^{\delta}$, and the analogous disks with boundary on $L^{0}$. More precisely, there exist disjoint connected $C^{0}$-neighborhoods of those disks with boundary on $L^{0}$, such that each connected neighborhood contains exactly one disk with boundary on $L^{\delta}$. 
\end{itemize} 
\end{lemma}

\begin{proof}
Statement $(i)$ is a consequence of SFT compactness. If there is a non-constant holomorphic disk $u^\delta\colon (D^2,\partial D^2) \to (T^*\R^3,L^\delta)$, then $[u^\delta|_{\partial D^2}] \in H_1(L^\delta;\Z)$ is non-trivial, since otherwise $u^\delta$ would have vanishing area, by exactness of $\omega$ in $T^*\R^3$. Also, by monotonicity, for a ball $B_{R}\subset \R^{3}$ of sufficiently large radius $R$, any holomorphic curve with boundary on $L_{K}$ that maps some point outside $T^{\ast}_{B_{R}} \R^{3}$ has area $>E$.

Now, if $(i)$ does not hold then there is a sequence of such non-constant disks $u^{\delta}$, $\delta\to 0$. 
An SFT convergent subsequence must be bounded in $T^*\R^3$, since otherwise the SFT building in the limit would have a component at infinity with no positive puncture, in violation of the maximum principle. Thus, the curves in the SFT limit have no punctures and the limit is the usual Gromov limit. This Gromov limit would contain a non-constant holomorphic disk with boundary on $L^{0}$, which contradicts exactness of $L^{0}$. The existence of such a disk follows from the fact that the homotopy class of the boundary in a sequence of holomorphic disks with Lagrangian boundary is preserved under Gromov limit, see \cite[Proposition 3.2(ii)]{FrauenfelderGromovCpctness}. 

Statement $(ii)$ is a consequence of compactness and transversality. Denote by $\M^{a,E}(L^{\delta})$ the moduli space of punctured disks for $L^{\delta}$ in the statement. 
Consider the space 
$$
\M^{a,E}({[0,\delta_0]}) \coloneq \bigcup_{\delta\in[0,\delta_0]} \M^{a,E}(L^{\delta}).
$$
The assumed transversality for disks with boundary on $L^{0}$ implies that, for $\delta_{0}>0$ sufficiently small, each disk in $\M^{a,E}(L^{0})$ gives a boundary point of a single connected component of the space $\M^{a,E}({[0,\delta_0]})$. Combining transversality with Gromov compactness, we can assume that this component projects as a diffeomorphism to $[0,\delta_{0}]$ and that the solutions in $(0,\delta_{0}]$ are arbitrarily close to that at $0$. Gromov compactness also implies that every component of $\M^{a,E}({[0,\delta_0]})$ appears in this way, for some disk in $\M^{a,E}(L^{0})$. The result follows. 
\end{proof}

\begin{remark}
We will see in Remark \ref{remark: braided model DGA} that, in the flow tree model for knot contact homology \cite{KCH} we could arrange for no Reeb chord of degree 0 to be the asymptotic limit of a punctured disk with boundary on ${L_K}$. In that case, Gromov compactness would suffice to imply that below any energy threshold, and for $\delta$ sufficiently small, $\Ld$ would bound no holomorphic disks with one boundary puncture asymptotic to a Reeb chord of degree 0. Put differently, the 1-1 correspondence in part $(ii)$ of Lemma \ref{LK MK augment} would be between empty sets.
\end{remark}

\subsection{Holomorphic annuli and strips} \label{ssec:annuli}
We consider next holomorphic annuli in $T^*\R^3$ between the zero section and the Lagrangians $\Ld$ and $\Md$, respectively. 

Given a real number $R>0$, call $A_R:=[0,R] \times S^1$ the {\em annulus of modulus $R$} (or of {\em conformal ratio $R$}), where we identify $S^1$ with $\R/\Z$. Given a Lagrangian $L$ in $T^*\R^3$, define
\begin{align*}
\M_{\rm an}(L) &= \M_{\rm an}(\R^3,L)\\ 
&= \bigcup_{R\in(0,\infty)} \left\{ u\colon A_R \to T^*\R^3\colon 
    \overline\partial_J u = 0, u(\{0\}\times S^{1})\subset \R^3, \, u(\{R\}\times S^{1}) \subset L
  \right\}
\end{align*}
to be the space of $J$-holomorphic annuli of arbitrary modulus, stretching from $\R^3$ to $L$. Given $E>0$, denote by $\M_{\rm an}^E(L)$ the subset of annuli of energy (that is, $\omega$-area) at most $E$. 

The automorphism group of the Riemann surface $A_R$ consists of domain-rotations in the $S^1$-factor, and can be identified with $S^1$. Hence, $S^1$ acts on $\M_{\rm an}(L)$ by pre-composition with domain automorphisms. Let $u\in \M_{\rm an}(L)$ be an annulus with multiplicity $m(u)$ at least $2$. This means that $u$ factors as $v \circ \varphi$, where $\varphi\coloneq A_{R/{m(u)}} \to A_R$ given by $\varphi(s,t) = (m(u)s,m(u)t)$ is an $m(u)$-to-$1$ cover and $v$ is an annulus that does not admit such a factorization. Then, the isotropy group of $u$ under the $S^1$-action on $\M_{\rm an}(L)$ is $\Z/(m(u) \Z)$.

Denote the quotient by 
$$
\widetilde \M_{\rm an}(L)\coloneq \M_{\rm an}(L)/S^1.
$$
Since isotropy groups of multiply covered annuli are non-trivial, this quotient space is an orbifold rather than a manifold.

Let us consider now the case where $L$ is of the form $\Ld$. Since the Maslov class of $\Ld$ is zero, the expected dimension of $\M_{\rm an}(\Ld)$, i.e.~the space of parametrized annuli, is 1. 
Since $L_K$ is homotopy equivalent to $K$, we can write 
the homology class 
$\left[u|_{\{R\} \times S^1}\right]=d(u) x$ on $H_1(L_K^\delta;\Z)$, for some $d(u) \in \Z$, where $x$ is the fundamental class of the oriented knot $K$. 
The energy of a holomorphic annulus $u\in \M_{\rm an}(\Ld)$ of modulus $R$ is then
$$
E(u) = \int_{A_R} u^*\omega = \delta \, d(u).
$$ 
Since the energy of a holomorphic curves is non-negative and $\R^{3}\cap L_{K}^{\delta}=\varnothing$ it follows that $d(u)> 0$. Given an integer $d>0$, let
$$\M_{\rm an}^{\delta d}(\Ld) \subset \M_{\rm an}(\Ld)$$
denote the space of annuli of energy $\le \delta d$. Then $\M_{\rm an}^{\delta d}(\Ld)$ consists of the annuli $u$ with $d(u)\leq d$.

It will be useful to assume that $K$ is a real analytic submanifold of $\R^3$. By \cite[Lemma 8.6]{CELN}, there is a compatible almost complex structure $J$ on $T^*\R^3$ for which $K$ has a neighborhood $\nu(K)$ that admits a holomorphic parametrization $\varphi\colon S^1\times (-1,1) \times B_1^4 \to \nu(K)$, where $B_r^4\subset \C^2$ denotes the ball of radius $r>0$ and 
\begin{itemize}
\item $K = \varphi(S^1\times \{0\}\times \{0\})$,
\item $\R^3\cap \nu(K) = \varphi(S^1\times \{ 0 \} \times B_1^2)$ and 
\item $L_K\cap \nu(K) = \varphi(S^1\times \{ 0 \} \times i B_1^2)$,
\end{itemize}
where $B_r^2\subset \R^2$ is the ball of radius $r$. Identifying the factor $S^1\times (-1,1)$ with a neighborhood of the zero section in $T^*S^1$, we can further assume that 
\begin{itemize}
 \item $\Ld\cap \nu(K) = \varphi(S^1\times \{ \delta \} \times i B_1^2)$ for $\delta$ sufficiently small.
\end{itemize}

\begin{lemma} \label{annuli L}
 For any $d >0$, there is $\delta_0 > 0$ such that for any $0<\delta\leq \delta_0$ there is a natural 1-1 correspondence between covers of $K$ of order up to $d$ and elements in $\M_{\rm an}^{\delta d}(\Ld)$. Moreover, any annulus in $\M_{\rm an}^{\delta d}(\Ld)$ is regular for $J$ as above. 
\end{lemma}
\begin{proof}
We begin by showing that, for fixed $d>0$ and for $\delta>0$ sufficiently small, the images of all $u\in \M_{\rm an}^{\delta d}(\Ld)$ are contained in the fixed neighborhood $\nu(K)$. 
If we assume that $u(p)\notin \nu(K)$, then there is a ball of radius $\frac12$ around $u(p)$ that intersects at most one of the Lagrangians $\R^{3}$ and $\Ld$. It follows by the monotonicity lemma with Lagrangian boundary conditions \cite[Lemma 3.4]{ECL} that the area of $u$ is bounded below by some positive constant. Since the area of such $u$ goes to 0 as $\delta\to 0$, this then implies that $u$ lies inside $\nu(K)$ for $\delta$ sufficiently small. 

For such $u$, we consider the projection to $B_{1}^4$ as in the definition of $\varphi$ above, where $L\cup\R^{3}$ maps to $\R^{2}\cup i\R^{2}$. By the exactness of the projection to $B_1^4$ of the restriction of $\R^3$ and $\Ld$ to $\nu(K)$, the projection of $u$ to $B_{1}^4$ is constant. The lemma is now reduced to studying holomorphic annuli on the target $S^1\times (-1,1)$, with boundary components on $S^1\times \{0\}$ and $S^1\times \{\delta\}$, which are just unbranched covers of the annulus between these two circles. 

The linearized Cauchy--Riemann operator along any such annulus is the standard $\bar\partial$-operator on $\C\times \C^{2}$ with boundary condition  $\R\times\R^{2}$ along one boundary and $\R\times i\R^{2}$ along the other. It then follows by elementary complex analysis that the holomorphic annuli are regular. 
\end{proof}

Lemma \ref{annuli L} implies that, for every fixed integer $d>0$ and $\delta$ sufficiently small, the connected components of $\M_{\rm an}^{\delta d}(\Ld)$ are finitely many circles, each consisting of domain rotations of some annulus $u$ of a certain multiplicity $m(u)$. The moduli space $\widetilde \M_{\rm an}^{\delta d}(\Ld)$ is therefore a finite collection of orbifold points, with isotropies given by $\Z/(m(u) \Z)$. 


Consider next holomorphic annuli between $\R^3$ and $\Md$. Just like $\Ld$, the Lagrangian 
$\Md$ has Maslov class 0 and hence 
the expected dimension of the space $\M_{\rm an}(\Md)$ of parameterized annuli is 1.
Recall that we used an admissible pair $(f,g)$ in the definition of $\Md$. Given a flow loop or a flow line $\gamma \subset \R^3$ of $\nabla_g f$, we define its action as
\[ 
\mathfrak{a}(\gamma)=\int_{\gamma} df.
\] 



We will next adapt the metric and the Lagrangian $M_{K}$ to the rigid flow loops and lines as in \cite[Section 4.3]{EkholmFlowTrees}. The modifications here correspond to the most basic cases in that paper, since our flow objects here are simply flow lines (there is no branching as in flow trees). We deform the Lagrangian and the metric so that they are of the form described in \cite[Section 4.3.6--7]{EkholmFlowTrees} near a flow line. In addition, the metric is flat and the gradient is constant along a flow loop. The construction here depends on the action level $\mathfrak{a}_{0}$. As we increase the action, new flow objects need to be taken into account and we shrink the neighborhoods where the Lagrangian and metric in previous steps were normalized. We do that in such a way that when passing from action level $\mathfrak{a}_{0}$ to $\mathfrak{a}_{1}$, $\mathfrak{a}_{1}>\mathfrak{a}_{0}$, we keep the metric and Lagrangian unchanged in some subset of the set where it was previously normalized. Note also that we can keep the Lagrangian fixed near the knot and the function of the deformed Lagrangian together with the metric is still an admissible in the sense of Definition \ref{def: admissible}. 

Fix the natural almost complex structure $J$ on $T^{\ast} \R^{3}$ determined by the metric $g$, see \cite[Section 4.4]{EkholmFlowTrees}. 


\begin{lemma} \label{annuli loops}
For any $\mathfrak{a}>0$ there exists $\delta_{0}>0$ such that, for any $\delta<\delta_{0}$, there is a natural 1-1 correspondence between rigid flow loops and flow lines of $\nabla_g f$ of action $\leq \mathfrak{a}$, and $J$-holomorphic annuli and $J$-holomorphic strips, respectively, between $\R^{3}$ and $\Md$ of action $\leq \delta\mathfrak{a}$. More precisely, there exists a neighborhood of any simple rigid flow loop that contains exactly one simple rigid holomorphic annulus, and there exists a neighborhood of any rigid flow line that contains exactly one holomorphic strip. Moreover, there is a natural 1-1 correspondence between multiple flow loops over simple loops and multiply covered annuli over simple annuli. These $J$-holomorphic annuli and strips are regular.    
\end{lemma}

\begin{proof}
The proof is a simpler version of the corresponding result in \cite{EkholmFlowTrees}. The case of strips follows immediately from \cite[Theorem 1.1]{EkholmFlowTrees}. 

Consider the case of annuli. We first show that rigid annuli converge to rigid flow loops as $\delta\to 0$. The first step is to bound the length of the boundary of the annuli in terms of $\delta$ times the action as in \cite[Lemma 5.2 and 5.4]{EkholmFlowTrees}. Then one uses subharmonicity of the square of the momentum coordinate to confine the holomorphic curves to an $\mathcal{O}(\delta)$-neighborhood of the zero section as in \cite[Lemma 5.5]{EkholmFlowTrees}. With this established, basic elliptic estimates give control of the $C^{1}$-norm of the holomorphic maps as in \cite[Lemma 5.6]{EkholmFlowTrees}. (Compared to the flow tree case, only the simplest case is required here.) At this stage \cite[Lemmas 5.13 and 5.18]{EkholmFlowTrees} gives the desired convergence to flow loops.

We next show that there is a unique holomorphic annulus near each flow loop. Our choice of flat metric near the flow loop and the deformation of the Lagrangian gives an obvious holomorphic annulus over the flow loop. An explicit calculation shows that the linearized operator is uniformly surjective, and it follows from this transversality that the annulus is unique in a small neighborhood. 
\end{proof}

%
%

\subsection{Orientations of moduli spaces}
As usual for holomorphic curves with boundary, we orient moduli spaces via the Fukaya orientation \cite{FOOO}. Fukaya's original construction concerns moduli spaces of holomorphic disks, and needs a coherent orientation of the index bundle over the space of maps from a disk into a symplectic manifold with Lagrangian boundary condition. Such orientations exist provided the Lagrangian is relatively spin. In our case, the symplectic manifold is spin and therefore the relative spin condition reduces to the Lagrangian being spin. A spin structure on the Lagrangian allows us to trivialize the tangent bundle over the 2-skeleton, which in turn gives a trivialization (up to homotopy) over the boundary of any map of the disk. The index bundle over the space of maps is then oriented by deforming any bundle to a connect sum of a trivialized bundle over a disk and a bundle over $\C P^{1}$ attached at the midpoint of the disk. 
As the bundle over the projective line is canonically oriented, this together with the trivialization over the disk induces the Fukaya orientation.

\subsubsection{Disks with corners}\label{ssec:orientdisks}
Below we will discuss also orientations of holomorphic disks with punctures asymptotic to Reeb chords or Lagranagian intersection points. Here we import the Fukaya orientation by choosing capping disks at each corner or Reeb chord, i.e., a disk with one puncture and a linear Lagrangian boundary condition that connects the transverse tangent spaces to the Legendrian or Lagrangian at the Reeb chord or double point and an orientation of the determinant of the $\bar\partial$-operator with the corresponding boundary condition, see \cite[Section 3.3]{EkholmEtnyreSullivanOrientations} for details. The key property of these capping disks is that if we glue the operators of two capping disks at a Reeb chord or double point then we get a linearized boundary condition for the $\bar\partial$-operator on the closed disk that has the Fukaya orientation. 

\subsubsection{Annuli}\label{ssec:orientannuli}
In a symplectic manifold of dimension $2n$ with $n$ odd, we can orient the index bundles over more complicated surfaces using a similar construction. Start from a closed surface (in the case of disks discussed above, this closed surface was $\C P^{1}$) and attach disks punctured at their centers at interior punctures. The orientation of the resulting index bundle now a priori depends on an ordering of the boundary components. However, since the index of the disks with constant boundary condition is $n$ and the automorphism group of the disk punctured at the origin is a circle, the total index contribution of each disk is $n-1$. Since this is even, the orientation of the index bundle is independent of ordering. We will use this Fukaya orientation to orient the space of holomorphic annuli. For the Lagrangians $\Ld$, we define the annulus counting functions
$$
\An_{L_K}^{\delta,d}(\lambda) := \sum_{[u] \in \widetilde \M_{\rm an}^{\delta d}(\Ld)} \frac{\sigma(u)}{m(u)} \lambda^{d(u)} \in \C[\lambda],
$$
where $\sigma(u)\in \{\pm 1\}$ is determined by the orientation scheme that was just described. 
We also define 
$$
\An_{L_K}(\lambda) := \lim_{d\to \infty} \lim_{\delta \to 0} \An_{L_K}^{\delta,d}(\lambda) \in \C[[\lambda]]. 
$$
Analogously, for $\Md$ we define generating functions $\An_{M_K}^{\delta,d}(\mu)$ and $\An_{M_K}(\mu)$. The definition also includes signs $\sigma(u)\in \{\pm 1\}$. 
By analogy with \eqref{zeta_loops}, we define $\zeta_{\rm{an}}(\mu) = \exp\left(\An_{M_K}(\mu)\right)$. We also define $\tau_{\rm{str}}$ as the analogue of $\tau_{\rm{Morse}}$ in \eqref{tau flow lines}, replacing counts of flow lines by counts of holomorphic strips with boundary components mapping to $\R^3$ and $\Md$.

We next specify spin structures on $L_{K}$ and $M_{K}$. There is a natural identification $\iota$ of a neighborhood of the Legendrian torus at infinity in $L_{K}$ and $M_{K}$ with a punctured neighborhood of the knot $K$. Let $\iota^\ast$ denote $\iota$ composed with the orientation reversing diffeomorphism that changes the orientation of the fibers in the punctured neighborhood of $K$. This neighborhood contains the longitude and meridian of the knot $K$ that generates $H_{1}(L_{K})$ and $H_{1}(M_{K})$, respectively. Equip $L_{K}$ and $M_{K}$ with the unique spin structures that agree with the unique spin structure on $\R^{3}$ in this neighborhood (under the natural identification of neighborhoods). Equivalently, there are also natural identifications of $L_K$ with a neighborhood $\nu(K)\subset \R^3$ of $K$, and of $M_K$ with $\R^3\setminus K$; the spin structures on $L_K$ and $M_K$ are the ones inherited from $\R^3$ by restriction, under these identifications.

\begin{lemma}\label{signs of flow loops and annuli}
	Consider $\Ld$ or $\Md$ for sufficiently small $\delta>0$ so that Lemma \ref{annuli L} respectively \ref{annuli loops} holds. Then with one of the two identifications $\iota$ or $\iota^\ast$, the orientation sign $\sigma(u)$ of a holomorphic annulus and the sign of the corresponding flow loop agree.
\end{lemma}

\begin{proof}
	Consider a model Maslov zero annulus corresponding to a flow loop with normal structure corresponding to two eigenvalues of the return map smaller than 1, that we take to be a Bott family of local minima in fixed local coordinates. Consider also similar model annuli with normal structure corresponding to maxima (eigenvalues or the return map larger than 1) or saddles (eigenvalues of return map on opposite sides of 1). The boundary condition in the case of maxima can be rotated to that of minima by changing the pull-back trivialization to its negative, without affecting the kernel or the cokernel. Connecting the trivializations then shows that the signs agree. For the case of the saddle, the same type of deformation reverses the orientation of one of the Lagrangians. This alters the Fukaya orientation. 
 
    We next show that any transverse flow loop of a function $f\colon M_K\to S^1$  admits coordinates in which it can be deformed to one of the model problems above, without changing the signs of the differences of the eigenvalues of its return map and $1$. Pick two nearby points $p_1$ and $p_2$ on the orbit, so that there is a short flow line from $p_2$ to $p_1$. The kernel of $df$ then gives planes $\Pi$ normal to the flow loop along the flow loop, that we think of as the fibers in a tubular neighborhood of the flow loop. Integrating the restricted vector field $\nabla f|_{\Pi}$, along the sub-interval along the loop from $p_1$ to $p_2$ with initial conditions in $\Pi_{p_1}$ gives coordinates in a neighborhood of the origin in the normal planes. Similarly, integrating along the short path between $p_1$ and $p_2$ gives (almost constant coordinates) along the short arc between $p_2$ and $p_1$. In these coordinates, the linearized return map between $p_1$ and $p_2$ equals $0$. Changing coordinates near the midpoint of the short arc by the flow in the normal plane of the gradient of a suitable quadratic form so that the coordinate system closes up, then gives coordinates along the flow loop in which the return map is that of a constant quadratic form supported in a small arc. If we then spread out the support of the quadratic loop over the whole loop we get a flow loop corresponding to the model annuls above. It then follows that any flow loop admits a coordinate deformation that does not change the sign of the eigenvalues of the difference between the linearized return map and the identity to a model annulus. 

    It follows that the flow loop sign and the Fukaya sign agree up to an over all sign. Since the Fukaya sign changes when the orientation of one of the Lagrangians change, the lemma follows.          
\end{proof}


Lemma \ref{annuli L} has the following consequence, which will be useful below. 
\begin{corollary} \label{AnL}
	For any knot $K\subset\R^{3}$ we have	
	$$
	\An_{L_K}(\lambda) = \sum_{k>0} \frac{\lambda^{k}}{k} = - \log(1-\lambda).
	$$
\end{corollary}

\begin{remark} \label{simple annuli}
	The count of annuli $\An_{L_K}$ can be compared to the count of holomorphic disks in \cite{EkholmKucharskiLonghi1} and \cite{EkholmKucharskiLonghi2}, where so called basic disks are counted by the contribution of all their multiple covers. In the present case one could reformulate the count above by counting only simply covered annuli, but counting them by the contributions of all their multiple covers. In other words, if $\mathcal{M}_{\rm an}'(\Ld)$ is the moduli space of simple annuli then
	\[ 
	\An_{L_K}(\lambda) := \lim_{d\to \infty} \lim_{\delta\to 0} \sum_{[u] \in \widetilde \M_{\rm an}^{'d \delta}(\Ld)} \sum_{m=1}^{\infty}\frac{\sigma(u^m)}{m} \lambda^{m \, d(u)} \in \C[[\lambda]],
	\]	
where $u^m$ is the $m$-fold cover of $u$.
 
\end{remark}

\begin{remark}
	The count of annuli can also be compared to the count of holomorphic curves in the framed skein module of the brane as in \cite{EkholmShende}, where only so called bare curves are counted. Here we would have to perturb out the multiply covered annuli and count them in the $U(1)$-skein module. One could then simplify the count and map the $U(1)$ framed skein to `homology and framing', which corresponds to counting generalized holomorphic curves. Our count here should then correspond to counting generalized holomorphic curves of Euler characteristic $\chi=0$, which is a certain limit of the generalized curve count. In particular, the exact linking and self-linking of the boundary does not matter here, since such terms would contribute only to the counts of generalized curves of lower Euler characteristic. Since there are no holomorphic disks, the count $\Psi_{L_K,\R^{3}}$ of disconnected generalized curves should then be of the form
	\[ 
	\Psi_{L_K,\R^{3}}=\exp\left(\An_{L_K}(\lambda)+\mathcal{O}(g_{s})\right).
	\] 
\end{remark}

\begin{proof}[Proof of Theorem \ref{T:Alex zeta tau}]
The result follows from combining Equation \eqref{alex} with Lemmas \ref{annuli loops} and \ref{signs of flow loops and annuli}. 
\end{proof}

\begin{remark}
The product $\zeta_{\text{an}}(\mu) \cdot \tau_{\text{str}}(\mu)$ in Theorem \ref{T:Alex zeta tau} and the corresponding $\lambda$-dependent analogue for $L_K$  
can be thought of as torsions of the Lagrangian Floer complexes $CF^*(\R^3,\Md)$ and $CF^*(\R^3,\Ld)$, respectively (where $\Md$ and $\Ld$ are twisted by local systems). The invariance of Floer torsion under change of almost complex structure and Hamiltonian isotopy is studied in \cite{LeeTorsion}. We also discuss the invariance of Floer torsion in Section \ref{sec:SFT-stretch}. 
\end{remark}

\section{Knot contact homology and augmentations} 
\label{choices}
In this section we first review some aspects of knot contact homology, i.e., the Chekanov-Eliashberg dg-algebra of the Legendrian conormal of a knot which is generated by Reeb chords and with differential that counts punctured holomorphic disks, and discuss some geometric data that is used to define and compute the differential. We then turn to augmentations of knot contact homology and discuss augmentation varieties and augmentation polynomials. 

%

\subsection{Geometry of coefficients in the dg-algebra} 
\label{choices and coeffs}

The dg-algebra of an oriented knot $K\subset \R^3$ is the Chekanov--Eliashberg differential graded algebra of its Legendrian conormal $\Lambda_{K}\subset ST^{\ast}\R^{3}$. We denote it $\mathcal A_K$. It is a tensor algebra freely generated by the Reeb chords of $\Lambda_K$. The differential counts punctured holomorphic disks in $\R\times ST^*\R^3$, with boundary components mapping to $\R\times \Lambda_K$ and with boundary punctures (one positive and arbitrarily many negative) asymptotic to Reeb chords. The coefficient ring of $\mathcal A_K$ is the group ring $\C[H_2(ST^*\R^3,\Lambda_K;\Z)]$. We make the following choices. 
\begin{itemize}
 \item For each Reeb chord $c$, choose a capping half-disk $v_{c}\colon D^{2}_{+}\to ST^{\ast}\R^{3}$, where $D^{2}_{+}\subset \C$ is the intersection of the unit disk in the complex plane and the upper half plane. We require that $v_{c}|_{D^{2}_{+}\cap\R}$ is a parametrization of $c$, where $\R\subset \C$ is the real line. Also, $v_{c}(D^{2}_{+}\cap \mathbb{S}^{1})\subset \Lambda_{K}$, where $\mathbb{S}^{1}\subset \C$ is the unit circle. Up to homotopy, two capping half-disks differ by an element in $\pi_2(ST^*\R^3,\Lambda_K)$. 
 \item We fix embedded curves $x$ and $p$ in $\Lambda_{K}$ whose homology classes generate $H_1(\Lambda_K;\Z)$ and with intersection number $x\cdot p=1$ as follows. We require $x$ to be null-homologous in $M_{K}\approx \R^3\setminus K$ and $p$ to be null-homologous in $L_{K}\approx S^1\times \R^2$. We also require the projection of $x$ to $K$ to be orientation-preserving. This fixes the classes of $x$ and $p$ in $H_1(\Lambda_K;\Z)$. We call $x$ the \emph{longitude} and $p$ the \emph{meridian}, respectively, and we sometimes identify them with the homology classes that they represent. 
 As we shall see, the formulas in Theorem \ref{T:Alex Aug} are still valid if we change the meridian curve by a  transformation of the form 
 \begin{equation} \label{change framing}
 (x,p) \mapsto (x , p + k x) 
 \end{equation}
 for some $k\in \Z$. We call this a \emph{change of framing}. 
\item Choose a splitting of the short exact sequence
\begin{equation} \label{splitting ses}
 \begin{tikzcd}
    0 \arrow{r} & H_2(ST^*\R^3;\Z) \arrow{r} & H_2(ST^*\R^3,\Lambda_K;\Z) \arrow{r} & H_1(\Lambda_K;\Z) \arrow{r} \arrow[bend right=18,dashrightarrow]{l}{} & 0
 \end{tikzcd},
\end{equation}
as indicated by the dashed line. The sequence starts with $0$ because the fundamental class of $\Lambda_K$ vanishes in $H_2(ST^*\R^3;\Z)$.%
\footnote{The image of the fundamental class of $\Lambda_K$ in $H_2(ST^*\R^3;\Z)\cong H_2(S^2;\Z)\cong \Z$ is given by the degree of the Gauss map $T_K\to S^2$, where $T_K$ is the boundary of a tubular neighborhood of $K$. This degree vanishes, since it is one half of the Euler characteristic of the torus $T_K$, by the Poincar\'e--Hopf theorem.}
Let $t\in H_2(ST^*\R^3;\Z)$ be the class of a generator, which we think of as a fiber of the $S^2$-bundle $ST^*\R^3\to \R^3$. This class is unique if we require the intersection with a section of the bundle to be $+1$. The splitting of the sequence can be thought of as a choice of two surfaces in $ST^*\R^3$, one having $x$ as its boundary and the other having $p$ as its boundary. Denote these surfaces by $\Sigma_x$ and $\Sigma_p$, respectively. Note that the splitting is well-defined up to adding integer multiples of $t$ to $\Sigma_x$ and $\Sigma_p$.
We also pick a section $s$ of the trivial $S^2$-bundle $ST^*\R^3 \to \R^3$ that is generic in the following sense. The graph of $s$ is disjoint from all Reeb chords $c$ and is transverse to all capping half-disks $v_c$, to the capping surfaces $\Sigma_x$ and $\Sigma_p$, and to $\Lambda_K$. Any class in $H_2(ST^*\R^3)$ is determined by its intersection with the graph of $s$. 
\end{itemize}

Once we make these choices, we can identify the group ring $\C[H_2(ST^*\R^3,\Lambda_K;\Z)]$ with the Laurent polynomial ring $R \coloneq \C[\lambda^{\pm1},\mu^{\pm1},Q^{\pm1}]$, with $\lambda = e^x$, $\mu = e^p$ and $Q = e^t$. 
We describe how to compute the $R$-coefficient of the contribution from a holomorphic curve $u$, transverse to the section $s$, to the differential in $\mathcal A_K$, in accordance with the choices that we made. Some of the relevant information is sketched in Figure \ref{closed surface_fig}. 

Consider first the $\lambda$ and $\mu$ powers: use the capping disks associated to the asymptotic chords of $u$ to produce a closed disk $\overline u$. This closed disk represents an element in $H_2(ST^*\R^3,\Lambda_K;\Z)$. The boundary $\gamma$ of $\overline u$ gives a class $a [x] + b [p] \in H_1(\Lambda_K;\Z)$, where $a$ is the intersection number $\gamma \cdot p$ and $b$ is $x \cdot \gamma$. The $\lambda,\mu$ coefficients of the contribution of $u$ will then be $\lambda^a \mu^b$. 

\begin{remark} \label{rmk:derivatives}
We observing the following simple fact for future reference: 
$$\frac{d(\lambda^k)}{d x} =\lambda \frac{d(\lambda^k)}{d \lambda} = k \lambda^k.$$
\end{remark}

Consider next the $Q$ power: 
the closed disk $\overline u$ gives an element of $H_2(ST^*\R^3,\Lambda_K;\Z)$, and the boundary $\gamma$ of $\overline u$ is homologous to the 1-dimensional cycle $\gamma'\coloneq a x + b p$. We pick $S(\gamma)$ and a map $S(\gamma) \to \Lambda_K$, such that $S(\gamma)$ is a 2-dimensional oriented surface with boundary, and the boundary of $S(\gamma)$ maps to $\gamma'-\gamma$. We then have a closed surface in $ST^*\R^3$ by concatenating $\overline u$ with the image of $S(\gamma)$ and with $a$ copies of $-\Sigma_x$ and $b$ copies of $-\Sigma_p$ (where the sign indicates reversal of orientation). The $Q$-power of $u$ in $\mathcal A_K$ is the intersection number of this closed surface with the image of the section $s$. Note that $[\Lambda_K] = 0\in H_2(ST^*\R^3;\Z)$ implies that this intersection number is independent of the choice of cobordism $S(\gamma)$, as different choices differ by a multiple of the fundamental class of $\Lambda_K$.

\begin{figure}
  \begin{center}
    \def\svgwidth{.5\textwidth}
    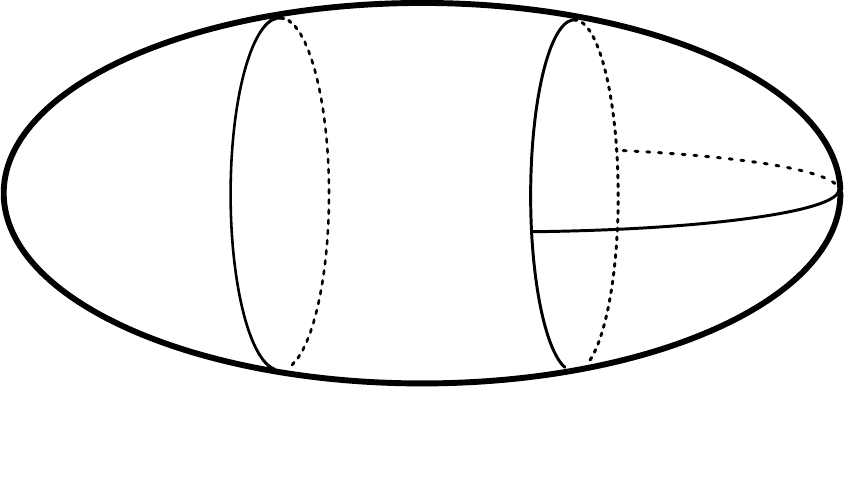
  \end{center}
  \caption{Completing $u$ into a closed surface to determine the $Q$-power}
  \label{closed surface_fig}
\end{figure} 

\begin{remark}
 The change of framing \eqref{change framing} above induces a change of variables in $R$ of the form 
 $(\lambda,\mu,Q)\mapsto (\lambda, \lambda^{k}\mu, Q)$, for some $k\in \Z$, and a change of splitting \eqref{splitting ses} induces $(\lambda,\mu,Q)\mapsto (\lambda Q^{l}, \mu Q^{m}, Q)$, for some $l,m\in \Z$. 
\end{remark}

For some steps in the derivation of the formula in Theorem \ref{T:Alex Aug}, it will be important to understand $Q$-powers in more detail. We have the following elementary result.

\begin{lemma} \label{compute Q}
 Given choices of capping half-disks and capping surfaces $\Sigma_x, \Sigma_p$, the exponent of $Q$ associated to a holomorphic curve $u$ is given by 
 $$
 \big(\overline u \cdot s)_{ST^*\R^3} + \big(\gamma(\overline u) \cdot \tau)_{\Lambda_K}
 $$
 for a suitable choice of 1-dimensional submanifold $\tau\subset \Lambda_K$ as above. Here, $\gamma(\overline u)$ is the boundary of the disk $\overline u$ as illustrated in Figure \ref{closed surface_fig}. Different choices of $\Sigma_x, \Sigma_p$ correspond bijectively to different choices of $\tau$. 
\end{lemma}

\begin{proof}
Consider first how the $Q$-power depends on the choices of capping surfaces $\Sigma_x, \Sigma_p$. These choices affect the intersection with the section $s$ of the surface $\Sigma(\gamma)$, which is the concatenation of the cobordism $S(\gamma)$ with $-a\Sigma_x - b\Sigma_p$ in Figure \ref{closed surface_fig}. Recall that $a = \gamma \cdot p$ and $b = x \cdot \gamma$. 
Note that $\Sigma(\gamma) \cdot s$, thought of as a function of $\gamma$, is well-defined on $H_1(\Lambda_K\setminus s;\Z)$. 
The capping surface $\Sigma_x$ can be changed by taking a connected sum with an $S^2$-fiber, and this changes $\Sigma(\gamma) \cdot s$ by subtracting $a$. Similarly, changing $\Sigma_p$ to $\Sigma_p \# S^2$ subtracts $b$ from $\Sigma(\gamma) \cdot s$. Therefore, we can think of $\Sigma(\gamma) \cdot s$ as a family of homomorphisms
$$
\Phi_{\Sigma_x,\Sigma_p}\colon H_1(\Lambda_K\setminus s;\Z) \to \Z
$$
such that, if we fix one element $\Phi_0$, we can obtain all other elements by adding homomorphisms of the form $\gamma \mapsto m \gamma\cdot x + n \gamma\cdot p$ for some $(m,n)\in \Z^2$. Since linear combinations $m [x] + n [p]$ denote arbitrary elements of $H_1(\Lambda_K;\Z)$, the collection of homomorphisms $\{\Phi_{\Sigma_x,\Sigma_p}\}$ can be thought of as a torsor over $H_1(\Lambda_K;\Z)$.

Let us now consider an alternative way of computing the powers of $Q$ associated to $u$, by expressing the intersection $\Sigma(\gamma) \cdot s$ differently. This will be used later in the paper, and goes as follows. 
Recall that $\Lambda_K$ is null-homologous in $ST^*\R^3$, and intersects $s$ transversely. Hence, this zero-dimensional intersection has signed count equal to zero, and is the boundary of a 1-dimensional submanifold $\tau\subset \Lambda_K$. 
Given two such choices $\tau, \tau'$, their difference defines an element in $H_1(\Lambda_K;\Z)$.
We can define another family of homomorphisms
$$
\Psi_{\tau}\colon H_1(\Lambda_K\setminus s;\Z) \to \Z
$$
by $\Psi_{\tau}(\gamma) = \gamma\cdot \tau$, where the intersection is now taken in $\Lambda_K$. Observe that $\Psi_{\tau \# x}(\gamma) = \Psi_{\tau}(\gamma) + \gamma\cdot x$ and $\Psi_{\tau \# p}(\gamma) = \Psi_{\tau}(\gamma) + \gamma\cdot p$, hence the collection $\{\Psi_{\tau}\}$ is also a torsor over $H_1(\Lambda_K;\Z)$. If we show that $\Phi_{\Sigma_x,\Sigma_p} = \Psi_{\tau}$ for one choice of $\Sigma_x,\Sigma_p$ and of $\tau$, then we can conclude that the collections $\{\Phi_{\Sigma_x,\Sigma_p}\}$ and $\{\Psi_{\tau}\}$ agree, and can thus replace an intersection $\Sigma(\gamma) \cdot s$ with $\gamma \cdot \tau$ for an appropriate $\tau$. 

Fix now a choice of $\Sigma_x,\Sigma_p$ such that $\Sigma_x \cdot s = \Sigma_ p \cdot s = 0$ (our freedom to take connected sums with $S^2$ guarantees the existence of such choices). Denote the corresponding homomorphism $\Phi_{\Sigma_x,\Sigma_p}$ by $\Phi_0$. Fix also a choice of $\tau$ such that $\tau \cdot x = \tau \cdot p = 0$ (which exists by our freedom to take connected sums with $x$ and $p$). Denote the corresponding $\Psi_{\tau}$ by $\Psi_0$. Then, indicating by a subscript where intersections take place, we have 
$$
\Phi_0(\gamma) = \big(\Sigma(\gamma) \cdot s\big)_{ST^*\R^3} = \big(S(\gamma) \cdot s\big)_{ST^*\R^3} = \big(S(\gamma)\cdot (s\cap \Lambda_K)\big)_{\Lambda_K} = \big(\gamma \cdot \tau\big)_{\Lambda_K} = \Psi_0(\gamma), 
$$
where the fourth identity follows from the usual argument showing that the linking number of two submanifolds of Euclidean space can be computed by intersecting either submanifold with a submanifold bounding the other. The lemma follows.
\end{proof}
\begin{remark}
 We can think of the intersection $\Sigma(\gamma)\cdot s$ as a version of a linking number between $\gamma$ and $s$ in $ST^*\R^3$. The discussion above shows that this can be replaced with $\gamma \cdot \tau$, which can be interpreted as a linking number between $\gamma$ and $s\cap \Lambda_K$ in $\Lambda_K$. 
\end{remark}


\subsection{The augmentation variety and polynomial} \label{sec:aug var poly}
In this section we discuss various aspects of augmentations.  
Consider, as in Section \ref{choices and coeffs}, $\mathcal A_K$ with coefficients in 
$R = \C[\lambda^{\pm 1}, \mu^{\pm 1}, Q^{\pm 1}]$. Recall our standing assumption that $\Lambda_K$ bounds no chords of negative degree. An \emph{augmentation} of $\A_K$ is a unital chain map of dg-algebras, 
$$
\epsilon  \colon \A_K \to \C,
$$
where $\C$ is supported in degree 0 and the chain map condition means that $\epsilon \circ \partial = 0$.%
\footnote{Since we write $\lambda = e^x$, we can heuristically think that the augmentation also assigns a value to $x$, namely $\log(\epsilon(\lambda))$ (which is only defined up to integer multiples of $2 \pi i$), and analogously for $\mu = e^p$ and $Q = e^t$. In Theorem \ref{T:Alex Aug}, it would be slightly more precise to replace the condition $(x,t)=(0,0)$ by $(\lambda,Q)=1$, or by $(e^x,e^t)=(1,1)$.} 
The linearized contact homology with respect to $\epsilon$, denoted $KCH^{\epsilon}_*(K;\C)$, is the homology of the $\C$-vector space $KCC^{\epsilon}_*(K;\C)$ generated by Reeb chords of $\Lambda_K$, with the following differential. Given a chord $c$, interpret its knot contact homology differential $\partial c$ as a polynomial in variables corresponding to the Reeb chords of $\Lambda_K$. To find the coefficient of a chord $b$ in the linearized differential of $c$, 
take the partial derivative of that polynomial with respect to the variable $b$, and apply $\epsilon$ to the polynomial resulting from the partial derivative. More formally, we can write
\begin{equation} \label{def:lin diff}
\partial^\epsilon(c) = \sum_{b \, \rm{ chords}} \epsilon\left( \frac{\partial (\partial c)}{\partial b}\right) b.
\end{equation}
Equivalently,  
for each non-constant monomial on chords in $\partial c$, 
sum the result of applying $\epsilon$ to the coefficient variables $\lambda$, $\mu$ and $Q$, and to all but one of the chords in the monomial, in all possible ways. If $c$ is a Reeb chord of degree $k$, we will sometimes write $\partial_k^\epsilon(c)$ instead of $\partial^\epsilon(c)$.

An exact Lagrangian filling $L$ of $\Lambda_K$, inside an exact symplectic filling $(X,d\alpha)$ of $ST^*\R^3$, induces a family of augmentations $\epsilon_L$ of $\A_K$ as follows. Associate to every degree 0 Reeb chord $a$ of $\Lambda_K$ the count of holomorphic disks in $X$, with boundary on $L$ and one positive puncture asymptotic to $a$. The choices of capping data that were made above (with the goal of defining $\mathcal A_K$ over the group ring $R=\mathbb C[H_2(ST^*\mathbb R^3,\Lambda_K;\Z)]$) enable us to define these counts of disks in $X$ with coefficients in $\C[H_2(X,L;\Z)]$. The values of the augmentations $\epsilon_L$ on elements of the coefficient ring $R$ must be compatible with the inclusion maps on homology:
\begin{equation} \label{diagram ses}
\begin{tikzcd}
    0 \arrow{r} & H_2(ST^*\R^3;\Z) \arrow{r}\arrow{d} & H_2(ST^*\R^3,\Lambda_K;\Z) \arrow{r}\arrow[d,"i_*"] & H_1(\Lambda_K;\Z) \arrow{r}\arrow{d} & 0 \\
    H_2(L;\Z) \arrow{r} & H_2(X;\Z) \arrow{r} & H_2(X,L;\Z) \arrow{r} & H_1(L;\Z) \arrow{r} & 0
\end{tikzcd}. 
\end{equation}
Specifically, the restrictions of the augmentations to $H_2(ST^*\R^3,\Lambda_K;\Z)$ must factor through the map $i_*$ in the diagram above. This is illustrated in the following example. 

\begin{example} \label{LK and MK}
For every knot $K$, the conormal Lagrangian $L_K$ in $T^*\R^3$ is an exact Lagrangian filling of $\Lambda_K$. The meridian $p$ is null-homologous in $L_K$ and the generator $t\in H_2(ST^* \R^3;\Z)$ vanishes in $H_2(T^* \R^3;\Z)= 0$. There is an associated 1-parameter family of augmentations, that we denote by $\epsilon_{L_K}$, assigning the value 1 to both $\mu$ and $Q$, and assigning to $\lambda$ an arbitrary value in $\C^*$. 
Similarly, the exact Lagrangian filling $M_K$, obtained from clean intersection surgery on $L_K$ and $\R^3$, on which the longitude $x$ is null-homologous, gives a family of augmentations $\epsilon_{M_K}$, assigning 1 to $\lambda$ and $Q$ and an arbitrary element in $\C^*$ to $\mu$. 
\end{example}


\begin{lemma} \label{epsilon LK v epsilon MK}
The augmentation $\epsilon_{L_K}$ with $\lambda=1$ is the same as the augmentation $\epsilon_{M_K}$ with $\mu=1$, for $M_K$ obtained from surgery on $L_K\cup \R^3$ with a small surgery parameter. 
\end{lemma}
\begin{proof}
Since $M_K$ is obtained by surgery on $L_K$ and $\R^3$, we can assume that $M_K$ is very close to the union $L_K\cup\R^3$. For such $M_K$, augmentation disks correspond to either disks with boundary on $L_{K}$ or disks with switching boundary conditions on $L_K\cup \R^3$, as in \cite{CELN}. In \cite{AENV}, the moduli space of disks for small smoothings near a rigid disk with switching boundary condition was described, see also \cite[Lemma 4.1]{ED-RT} for a more detailed description. The result is that a corner switching from $L_{K}$ to $\R^{3}$ has a unique smoothing, and a corner from $\R^{3}$ to $L_{K}$ has two smoothings with boundaries that differ by the meridian. This means that 
\begin{equation} \label{e MK from e LK}
\epsilon_{M_{K}}=\epsilon_{L_{K}} + \sum_{k=1}^{m}\pm (1-\mu)^{k}\phi^{k}_{L_{K}\cup\R^{3}},
\end{equation}
where $\phi_{L_{K}\cup\R^{3}}^{k}$ is a count of rigid disks with switching boundary conditions in $L_{K}\cup\R^{3}$ and $2k$ corners. 
Therefore, at $\mu=1=\lambda$ we have $\epsilon_{M_K}(a_{i})=\epsilon_{L_K}(a_{i})$ as claimed. 
\end{proof}

\begin{remark} \label{remark: braided model DGA}
Although not necessary for this article, it is sometimes useful to know that for any knot $K$ there is a model for $\Lambda_K$ and its DGA $\mathcal A_K$ for which the augmentations induced by $L_{K}$ vanish on all Reeb chords of degree 0. To this end, we braid $K$ with $n$ strands around the round unknot $U$. Then, as explained in \cite{KCH}, the Legendrian $\Lambda_K$ has a natural $n$-to-one cover of $\Lambda_U$. There are no Reeb chords in negative degrees for $\Lambda_K$, and there are $n(n-1)$ Reeb chords of degree zero, that we denote $a_{i}$. As usual, we write $\boldsymbol{a}=(a_{i})$. 

Let us now see why the augmentations induced by $L_K$ vanish on degree zero chords. This is trivial for the unknot $U$, since $\Lambda_U$ does not have any Reeb chords of degree zero. Note also that, since $L_{U}$ is exact, there are no unpunctured holomorphic disks with boundary on $L_{U}$. For a general knot $K$, braid it closer and closer to $U$. As $K\to U$, the limit of an augmentation disk on $L_{K}$ is a holomorphic disk on $L_{U}$ with flow trees attached. This is analogous to the description of the differential in the DGA of $\Lambda_K$ in terms of multiscale flow trees, see \cite{KCH}. As $K$ approaches $U$, $L_{K}$ limits to a cover of $L_{U}$ without singularities, i.e., as a local diffeomorphism. Therefore there are no flow trees corresponding to disks with one positive puncture at a degree $0$ chord, and also no disks with flow trees attached since there are no disks on $L_{U}$. It follows that the augmentation $\epsilon_{L_{K}}$ is trivial on the $a_{i}$.

Note that Lemma \ref{epsilon LK v epsilon MK} implies that in this model the augmentation induced by $M_{K}$ at $(\lambda,\mu,Q)=(1,1,1)$ also vanishes on all degree 0 Reeb chords. The values of the augmentations $\epsilon_{M_K}$ on chords of degree 0 are written in \cite[Definition 4.6]{NgFramed}, using slightly different conventions than the ones we use. For relations between conventions used in several papers on knot contact homology, see \cite[page 527]{NgTopological}).  
\end{remark}

It is useful to think of the collection of all augmentations of $\mathcal A_K$ geometrically. To this end, we use the following notation. Label the Reeb chords of degree 0  by $a_i$ and denote the ordered set of such chords by $\boldsymbol{a} = (a_{i})$. Write $\C^{\boldsymbol{a}}$ for a vector space with basis $\boldsymbol{a}$ and write $\tilde R \coloneq R[\boldsymbol{a}]$ for the polynomial ring in variables $\boldsymbol{a}$. Let $\tilde I_K \subset \tilde R$ be the ideal generated by the differentials of chords of degree 1 in $\mathcal A_K$. 
We will consider the following geometric objects in connection with augmentations:

\begin{definition}
$\quad$
\begin{itemize}
\item The \emph{full augmentation scheme} of $K$ is the affine scheme
$$
\mathbb{\tilde V}_K \coloneq \Spec \left( \tilde R /\tilde I_K \right).
$$
\item The \emph{full augmentation variety} is the algebraic set 
$$
\tilde V_K \coloneq  V(\tilde I_K) = \left\{(\epsilon(\lambda),\epsilon(\mu),\epsilon(Q),\epsilon(\boldsymbol{a})) \in (\C^*)^3 \times \C^{\boldsymbol{a}} \, |\, \text{$\epsilon$ is an augmentation} \right\}.
$$
This is the set of closed points in $\mathbb{\tilde V}_K$.
\item The \emph{augmentation variety} $V_K$ of $K$ is the union of maximal-dimensional components of the Zariski-closure of the set
$$
\left\{(\epsilon(\lambda),\epsilon(\mu),\epsilon(Q))\in (\C^*)^3 \,|\, \epsilon \text{ is an augmentation for }\A_K\right\} = \pi(\tilde V_K),
$$
where $\pi\colon (\C^*)^3 \times \C^{\boldsymbol{a}} \to (\C^*)^3$ is the projection that forgets the Reeb chord variables $a_{i}$.
\item Corollary \ref{Aug exists} below shows that $I(V_K)$ is a principal ideal. A generating element of $I(V_K)$ is called an \emph{augmentation polynomial} of $K$, and denoted $\Aug_K(\lambda,\mu,Q)$.
\end{itemize}
\end{definition}

\begin{remark}
The augmentation polynomial $\Aug_K$ can often be computed using elimination theory. See Section \ref{elimination} for details. 
\end{remark}

\begin{example}
 The augmentation polynomial of the unknot is $1-\lambda-\mu+\lambda\mu Q$, see \cite{AENV}.
\end{example}

We next consider first order properties of the geometric spaces of augmentations. We start with a general fact, see \cite[Remark 4.8]{EkholmNgHigherGenus}. 
\begin{lemma} \label{KCH0 tangent to fiber}
Let $\epsilon$ be an augmentation of $\mathcal A_K$, and denote the corresponding closed point in $\mathbb{\tilde V}_K$ also by $\epsilon$. 
The Zariski cotangent space to $\pi^{-1}\left(\pi(\epsilon)\right) \subset \mathbb{\tilde V}_K$ at the point $\epsilon$ is isomorphic to the degree zero $\epsilon$-linearized homology $KCH_0^\epsilon(K;\C)$. 
\end{lemma}

\begin{proof}
Recall that we assume that $KCC_*$ is supported in non-negative degrees.
Denote the set of Reeb chords of degree 0 by $\boldsymbol{a} = (a_i)$ and the set of chords of degree 1 by $\boldsymbol{b} = (b_j)$. 
Write $\Phi_{j}$ for the contact homology differential $\partial b_{j}$, viewed as a polynomial in  $\lambda^{\pm 1}$, $\mu^{\pm 1}$, $Q^{\pm 1}$, and $a_{i}$. Consider the covector 
$$
  d_\epsilon\Phi_{j} = \sum_{i}\left.\frac{\partial\Phi_{j}}{\partial a_{i}}\right|_{\epsilon}da_{i} \in \Span_\C\left(\{da_i\}\right)
$$
and let $d\Phi^\epsilon$ denote $\Span_\C\left(\{d_\epsilon\Phi_{j}\}\right)$. 
We can identify the Zariski cotangent space to the fiber $\pi^{-1}\left(\pi(\epsilon)\right)$ at $\epsilon\in\C^{\boldsymbol{a}}$ with the quotient 
\begin{equation} \label{eq:cotangent space}
\Span_\C\left(\{da_i\}\right)/ d\Phi^\epsilon. 
\end{equation}

On the other hand, the $\epsilon$-linearized contact homology differential of $b_{j}$, defined in \eqref{def:lin diff} above, can be written as 
\begin{equation} \label{eq:lin diff}
\Psi_{j}^\epsilon = \partial_1^\epsilon(b_{j}) = \sum_{i}\left.\frac{\partial\Phi_{j}}{\partial a_{i}}\right|_{\epsilon} a_{i} \in \Span_\C\left(\{a_i\}\right) = \C^{\boldsymbol{a}}.
\end{equation}
Writing $\Psi^\epsilon = \im \partial_1^\epsilon = \Span_\C\left(\{\Psi_{j}^\epsilon\}\right) \subset \C^{\boldsymbol{a}}$, we see that 
$$
KCH_0^\epsilon(K;\C) = \C^{\boldsymbol{a}} / \Psi^\epsilon
$$
can be naturally identified with \eqref{eq:cotangent space} above, replacing the basis vectors $a_i$ with the $da_i$. 
\end{proof}

It will be useful below to have a similar description of the whole cotangent space to a point in $\mathbb{\tilde V}_K$, instead of just the cotangent spaces to the fibers of the projection $\pi$. To this end, we define the \emph{full linearized contact homology} $\widehat{KCH}_*^\epsilon(K;\C)$ as the  
homology of the chain complex obtained by linearizing the knot contact homology differential of $K$ also with respect to the coefficients $\lambda$, $\mu$, and $Q$. More precisely, and analogously to \eqref{eq:lin diff}, the $\epsilon$-linearized differential of a degree 1 chord $b_{j}$ is the vector 
\begin{equation} \label{diff tot}
\widehat \partial_1^\epsilon(b_{j}) = \left.\frac{\partial\Phi_{j}}{\partial \lambda}\right|_{\epsilon} \lambda+\left.\frac{\partial\Phi_{j}}{\partial \mu}\right|_{\epsilon} \mu+\left.\frac{\partial\Phi_{j}}{\partial Q}\right|_{\epsilon} Q + \sum_{j}\left.\frac{\partial\Phi_{j}}{\partial a_{i}}\right|_{\epsilon} a_{i} \in \C^{3}\oplus\C^{\boldsymbol{a}}.
\end{equation}


We then have the following:
\begin{lemma} \label{KCH0 tangent}
Let $\epsilon$ be an augmentation of $\mathcal A_K$. The Zariski cotangent space to $\mathbb{\tilde V}_K$ at $\epsilon$
is isomorphic to $\widehat{KCH}_0^\epsilon(K;\C)$. 
\end{lemma}
\begin{proof}
Identical to the proof of Lemma \ref{KCH0 tangent to fiber}.	
\end{proof}

\section{Analysis of boundaries of moduli spaces} \label{sec:main proof}

In this section, we begin by discussing gluing theorems associated to degenerations of annuli that do not have counterparts for disks. We then identify boundaries of moduli spaces that will be used in the proof of Theorems \ref{t:linearizedhomology} and \ref{T:Alex Aug} in Section \ref{sec:LCH computations}. There, the key point is to equate quantities that count ends of one-dimensional moduli spaces. To control compactifications of moduli spaces there are two ingredients: compactness and gluing. Compactness controls the possible limits of a sequence holomorphic curves. Here we use Gromov compactness and SFT compactness. Gluing produces a neighborhood in the moduli space of a broken configuration in such a limit. Here we use Floer gluing, which can be described as an infinite dimensional version of Newton iteration. We consider the results on disks as standard and focus on two gluing theorems that arise for annuli.

\subsection{Gluing for once punctured annuli}
 The first case of degeneration and gluing for annuli was called {\em elliptic boundary splitting} in \cite{EkholmShende}, and is depicted in Figure \ref{elliptic_bdry_fig}. It corresponds to the limit where the modulus of the annulus converges to infinity, by having one of its boundary loops shrink to a point, and can be described as follows. 
Use coordinates $(z_{1},z_{2},z_{3})=(x_{1}+iy_{1},x_{2}+iy_{2},x_{3}+iy_{3})$ on $\C^{3}$. The Lagrangian boundary condition corresponds locally to $\R^{3}=\{y_{1}=y_{2}=y_{3}=0\}$. Take the vector field $s=\partial_{y_{3}}$, and the 4-chain  
\[ 
V= \{(z_{1},z_{2},z_{3})\colon y_{1}=y_{2}=0,y_{3}\ge 0\},
\] 
oriented so that it induces the positive orientation on $\R^{3}$.

\begin{figure}
	\begin{center} \hspace{.3cm}
		\def\svgwidth{.55\textwidth}
		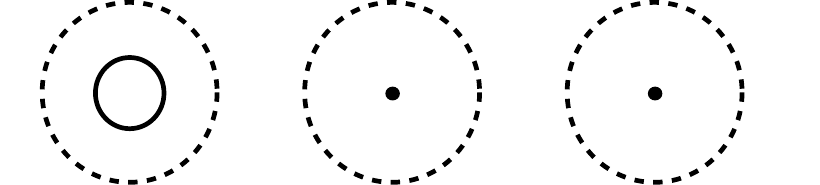
	\end{center}
	\caption{Elliptic boundary splitting}
	\label{elliptic_bdry_fig}
\end{figure}

Consider the family of maps
$u_{\rho}\colon [-\rho,\infty)\times S^{1}\to\C^{3}$, depending on a parameter $\rho>0$, given by
\begin{equation}\label{eq:imaginarynode} 
u_\rho(\zeta)=\left(e^{\zeta}+e^{-\zeta-2\rho},-i(e^{\zeta}-e^{-\zeta-2\rho}),0\right),
\end{equation}
where the domain is thought of as a subset of the cylinder $\C/2\pi i\Z$. Note that $u_\rho(\{-\rho\}\times S^{1})$ is a circle of radius ${2}e^{-\rho}$ in the $x_{1}x_{2}$-plane and that the interior of $u_\rho$ is disjoint from the $4$-chain. 
Consider also the family 
$v_{t}\colon[-\infty,\infty)\times S^{1}/(\{-\infty\}\times S^1)\to\C^{3}$, depending on a parameter $t\geq 0$, given by
\[ 
v_{t}(\zeta)= (e^{\zeta},-ie^{\zeta},it), \quad t\in(-\epsilon,\epsilon),
\]
where we think of the domain as a smooth disk. The intersection $v_t\cap V$ is the point $(0,0,it)$. As $\rho\to\infty$ the map $u_\rho$ converges to $v_{0}$, which corresponds to when the family $v_{t}$ crosses the Lagrangian $\R^3$. This allows us to identify the limit nodal curve $u_{\infty}$ with the smooth curve $v_{0}$. We will sometimes write $u_t$ instead of $u_\rho$, under the identification $t = 1/\rho$.

Boundaries of a generic 1-parameter family of annuli in which the conformal ratio is going to infinity are locally modelled on the above family, in the following sense. Let $a_t$ for $t \in [0, \epsilon)$ and $b_t$ for $t \in (-\epsilon, \epsilon)$ be two families of maps to $(T^{\ast}\R^{3}, L)$, where $L$ is a Lagrangian (in practice $L=\R^{3}$, below), such that 
$a_0$ is a map from a domain with an elliptic node, with the node mapping
to $p \in L$, and $b_0$ is the corresponding map from the normalization of $a_0$. 

Then the pair $(a_t, b_t)$ is a {\em standard elliptic degeneration} if there exists a neighborhood $U$ around $p$ that can be identified with a neighborhood of the origin in $\C^{3}$ with the Lagrangian corresponding to $\R^{3}$ and an explicit 4-chain such that there is a diffeomorphism $U\to \C^{3}$ that respects the Lagrangian and the 4-chain and carries the intersections of the curves in the family with $U$ to the curves in the model family $(u_t, v_t)$ above. If instead $a_t$ is a nodal family for $t \in (-\epsilon, 0]$, we again say the pair is a standard elliptic degeneration if there is an identification as above with time reversed.     

Let $L\subset T^{\ast}\R^{3}$ be a Lagrangian submanifold asymptotic to a Legendrian $\Lambda\subset ST^{\ast}\R^{3}$ and let $c$ be a degree one Reeb chord of $\Lambda$. Write $A^{\circ}$ and $D^{\circ}$ for an annulus and a disk, respectively, with a boundary puncture.

\begin{lemma}\label{l:standardnode}
Let $u_{t}\colon (A_{t}^{\circ},\partial A_{t}^{\circ})\to (T^{\ast}\R^{3},\R^{3}\cup L)$, $t\in(0,\epsilon]$ be a transversely cut out 1-parameter family of annuli with a positive boundary puncture at $c$, with limit $u_0=\lim_{t\to 0}u_t$ a map with an elliptic node. Let $v_{t}\colon (D^{\circ},\partial D^{\circ})\to(T^{\ast}\R^{3},L)$ be a 1-parameter family of disks with a positive boundary puncture at $c$ and intersecting $V$ transversely at an interior marked point. If the normalization of $u_{0}$ equals $v_{0}$, then the pair $(u_{t},v_{t})$ is a standard elliptic degeneration. 

It follows in particular that in the transverse case there is a natural bijection between those boundaries of 1-parameter families of annuli corresponding to maps with an elliptic node, on the one hand, and instances in moduli spaces of disks where the maps take an interior point to $\R^3$, on the other hand (see Figure \ref{elliptic_bdry_fig}).   
\end{lemma}

\begin{proof}
This is \cite[Lemma 4.16]{EkholmShende}. We recall the proof, which is a standard application of Floer gluing.
The normalization of $u_0$ equals $v_0$, which intersects $\R^3$ at an interior point. Consider half-cylinder coordinates $[0,\infty)\times S^1$ around the preimage of the intersection point. We construct a pre-glued annulus by cutting the neighborhoods at large $\rho_0\in[0.\infty)$ and interpolating to a cut-off version of the standard model of the elliptic node \eqref{eq:imaginarynode}. At this pre-glued curve, we invert the differential of the Cauchy--Riemann operator on the complement of the linearized variation of the 1-parameter family and then establish the quadratic estimate needed for Floer’s Newton iteration argument. Then, for sufficiently large $\rho>0$, there are uniquely determined solutions in a neighborhood of the pre-glued annuli and arbitrarily near to it in the functional norm that we use (which as in \cite{EkholmShende} we take to be the Sobolev norm of maps with two derivatives in $L^2$) and thus the solution will be $C^1$-close to the pre-glued curve.

This gives the desired gluing result for the moduli spaces involved, i.e., the 1-dimensional space of smooth annuli near a nodal annulus is obtained by gluing that nodal annulus. Since the pre-glued family is clearly conjugate to the standard family and the glued family is arbitrarily $C^1$-close, it is also conjugate to the standard family.
\end{proof}

%
%
%

The second degeneration concerns instants when annuli without positive punctures split off of 1-parameter families of annuli with one positive puncture. This was called {\em hyperbolic boundary splitting} in \cite{EkholmShende}, see the bottom arrow in Figure \ref{annulus_split_fig}. The limit configuration consists of a rigid annulus and a disk intersecting it at a boundary point. Call such a configuration a {\em split annulus}. The annulus component of a split annulus may be multiply covered, however it inherits a boundary marked point from the punctured disk that breaks the rotational symmetry of the multiple cover. With these marked points the gluing becomes standard: one disk family is glued to one annulus with boundary marked point. 

More specifically, let $\mathcal M^*(c)$ denote the moduli space of disks with one boundary puncture at 1 mapping to a Reeb hord $c$ of index 1 and a boundary marked point at $-1$. The virtual dimension of this space is 2. Let $\mathcal M_{\rm an}^*(L)$ denote the moduli space of annuli with a boundary marked point at $1$ (in the boundary component mapping to $L$). Its virtual dimension is 1. We observed in Section \ref{sec:flow lines and holo curves} that, in our setting, we can assume both moduli spaces to be transversely cut out. The moduli space of split annuli is $(\ev_{1} \times \ev_{-1})^{-1}(\Delta_L)$, where 
$$
\ev_{1} \times \ev_{-1} \colon \mathcal M_{\rm an}^*(L) \times \mathcal M^*(c) \to L\times L
$$
is the product of evaluation maps and $\Delta_L \subset L\times L$ is the diagonal. Adapting arguments from e.g.~\cite[Section 4.4.1]{EkholmShende} or \cite{McDuffSalamon2}, one shows that $\ev_{1} \times \ev_{-1}$ can be assumed transverse to $\Delta_L$. As illustrated by the bottom arrow in Figure \ref{annulus_split_fig}, $(\ev_{1} \times \ev_{-1})^{-1}(\Delta_L)$ is a boundary stratum of the Gromov compactification of the moduli space $\mathcal M_{\rm an}^*(c,L)$ of annuli with one boundary puncture asymptotic to $c$.

\begin{figure}
	\begin{center}
		\def\svgwidth{.5\textwidth}
		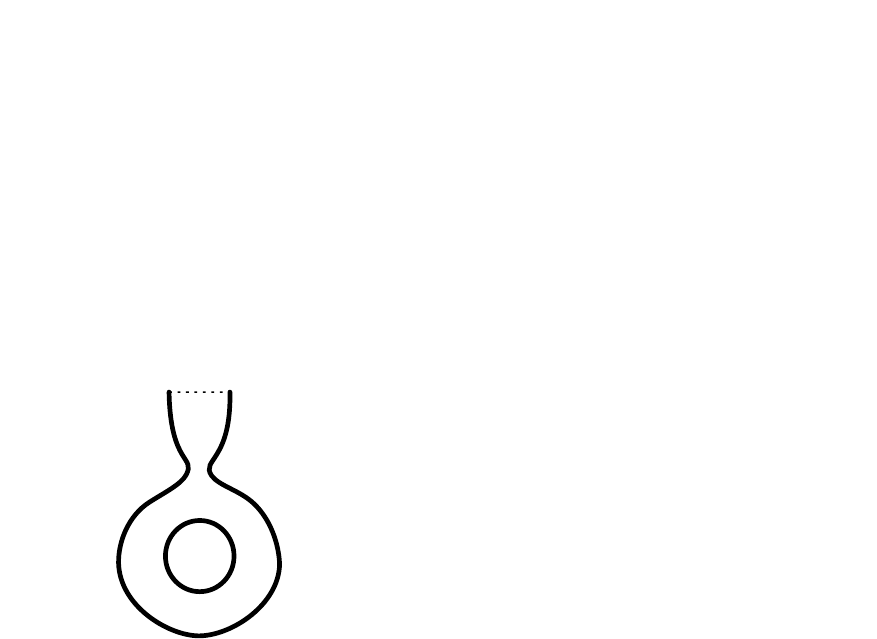
	\end{center}
	\caption{Horizontally, a family of disks crossing a rigid annulus. Diagonally, hyperbolic boundary splitting of a family of annuli with a boundary puncture}
	\label{annulus_split_fig}
\end{figure}

\begin{lemma}\label{l:gluingmultipleannuli}
Let $v_{t}\colon (D^{\circ},\partial D^{\circ})\to (T^{\ast}\R^{3},L)$ be a transversely cut out 1-parameter family of disks with one boundary puncture at $c$ and let $u\colon (A_{r},\partial A_r)\to (T^{\ast}\R^{3},\R^{3}\cup L)$ be an embedded rigid annulus, such that $v_{t}$ intersects the boundary of $u(A_{r})$ transversely at $t=0$. Then, for each such intersection point and each $d>0$, there is a 1-parameter family of annuli with one positive puncture at $c$ and whose Gromov-limit consists of $v_0$ and the $d$-fold cover of $u$.

In other words, there is a neighborhood of $(\ev_{1} \times \ev_{-1})^{-1}(\Delta_L)$ in the Gromov compactification of $\mathcal M_{\rm an}^*(c,L)$ that is diffeomorphic to the product of $(\ev_{1} \times \ev_{-1})^{-1}(\Delta_L)$ with the interval $[0,1)$.
\end{lemma}

\begin{proof}
	Since the markers kill the automorphisms of the annuli, the proof reduces to the well-known case of gluing of holomorphic curves with two boundary components meeting with transverse evaluation maps at a nodal point, see for example \cite{FOOO} or \cite[Theorem 4.1.2]{BiranCorneaQuantum}.
\end{proof}

\begin{remark}
Lemma \ref{l:gluingmultipleannuli} is a simple example of a gluing theorem where multiple covers are left unperturbed. In the present situation, only one of the components is multiply covered and since it is an annulus only regular covers contribute. In general, when two simple curves that both admit formally rigid multiple covers intersect on the boundary, one must take into account contributions from all of the branched covers with constant curves attached when gluing. For a further discussion of this point, see \cite{EkholmKucharskiLonghi2}.
\end{remark}

%
%
%

\subsection{The boundary of the space of once-punctured annuli}
Let $K\subset\R^{3}$ be a knot and let $\Lambda_{K}$ be its Legendrian knot conormal. We write $L$ for a Lagrangian filling of $\Lambda_K$, where either $L=L_{K}^{\delta}$ or $L=M_{K}^{\delta}$. We assume that $L\cap \R^{3}$ is transverse and consists of a collection of points $\xi_{j}$ of index $1$ and $\eta_{j}$ of index $2$. (For $L_{K}^{\delta}$, there are no points in the intersection and for $M_{K}^{\delta}$, intersection points correspond to critical points of a circle valued Morse function.) Fix a Reeb chord $c$ of degree 1.

Denote by $\mathcal{M}_{\mathrm{an}}(c)$ the moduli space of holomorphic annuli with one positive boundary puncture at $c$, one boundary component on $\R^{3}$ and the other on $L$. Its boundary components are the following, see Figure \ref{bdryL1_fig}. 
%
%
%
\begin{itemize}
\item $\mathcal{M}(c)\times_{\ev_{\partial}}\mathcal{M}_{\mathrm{an}}(\R^{3},L)$ is a fibered product of moduli spaces. $\mathcal{M}(c)$ is the moduli space of holomorphic disks with boundary on $L$ and one positive puncture asymptotic to $c$, and $\mathcal{M}_{\mathrm{an}}(\R^{3},L)$ is the moduli space of annuli with one boundary component on $\R^{3}$ and the other on $L$. The fibered product is over boundary evaluation maps. Note that the orientation of the moduli spaces, of the boundary of the disk, and of the Lagrangians induces an orientation of $\mathcal{M}(c)\times_{\ev_{\partial}}\mathcal{M}_{\mathrm{an}}(\R^{3},L)$.  
\item $\mathcal{M}(c;L)$ is the moduli space of holomorphic disks in $T^{\ast}\R^{3}$, with boundary on $L$ and an interior marked point mapping to $\R^{3}$. Note that the orientation of the moduli space, of the domain disk, and of the Lagrangian induces an orientation of $\mathcal{M}(c;L)$.
\item $\mathcal{M}_{\mathrm{an}}(c,\boldsymbol{a}):=\bigcup_{j}\mathcal{M}(c,a_{j})\times \mathcal{M}_{\mathrm{an}}(a_{j})$ is a union of products of moduli spaces. $\mathcal{M}(c,a_{j})$ is the moduli space of strips that give the contribution of $a_{j}$ to the $\epsilon_L$-linearized differential of $c$. These strips in $\R\times ST^*\R^3$ have a positive puncture asymptotic to $c$ and a negative puncture asymptotic to $a_{j}$. They may have additional negative punctures asymptotic to Reeb chords of degree 0, in which case each such chord comes with a punctured disk in $T^*\R^3$ with boundary on $L$ and boundary puncture asymptotic to the Reeb chord. The space $\mathcal{M}_{\mathrm{an}}(a_{j})$ is the moduli space of annuli with a positive boundary puncture asymptotic to $a_{j}$, one boundary component on $L$ and the other on $\R^{3}$.
\item $\mathcal{M}_{\mathrm{an}}(c)\times_{\ev_{\partial}}\mathcal{M}(L)$ is a fibered product over boundary evaluation maps, where $\mathcal{M}(L)$ is the moduli space of holomorphic disks with boundary on $L$. As above, there is a natural orientation induced on $\mathcal{M}_{\mathrm{an}}(c)\times_{\ev_{\partial}}\mathcal{M}(L)$.
\item $\mathcal{M}(c,\xi)=\bigcup_{j}\mathcal{M}(c,\xi_{j},\xi_{j})$ is a union of moduli spaces. $\mathcal{M}(c,\xi_{j},\xi_{j})$ is the space of holomorphic disks with boundary on $L\cup\R^{3}$ with three boundary punctures: one positive puncture asymptotic to $c$ and two punctures asymptotic to the intersection point $\xi_{j}$.
\item $\mathcal{M}(c,\eta)=\bigcup_{j}\mathcal{M}(c,\eta_{j},\eta_{j})$ is a similiar union of moduli spaces. $\mathcal{M}(c,\eta_{j},\eta_{j})$ is the space of holomorphic disks with boundary on $L\cup\R^{3}$ with three boundary punctures: one positive puncture asymptotic to $c$ and two punctures asymptotic to the intersection point $\eta_{j}$.
\end{itemize}

\begin{figure}
	\begin{center}
		\def\svgwidth{.9\textwidth}
		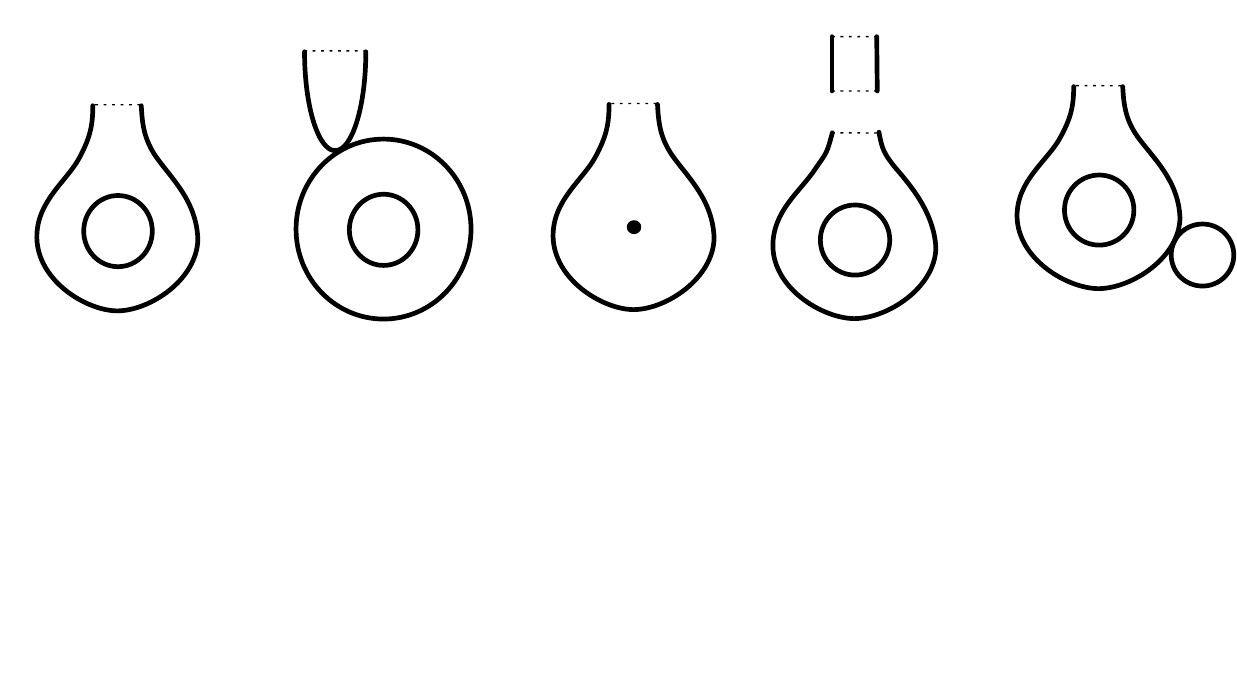
	\end{center}
	\caption{The boundary of $ \mathcal{M}_{\mathrm{an}}(y)$}
	\label{bdryL1_fig}
\end{figure}

\begin{lemma}\label{l:boundaryannuliy} 
For a generic almost complex structure, the moduli space $\mathcal{M}_{\mathrm{an}}(c)$ is transversely cut out and forms a 1-manifold with a natural compactification. Its (oriented) boundary is 
\begin{align}\notag
\partial \mathcal{M}_{\mathrm{an}}(c) \ &= \ 
\mathcal{M}(c)\times_{\ev_{\partial}}\mathcal{M}_{\mathrm{an}}(\R^3,L) \ \cup \ 
\mathcal{M}(c;L) \ \cup \ 
\mathcal{M}_{\mathrm{an}}(c,\boldsymbol{a}) \ \cup  \\ \label{eq:intpointsy}
  \ &\qquad \cup \
\mathcal{M}_{\mathrm{an}}(c)\times_{\ev_{\partial}}\mathcal{M}(L) \cup \ \mathcal{M}(c,\xi) \ \cup \
\mathcal{M}(c,\eta).
\end{align}
See Figure \ref{bdryL1_fig}. For almost complex structures as in Lemma \ref{LK MK augment}, the set $\mathcal{M}(L)$ is empty and hence $\mathcal{M}_{\mathrm{an}}(c)\times_{\ev_{\partial}}\mathcal{M}(L)$ is also empty.
\end{lemma}

\begin{proof}
Note that every curve in $\mathcal{M}(c)$ has injective points near the positive puncture, and hence standard arguments show that perturbing the almost complex structure suffices to achieve transversality. To understand the boundary, we consider degenerations of the domain where the limit is either a nodal curve with one disk and one annulus component or the conformal ratio of the annulus goes to infinity. The nodal limits are curves in 
\[ 
\mathcal{M}(c)\times_{\ev_{\partial}}\mathcal{M}_{\mathrm{an}}(\R^3,L) \ \cup \
\mathcal{M}_{\mathrm{an}}(c,\boldsymbol{a}) \ \cup \ 
\mathcal{M}_{\mathrm{an}}(c)\times_{\ev_{\partial}}\mathcal{M}(L) .
\]   
To parameterize a neighborhood of the boundary of $\mathcal{M}_{\mathrm{an}}(c)$, we use Lemma \ref{l:gluingmultipleannuli} for $\mathcal{M}(c)\times_{\ev_{\partial}}\mathcal{M}_{\mathrm{an}}(\R^3,L)$, for $\mathcal{M}_{\mathrm{an}}(c,\boldsymbol{a})$ we use SFT-gluing, see e.g. \cite[Section 10]{CELN} 
and for $\mathcal{M}_{\mathrm{an}}(c)\times_{\ev_{\partial}}\mathcal{M}(L)$ we use standard boundary gluing, see \cite[Section 6]{EkholmSmithjams}.

Degeneration of the conformal ratio of the annulus corresponds to
\[ 
\mathcal{M}(c;L) \ \cup \ 
\mathcal{M}(c,\xi) \ \cup \
\mathcal{M}(c,\eta),
\]
where a neighborhood of the boundary is parameterized via  Lemma \ref{l:standardnode} for $\mathcal{M}(c,L)$ (which corresponds to when the modulus of the annulus goes to infinity), 
and gluing at a Lagrangian intersection point for $\mathcal{M}(c,\xi)\cup \mathcal{M}(c,\eta)$ (corresponding to when the modulus goes to zero), see \cite[Appendix A.4--6]{EkholmSmithmathann}. Note here that the orientation of the moduli space of disks with corners is induced by the oriented capping disks that are required to glue to the a closed disk with the Fukaya orientation, see Section \ref{ssec:orientdisks}. Then, gluing the orientation of the disks and those of the caps gives the Fukaya orientation of the annulus together with the orientation of the gluing of the capping disks. It follows that the disk has the boundnary orientation of the space of annuli.
\end{proof}

\subsection{Curve counts at infinity}
In this section, we construct 1-dimensional moduli spaces such that some of their boundary components coincide with certain boundary components of $\mathcal{M}_{\mathrm{an}}(c)$ in \eqref{eq:intpointsy}, namely $\mathcal{M}(c)\times_{\ev_{\partial}}\mathcal{M}_{\mathrm{an}}(\R^3,L)$ and $\mathcal{M}(c;L)$. In order to treat $L=L_{K}^{\delta}$ and $L=M_{K}^{\delta}$ simultaneously, we consider curves $\alpha$ and $\beta$ in $\Lambda_{K}$ whose homology classes generate $H_{1}(\Lambda_K,\Z)$, so that $[\alpha]$ generates $H_{1}(L,\Z)$ and $\beta$ is null-homologous in $L$ (so, $\alpha=x$ and $\beta=p$ if $L=L_{K}^{\delta}$, and vice versa if $L=M_{K}^{\delta}$.)   
%

We start with $\mathcal{M}(c)\times_{\ev_{\partial}}\mathcal{M}_{\mathrm{an}}(\R^3,L)$, and use a construction similar to one used in the definition of the disk potential in \cite{AENV}. Each holomorphic annulus $u\colon A_R \to T^*\R^3$ in $\M_{\mathrm an}(\R^3,L)$ specifies, by restricting to the boundary component $\partial_{0} A_{R}$ that maps to $L$, an element in $H_1(L;\Z)$. An area argument shows that this element is of the form $k_u[\alpha]$, where $k_u\in \Z_+$ and $\alpha$ is the curve specified in the previous paragraph.
Let $\sigma_u$ be a smooth Borel--Moore 2-chain with integer coefficients in $L$, such that $\sigma_u$ gives a homology 
between the loop $u(\partial_{0} A_{R})$ and $k_u \alpha$ in $\Lambda_K$ at infinity. We call $\sigma_u$ a \emph{bounding chain} for $u$, as in \cite[Section 6.3]{AENV}. 

\begin{figure}
	\begin{center}
		\def\svgwidth{.9\textwidth}
		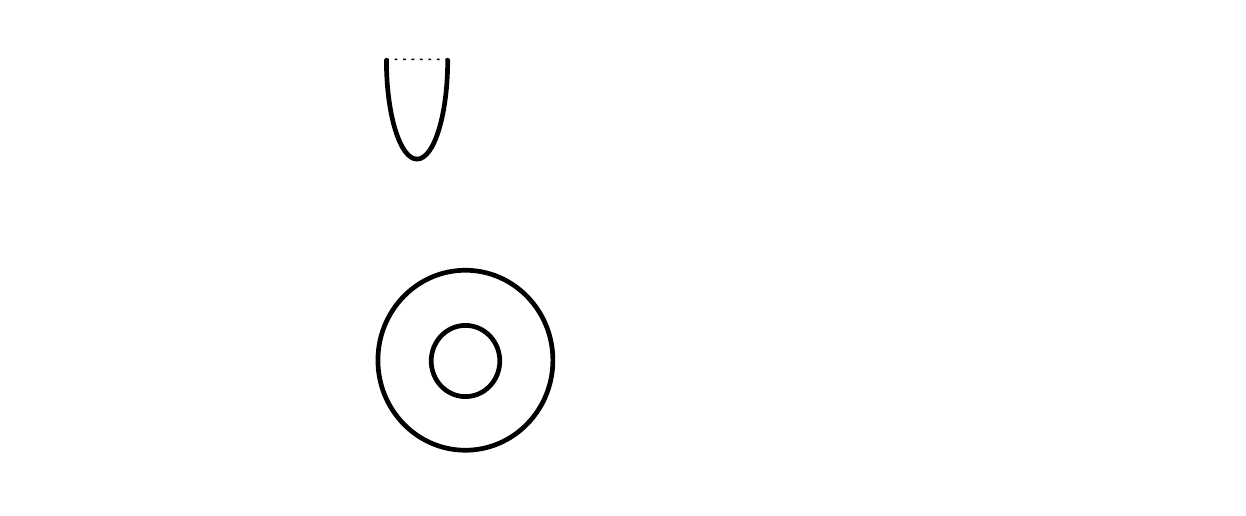
	\end{center}
	\caption{The boundary of $\MMh$}
	\label{bdryL2_fig}
\end{figure}


Let
\[ 
\sigma=\left\{(u,\sigma_{u}) | u\in\mathcal{M}_{\mathrm{an}}(L) \right\}
\] 
denote the set of all the pairs consisting of an annulus in $\mathcal{M}_{\mathrm{an}}(L)$ with its bounding chain. Recall that for every fixed $E>0$, there are finitely many annuli of energy at most $E$. 

Define $\mathcal{M}(c,\sigma)$ as the 1-dimensional moduli space of pairs $(u_1,u_2)$, where $u_1\in \mathcal{M}_{\mathrm{an}}(L)$ is a holomorphic annulus and $u_2$ is a holomorphic disk in $T^*\R^3$ with one positive boundary puncture asymptotic to $c$ and a boundary point intersecting $\sigma_{u_1}$ transversely. See Figure \ref{bdryL2_fig}. 

We next describe the boundary of $\mathcal{M}(c,\sigma)$ in terms of other moduli spaces as follows.
\begin{itemize}
\item Let $\mathcal{M}^{\infty}(c,\sigma)$ denote the moduli space of pairs $(u_1,u_2)$ as follows. Here $u_1\in \mathcal{M}_{\mathrm{an}}(L)$ is a holomorphic annulus and $u_2$ is a holomorphic disk in $\R\times ST^*\R^3$ with one positive boundary puncture asymptotic to $c$ and a boundary point intersecting $k_{u_2}\alpha$ transversely in $\Lambda_K$. Note that the orientation of the boundary of the annulus induces an orientation of the bounding chain, which together with the orientations of the moduli spaces and of the boundary of the disk gives an orientation of $\mathcal{M}^{\infty}(c,\sigma)$.

The schematic representation of $\MMj$ in Figure \ref{bdryL2_fig} should be thought of as a simplification, since $u_2$ is allowed to have negative punctures asymptotic to Reeb chords, and each of these Reeb chords comes with a punctured holomorphic disk with boundary in $L$ and a puncture asymptotic to the Reeb chord. A more accurate illustration of the terms at infinity is given in Figure \ref{bdryL2infty_fig}. By Lemma 
\ref{LK MK augment} $(ii)$, we can use $\Ld$ to compute $\epsilon_{L_K}$. Similar observations should be made about curves at infinity in the various moduli spaces considered in this section.


\item Let $\mathcal{M}(c,\boldsymbol{a},\sigma)$ denote the union of products
\[ 
\mathcal{M}(c,\boldsymbol{a},\sigma)=\bigcup_{j}\mathcal{M}(c,a_{j})\times\mathcal{M}(a_{j},\sigma).
\]
\end{itemize}

\begin{figure}
	\begin{center}
		\def\svgwidth{.7\textwidth}
		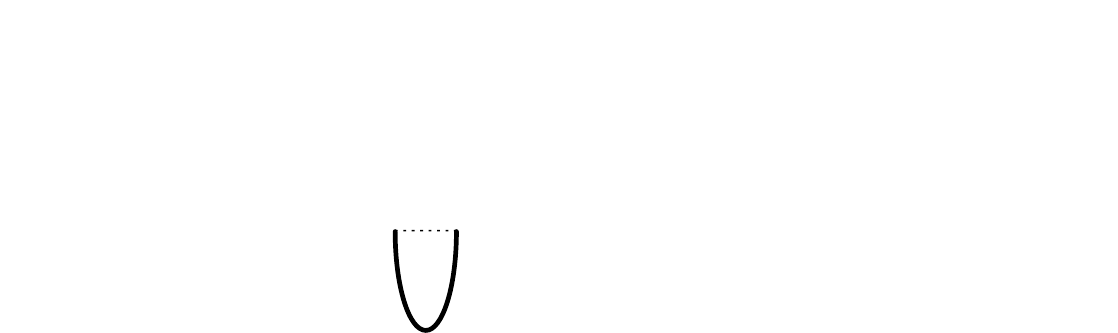
	\end{center}
	\caption{The components $u_2$ in $\MMj$}
	\label{bdryL2infty_fig}
\end{figure}

\begin{lemma}\label{l:boundingchainann}
The (oriented) boundary of the moduli space $\mathcal{M}(c,\sigma)$ is
\[ 
\mathcal{M}^{\infty}(c,\sigma) \ \cup \
\mathcal{M}(c)\times_{\ev_{\partial}}\mathcal{M}_{\mathrm{an}}(\R^3,L) \ \cup \
\mathcal{M}(c,\boldsymbol{a},\sigma).
\]
See Figure \ref{bdryL2_fig}. 
\end{lemma}

\begin{proof}
There are three sources of non-compactness: either $u_2$ breaks which corresponds to $\mathcal{M}(c,\boldsymbol{a},\sigma)$, or it moves to infinity which corresponds to $\mathcal{M}^{\infty}(c,\sigma)$, or the intersection of $u_2$ with $\sigma_{u_1}$ moves to the boundary of $u_1$ which corresponds to $\mathcal{M}(c)\times_{\ev_{\partial}}\mathcal{M}_{\mathrm{an}}(\R^3,L)$.  	
\end{proof}

We now consider the moduli space $\mathcal{M}(c,L)$ in \eqref{eq:intpointsy} from a similar point of view. Fix a smooth Borel--Moore 4-chain with integer coefficients in $V$ in $T^*\R^3$ whose boundary is $\R^3$. Explicitly, we can pick a section $s$ of the trivial bundle $ST^*\R^3\to \R^3$ (as in Section \ref{choices and coeffs}) and define 
\begin{equation} \label{4 chain}
V:= \{t s(x) \,|\, t\geq 0, x\in \R^3\} \subset T^*\R^3.
\end{equation}
We then identify the boundary at infinity of $V$, which we denote by $\partial^\infty V$, with the section $s$. 

We also need to take into account the fact that $V\cap L\ne \varnothing$. For generic $s$, $V\cap L$ is an oriented 1-manifold $\gamma$ with boundary at infinity $\partial^{\infty}\gamma=s \cap \Lambda_{K}$, which is an oriented 0-manifold. We pick a $1$-manifold $\tau\subset \Lambda_{K}$ such that $\partial\tau=\partial^{\infty}\gamma$. Such $\tau$ was introduced in the discussion that precedes Lemma \ref{compute Q}. That discussion addresses the homological ambiguity in the choice of $\tau$, and implies that we can choose $\tau$ so that the 1-cycle $\gamma + \tau$ is null-homologous in $L$. We make this assumption on $\tau$. Finally, pick a bounding chain $\sigma_\tau$ for $\gamma + \tau$, in the same manner as was done above for the curves in $L$ that are boundaries of holomorphic annuli $u\in \M_{\mathrm an}(\R^3,L)$. See Figure \ref{chainQL_fig}. 

\begin{figure}
	\begin{center}
		\def\svgwidth{.35\textwidth}
		\hspace{2cm}
\begingroup%
  \makeatletter%
  \providecommand\color[2][]{%
    \errmessage{(Inkscape) Color is used for the text in Inkscape, but the package 'color.sty' is not loaded}%
    \renewcommand\color[2][]{}%
  }%
  \providecommand\transparent[1]{%
    \errmessage{(Inkscape) Transparency is used (non-zero) for the text in Inkscape, but the package 'transparent.sty' is not loaded}%
    \renewcommand\transparent[1]{}%
  }%
  \providecommand\rotatebox[2]{#2}%
  \newcommand*\fsize{\dimexpr\f@size pt\relax}%
  \newcommand*\lineheight[1]{\fontsize{\fsize}{#1\fsize}\selectfont}%
  \ifx\svgwidth\undefined%
    \setlength{\unitlength}{190.43943558bp}%
    \ifx\svgscale\undefined%
      \relax%
    \else%
      \setlength{\unitlength}{\unitlength * \real{\svgscale}}%
    \fi%
  \else%
    \setlength{\unitlength}{\svgwidth}%
  \fi%
  \global\let\svgwidth\undefined%
  \global\let\svgscale\undefined%
  \makeatother%
  \begin{picture}(1,0.70055647)%
    \lineheight{1}%
    \setlength\tabcolsep{0pt}%
    \put(0,0){\includegraphics[width=\unitlength,page=1]{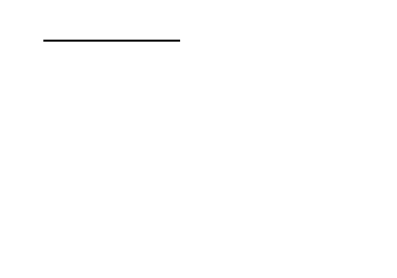}}%
    \put(0.188215,0.62843832){\color[rgb]{0,0,0}\makebox(0,0)[lt]{\lineheight{0}\smash{\begin{tabular}[t]{l}$\tau\subset \LK$\end{tabular}}}}%
    \put(0.10039401,0.01810698){\color[rgb]{0,0,0}\makebox(0,0)[lt]{\lineheight{0}\smash{\begin{tabular}[t]{l}$V\cap L=\gamma$\end{tabular}}}}%
    \put(0,0){\includegraphics[width=\unitlength,page=2]{chainQL.pdf}}%
    \put(0.67509633,0.63309405){\color[rgb]{0,0,0}\makebox(0,0)[lt]{\lineheight{0}\smash{\begin{tabular}[t]{l}$\tau$\end{tabular}}}}%
    \put(0,0){\includegraphics[width=\unitlength,page=3]{chainQL.pdf}}%
    \put(0.16768724,0.46978838){\color[rgb]{0,0,0}\makebox(0,0)[lt]{\lineheight{0}\smash{\begin{tabular}[t]{l}$\sigma_\tau\subset L$\end{tabular}}}}%
  \end{picture}%
\endgroup%

	\end{center}
	\caption{A component of $V\cap L = \gamma$, with $\tau\subset \Lambda_K$ such that $\partial \sigma_\tau = \gamma + \tau$}
	\label{chainQL_fig}
\end{figure}

\begin{remark}\label{r:nonuniquenessoftau}
Even under the assumption in the previous paragraph, the choice of $\tau$ is not homologically unique. More precisely, we can modify $\tau$ by adding a multiple $k\cdot\beta$ of the curve $\beta$ in $\Lambda_{K}$, which is null-homologous in $L$. 
\end{remark} 

\begin{figure}
	\begin{center}
		\def\svgwidth{.9\textwidth}
		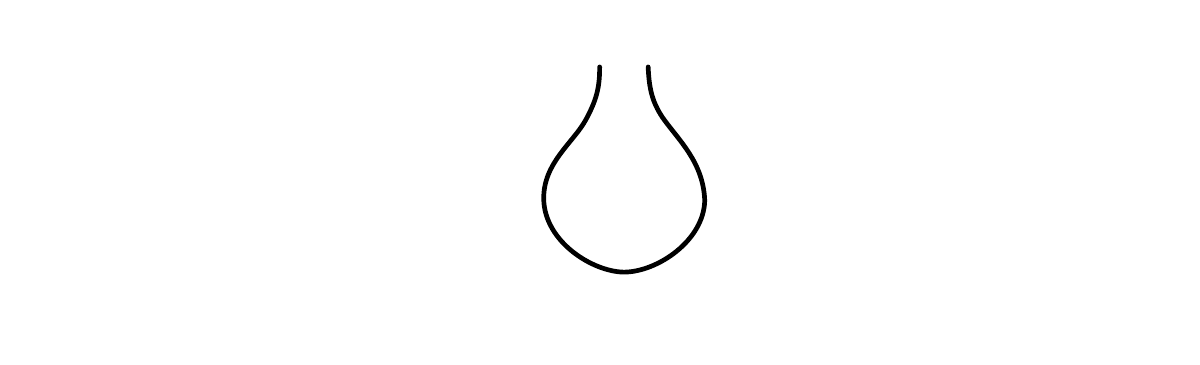
	\end{center}
	\caption{The boundary of $\MMl$}
	\label{bdryL3_fig}
\end{figure}

We now introduce new moduli spaces. The following are 1-dimensional:
\begin{itemize}
\item $\mathcal{M}(c,V)$ is the moduli space of holomorphic disks with boundary on $L$, with one positive boundary puncture asymptotic to $c$ and one interior marked point mapping to $V$, see Figure \ref{bdryL3_fig}. Note that the orientation of $\R^3$ gives an orientation on $V$, which together with orientations of the moduli space and the source disk gives and orientation on $\mathcal{M}(c,V)$. 
\item $\mathcal{M}(c,\sigma_\tau)$ is the moduli space of holomorphic disks with boundary on $L$, with one positive boundary puncture asymptotic to $c$ and one boundary marked point mapping to the bounding chain $\sigma_\tau$, see Figure \ref{bdryL4_fig}. 
\end{itemize}

\begin{figure}
	\begin{center}
		\def\svgwidth{.8\textwidth}
		\hspace{2cm}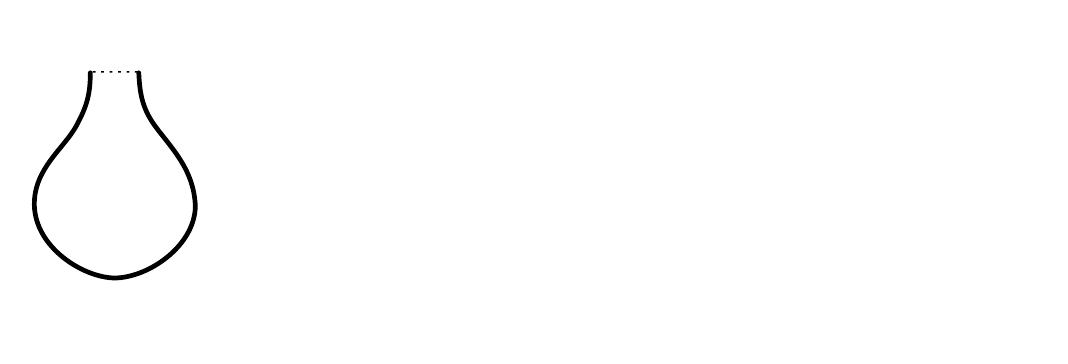
	\end{center}
	\caption{The boundary of $\MMq$}
	\label{bdryL4_fig}
\end{figure}

The boundaries of these 1-dimensional moduli spaces will contain the following 0-dimensional moduli spaces:
\begin{itemize}
\item The moduli space $\mathcal{M}^{\infty}(c,\partial^{\infty} V)$ of holomorphic disks in $\R\times ST^*\R^3$,
with a positive boundary puncture at $c$ and an interior marked point mapping to $\partial^{\infty} V$. This space is oriented in the same way as $\mathcal{M}(c,V)$.
\item The moduli space $\mathcal{M}(c,\gamma)$ of holomorphic disks with a positive boundary puncture at $c$ and a boundary marked point mapping to $\gamma$. This space is oriented in the same way as $\mathcal{M}(c,\sigma_\tau)$.
\item The moduli space $\mathcal{M}(c,\boldsymbol{a},V)$: the union of products of moduli spaces $\mathcal{M}(c,a_{j})\times\mathcal{M}(a_{j},V)$, where $\mathcal{M}(a_{j},V)$ is the moduli space of disks with boundary on $L$, one positive boundary puncture asymptotic to $a_{j}$ and one interior marked point mapping to $V$. The space is oriented in the same way as $\mathcal{M}(c,V)$.
\item The moduli space $\mathcal{M}^{\infty}(c,\tau)$ of holomorphic disks in $\R\times ST^*\R^3$, with one positive boundary puncture asymptotic to $c$ and one boundary marked point mapping to $\tau$. This space is oriented in the same way as $\mathcal{M}(c,\sigma_\tau)$.
\item The moduli space $\mathcal{M}(c,\boldsymbol{a},\sigma_\tau)$: the union of products of moduli spaces $\mathcal{M}(c,a_{j})\times\mathcal{M}(a_{j},\sigma_\tau)$, where $\mathcal{M}(a_{j},\sigma_\tau)$ is the moduli space of disks with boundary on $L$, one positive boundary puncture asymptotic to $a_{j}$ and one boundary marked point mapping to $\sigma_\tau$. This space is oriented in the same way as $\mathcal{M}(c,\sigma_\tau)$.
\end{itemize}

\begin{lemma}\label{l:boundary4-chaininsertion}
The 1-dimensional moduli space $\mathcal{M}(c,V)$ has a natural compactification with (oriented) boundary
\[ 
\mathcal{M}(c,L) \ \cup \
\mathcal{M}^{\infty}(c, \partial^{\infty}V) \ \cup \
\mathcal{M}(c,\gamma) \ \cup \
\mathcal{M}(c,\boldsymbol{a},V).
\]
See Figure \ref{bdryL3_fig}. 
\end{lemma}

\begin{proof}
The boundary of $\M(c,V)$ is given by three degenerations. First, the interior marked point can move to the boundary of $V$, which gives $\mathcal{M}(c,L)$. Second, the holomorphic curves can move to the boundary in the moduli space of maps. This corresponds here to the curve moving to infinity which gives $\mathcal{M}^{\infty}(c,\partial^{\infty}V)$, or to SFT-breaking which gives $\mathcal{M}(c,\boldsymbol{a},V)$. Third, the interior marked point can move to the boundary of the domain, which corresponds to $\mathcal{M}(c,\gamma)$.  
\end{proof}

Similarly, we have the following result.

\begin{lemma}\label{l:4chaincapL}
	The 1-dimensional moduli space $\mathcal{M}(c,\sigma_\tau)$ has a natural compactification with (oriented) boundary
	\[ 
	\mathcal{M}(c,\gamma) \ \cup \
	\mathcal{M}^{\infty}(c,\tau) \ \cup \
	\mathcal{M}(c,\boldsymbol{a},\sigma_\tau).
	\]
See Figure \ref{bdryL4_fig}. 
\end{lemma}

\begin{proof}
	The boundary of $\M(c,\tau)$ is given by two degenerations. First, the marked point can move to the boundary of $\sigma_\tau$, which gives $\mathcal{M}(c,\gamma)$. Second, the holomorphic curves can move to the boundary in the moduli space of maps. This corresponds here to the curve moving to infinity which gives $\mathcal{M}^{\infty}(c,\tau)$, or to SFT-breaking which gives $\mathcal{M}(c,\boldsymbol{a},\sigma_\tau)$. 
\end{proof}

\subsection{Counts of Floer strips}
In the previous section we related the first two boundary terms in \eqref{eq:intpointsy} with curve counts at infinity. Here we will carry out a similar analysis of the last two terms in \eqref{eq:intpointsy}.  

We now take $L=M_{K}^{\delta}$ and as above we denote by $\xi=\{\xi_{i}\}$ the intersection points in $L\cap \R^{3}$ of index 1 and by $\eta=\{\eta_i\}$ those of index 2. We consider the following moduli spaces, illustrated in Figure \ref{bdry6_fig}.
\begin{itemize}
\item The 1-dimensional moduli spaces $\mathcal{M}(c,\xi_{i},\eta_{j})$ of holomorphic disks in $T^*\R^3$ with one positive boundary puncture asymptotic to $c$, two boundary components mapping to $M_{K}^{\delta}$, one boundary component mapping to $\R^3$, and two boundary punctures asymptotic to Lagrangian intersection points $\xi_{i}$ and $\eta_{j}$. 
\item The 0-dimensional moduli spaces $\mathcal{M}(c,\xi_{i},\xi_{j})$ of holomorphic disks in $T^*\R^3$ with one positive boundary puncture asymptotic to $c$, two boundary components mapping to $M_{K}^{\delta}$, one boundary component mapping to $\R^3$, and two boundary punctures asymptotic to Lagrangian intersection points $\xi_{i}$ and $\xi_{j}$. 
\item The 0-dimensional moduli spaces $\mathcal{M}(c,\eta_{i},\eta_{j})$ of holomorphic disks in $T^*\R^3$ with one positive boundary puncture asymptotic to $c$, two boundary components mapping to $M_{K}^{\delta}$, one boundary component mapping to $\R^3$, and two boundary punctures asymptotic to Lagrangian intersection points $\eta_{i}$ and $\eta_{j}$. 
\item The union of products of moduli spaces 
\[ 
\mathcal{M}(c,\boldsymbol{a},\xi_{i},\eta_{j})=\bigcup_{a_{k}\in\boldsymbol{a}}
\mathcal{M}(c,a_{k})\times\mathcal{M}(a_{k},\xi_{i},\eta_{j}),
\]
where the union is over Reeb chords $a_k$ of index 0. 
\item The 0-dimensional moduli space $\mathcal{M}(\xi_{i},\eta_{j})$ of holomorphic strips with one boundary component on $M_{K}^{\delta}$ and the other on $\R^3$, one boundary puncture asymptotic to $\xi_{i}$ and the other asymptotic to $\eta_{j}$. 
\end{itemize}
Here, all the moduli spaces are oriented as described in Section \ref{ssec:orientdisks}.

\begin{figure}
  \begin{center}
    \def\svgwidth{1.05\textwidth}
    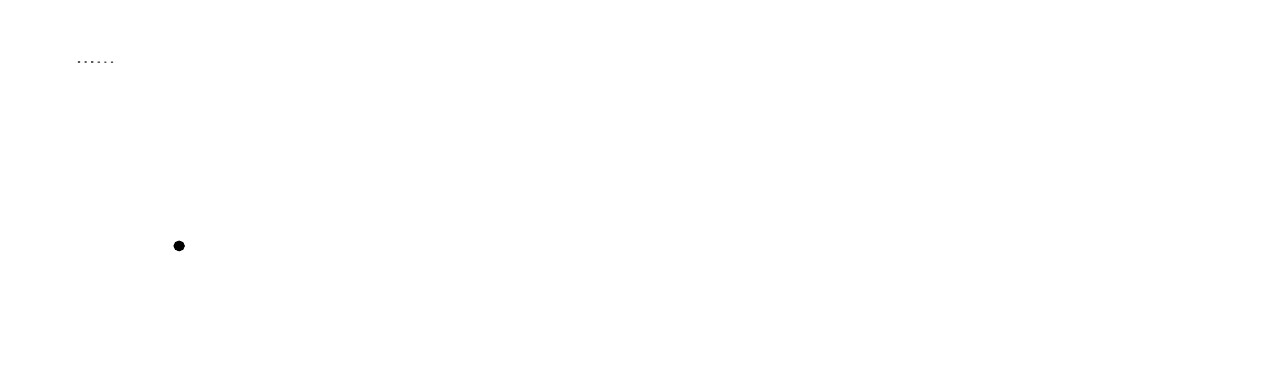
  \end{center}
  \caption{The boundary of $\MMu$}
  \label{bdry6_fig}
\end{figure}

Recall that $\mathcal M(c)$ is the moduli space of holomorphic disks with boundary on $\Md$ and one positive boundary puncture asymptotic to $c$. 

\begin{lemma}\label{l:boundaryoftriangles}
The (oriented) boundary of the 1-dimensional moduli space $\mathcal{M}(c,\xi_{i},\eta_{j})$ consists of the moduli spaces
\begin{align*} 
&\bigcup_{k} \mathcal{M}(\xi_{i},\eta_{k})\times \mathcal{M}(c,\eta_{k},\eta_{j}) \ \cup \\
&\bigcup_{k} \mathcal{M}(c,\xi_{i},\xi_{k})\times\mathcal{M}(\xi_{k},\eta_{j}) \ \cup \
\mathcal{M}(c)\times_{\ev_{\partial}} \mathcal{M}(\xi_{i},\eta_{j}) \ \cup \
\mathcal{M}(c,\boldsymbol{a},\xi_{i},\eta_{j}).
\end{align*}
Here, $\ev_{\partial}$ are evaluation maps at marked points in boundary components mapping to $\Md$. See Figure \ref{bdry6_fig}. 
\end{lemma}

\begin{proof}
The contributions to the boundary of $\mathcal{M}(c,\xi_{i},\eta_{j})$ correspond to three types of breaking. Breaking at an intersection point gives  
$\bigcup_{k} \mathcal{M}(\xi_{i},\eta_{k})\times \mathcal{M}(c,\eta_{k},\eta_{j})$ and
$\bigcup_{k} \mathcal{M}(c,\xi_{i},\xi_{k})\times\mathcal{M}(\xi_{k},\eta_{j})$. Boundary breaking gives $\mathcal{M}(c)\times_{\ev_{\partial}} \mathcal{M}(\xi_{i},\eta_{j})$, where the orientation is induced from the orientations of the moduli spaces, the boundaries of the domains and the orientation of $M_K^{\delta}$. Finally, SFT-breaking at a Reeb chord gives $\mathcal{M}(c,\boldsymbol{a},\xi_{i},\eta_{j})$. 
\end{proof}


Our next goal is to relate the moduli space $\mathcal{M}(c)\times_{\ev_{\partial}} \mathcal{M}(\xi_{i},\eta_{j})$ with curve counts at infinity.  Lemma \ref{l:brokentriangles} below will be an analogue of Lemma \ref{l:boundingchainann}, replacing the annuli with strips.
Pick smooth arcs $\gamma_{ij}$ in $\Md$ connecting the intersection points $\xi_{i}$ and $\eta_{j}$ to some fixed basepoint in $\Md$. Let $U$ be the closure of a small neighborhood of all such arcs which is topologically a disk. In order to remove geometrically irrelevant evaluation conditions on moduli spaces, we will count intersections in the quotient space $M_K^{\delta}/U$, which is naturally homotopy equivalent to $M_K^{\delta}$, compare \cite[Section 3.2]{EkholmLekili2}. The advantage of the quotient space is that boundary arcs of disks between $\xi_j$ and $\eta_j$ map to closed curves. Then for every $u\in\M(\xi_{i},\eta_{j})$, pick a bounding chain $\tilde\sigma_{u}$ in $\Md/U$ for the loop corresponding to its boundary component in $M_K^{\delta}$. 

%
%

We write $\mathcal{M}(c,\xi_i,\eta_j,\tilde\sigma)$ for the 1-dimensional moduli space of pairs $(u_1,u_2)$, where $u_1$ is a holomorphic strip between $\xi_i$ and $\eta_j$, and $u_2$ is a disk with boundary on $\Md$, one positive boundary puncture asymptotic to $c$ and a boundary marked point intersecting $\tilde\sigma_{u}$ transversely. See Figure \ref{bdry7_fig}. 

To describe the boundary of $\mathcal{M}(c,\xi_i,\eta_j,\tilde\sigma)$, we need the following 
0-dimen-sional moduli spaces:
\begin{itemize}
\item $\mathcal{M}^{\infty}(c,\xi_i,\eta_j,\tilde\sigma)$ is a moduli space of pairs $(u_1,u_2)$, where $u_1$ is a holomorphic strip between $\xi_i$ and $\eta_j$, and $u_2$ is a disk in $\R\times ST^*\R^3$ with boundary on $\R\times \Lambda_K$, one positive boundary puncture asymptotic to $c$ and one boundary marked point intersecting $\partial^\infty\tilde\sigma_{u}$ transversely in $\Lambda_K$. This space is oriented from by the orientation of the moduli spaces, the boundary of the disk, and the intersection in $\Lambda_K$. 
\item $\mathcal{M}(c,\boldsymbol{a},\xi_i,\eta_j,\tilde\sigma)$ is the union $\bigcup_{k}\mathcal{M}(c,a_{k})\times\mathcal{M}(a_{k},\xi_i,\eta_j,\tilde\sigma)$.
\end{itemize}

%
%

\begin{figure}
  \begin{center}
    \def\svgwidth{0.8\textwidth}
    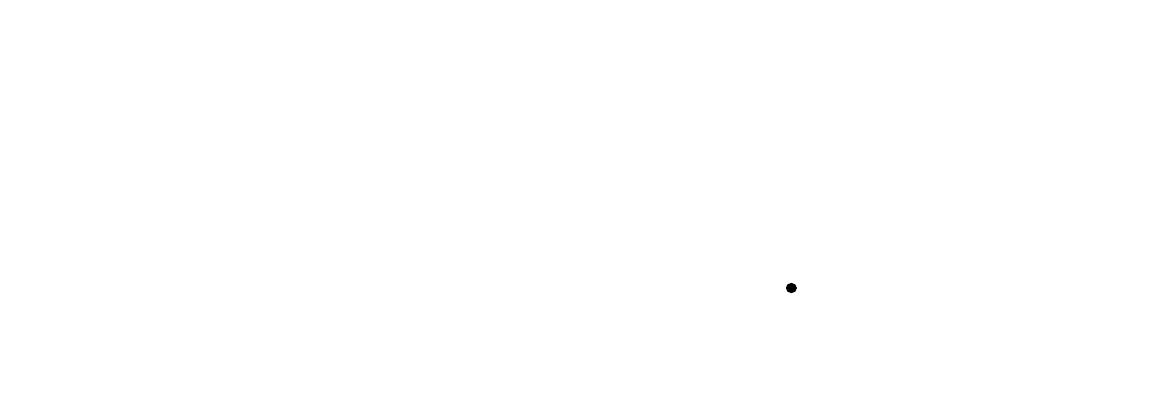
  \end{center}
  \caption{The boundary of $\MMz$}
  \label{bdry7_fig}
\end{figure}

\begin{lemma}\label{l:brokentriangles}
	The (oriented) boundary of the moduli space $\mathcal{M}(c,\xi_i,\eta_j,\tilde\sigma)$ is
	\[ 
	\mathcal{M}^{\infty}(c,\xi_i,\eta_j,\tilde\sigma) \ \cup \
	\mathcal{M}(c)\times_{\ev_{\partial}} \mathcal{M}(\xi_{i},\eta_{j}) \ \cup \
	\mathcal{M}(c,\boldsymbol{a},\xi_i,\eta_j,\tilde\sigma).
	\]
See Figure \ref{bdry7_fig}. 
\end{lemma} 

\begin{proof}
	The boundary component $\mathcal{M}^{\infty}(c,\xi_i,\eta_j,\tilde\sigma)$ corresponds to when the disk $u_2$ moves to infinity, and $\mathcal{M}(c,\boldsymbol{a},\xi_i,\eta_j,\tilde\sigma)$ corresponds to SFT-breaking of $u_2$. The remaining terms correspond to the boundary marked point in $u_2$ moving to the boundary of $\tilde\sigma_{u_1}$.
\end{proof}

\section{From knot contact homology to the Alexander polynomial} \label{sec:LCH computations}

In this section we prove Theorems \ref{t:linearizedhomology} and \ref{T:Alex Aug}. The argument will be based on combining the moduli space analysis of Section \ref{sec:main proof} with computations of the linearized contact homologies associated to specific augmentations of $\Lambda_K$.

\subsection{Linearized contact homology and the augmentation variety} \label{LCH VK}
We will now establish some basic properties of augmentation varieties. We first state Propositions \ref{full V_K 2d} and \ref{KCH1 generator} and discuss some immediate consequences. Later we will turn to the proofs of these propositions. Recall that associated to the conormal Lagrangian $L_K$ there is an augmentation of $\mathcal A_K$ with $(\lambda,\mu,Q) = (1,1,1)$. Denote by $\boldsymbol{\epsilon_0}$ the corresponding point in $\tilde V_K \subset (\C^*)^3 \times \C^{\boldsymbol{a}}$. 
\begin{proposition} \label{full V_K 2d}
Let $K\subset \R^{3}$ be any knot.
\begin{itemize}
\item There is an open neighborhood of $\boldsymbol{\epsilon_0} \in  (\C^*)^3 \times \C^{\boldsymbol{a}}$ with respect to the {\em Zariski topology}, in which the full augmentation variety $\tilde V_K$ is a smooth two-dimensional complex algebraic variety. 
\item There is an open neighborhood $U$ of $\boldsymbol{\epsilon_0} \in (\C^*)^3 \times \C^{\boldsymbol{a}}$ with respect to the {\em standard metric topology}, in which $\tilde V_K$ admits a holomorphic parametrization by the coordinates $(\lambda,\mu)$.
\end{itemize}
\end{proposition}

The following are useful consequences of this result. 


\begin{corollary} \label{Aug exists}
For every $K$, the augmentation variety $V_K$ has dimension 2 or 3. The ideal $I(V_K)$ is principal, hence the augmentation polynomial $\Aug_K$ exists (but might be 0).
\end{corollary}
\begin{proof}
Since $\tilde V_K$ is two-dimensional and parametrized by $(\lambda,\mu)$ near $\boldsymbol{\epsilon_0}$, the maximal-dimensional components of $\mathrm{Cl}\left(\pi(\tilde V_K)\right) \subset (\C^*)^3$ are at least 2-dimensional. Hence, $\dim_\C(V_K)\geq 2$. If this dimension is 3, then $V_K = (\C^*)^3$ and $\Aug_K = 0$. If the dimension is 2, then $V_K$ can be described as the zero set of some polynomial. 
\end{proof}

\begin{remark}
Corollary \ref{Aug exists} could also be deduced from results on the cord algebra in \cite{CELN} in combination with \cite[Section 6.11]{AENV}, see also \cite{Cornwell}.
\end{remark}

\begin{remark} \label{NgConjecture}
Corollary \ref{Aug exists} is compatible with \cite[Conjecture 5.3]{NgTopological}, which predicts that $\Aug_K = 0$ should not be possible. In Section \ref{elimination} below, we observe that one condition that implies $\Aug_K \neq 0$ is if $I_K=R\cap \tilde I_K$ is not the zero ideal in $R$ (recall Section \ref{sec:aug var poly} for the relevant notation).
\end{remark}


\begin{corollary}
For every $K$, $V_K$ contains the points $(\lambda,1,1)$ and $(1,\mu,1)$, for all $\lambda \in \C^*$ and all $\mu \in \C^*$. 
\end{corollary}
\begin{proof}
Let $U$ be the open set in Proposition \ref{full V_K 2d}. As we saw in Example \ref{LK and MK}, the set $\pi(U)$ contains open neighborhoods $U_\lambda$ and $U_\mu$ (in the standard metric topology) of the point $(1,1,1)$ in the lines $\{(\lambda,1,1)\}\subset (\C^*)^3$ and $\{(1,\mu,1)\}\subset (\C^*)^3$, respectively. On the other hand, it follows from the proof of Corollary \ref{Aug exists} that $\pi(U)$ is contained in $V_K$. Since $V_K$ is an affine set containing the non-empty open set $U_\lambda$ of the line $\{(\lambda,1,1)\}$, it contains the whole line. Similarly, it contains the whole line $\{(1,\mu,1)\}$.
\end{proof}

The next result concerns the linearized knot contact homology of augmentations near $\boldsymbol{\epsilon_0}$. We point out that the symbol $\partial$ is used with two different meanings in the statement: in the definition of $F$, the term $\partial (y(\lambda,\mu))$ denotes the contact homology differential of $y(\lambda,\mu)$; afterwards, $\partial_Q F$ denotes the partial derivative of the function $F(\lambda, \mu, Q)$ with respect to the variable $Q$.  


\begin{proposition}\label{prp:about F} 
${\quad}$
\begin{enumerate}
\item There is an open neighborhood $U'$ of $\boldsymbol{\epsilon_0} \in (\C^*)^3 \times \C^{\boldsymbol{a}}$ with respect to the standard metric topology such that, for every augmentation $\epsilon \in U'$, we have 
$$
KCH_k^{\epsilon}(K;\C) \cong 
\begin{cases}
\C & \text{if } k=1 \text{ or } 2 \\
0 & \text{otherwise}
\end{cases}.
$$
\item Let $U$ be as in Proposition \ref{full V_K 2d}, and denote the parametrization of $\tilde V_{K}$ in $U\cap U'$ by $\epsilon_{\lambda,\mu}$. Then there is a linear combination $y(\lambda,\mu)$ of Reeb chords, with coefficients analytic in $\lambda$ and $\mu$, representing a generator of $KCH_1^{\epsilon_{\lambda,\mu}}(K; \C)$ such that the following holds. Write 
  $$
  F(\lambda, \mu, Q) \coloneq \partial (y(\lambda,\mu))|_{\boldsymbol{a} \mapsto \epsilon_{\lambda,\mu}(\boldsymbol{a})},
  $$
  where $\partial$ denotes the differential in $\mathcal A_K$ and $\boldsymbol{a}=(a_{1},\dots,a_{n})$ denotes the ordered set of all index zero Reeb chords. Then, 
  \begin{equation} \label{partial Q F neq 0}
   (\partial_Q F)|_{\lambda = \mu = Q = 1} \neq 0  
  \end{equation}
and
  \begin{equation} \label{partial lambda F neq 0}
   (\partial_\lambda F)_{\lambda = Q = 1} \text{ is non-constant, as a function of }\mu. 
  \end{equation}
Since the function $(\partial_\lambda F)_{\lambda = Q = 1}$ is holomorphic, it has isolated zeros.
 \end{enumerate} \label{KCH1 generator}
\end{proposition}

Propositions \ref{full V_K 2d} and \ref{KCH1 generator} imply Theorem \ref{t:linearizedhomology}. Note that the restriction of $d_{\boldsymbol{\epsilon_0}}\pi$ to $\tilde V_K$ sends vectors tangent to the curve of augmentations associated to the filling $L_K$ to $\Span(\partial_x)$. Similarly, $d_{\boldsymbol{\epsilon_0}}\pi$ sends vectors tangent to the curve of augmentations associated to the filling $M_K$ to $\Span(\partial_p)$.

\begin{remark} \label{NgAlexander}
In \cite[Proposition 4.4]{NgFramed}, it is shown that there is an isomorphism of $\Q[\mu^{\pm 1}]$-modules
\begin{align*}
 KC&H^{\epsilon_{M_K}}_1(K; \Q) \cong\\ 
& (H_1(\tilde M_K;\Q) \oplus \Q[\mu^{\pm 1}]) \otimes_{\Q[\mu^{\pm 1}]} (H_1(\tilde M_K;\Q) \oplus \Q[\mu^{\pm 1}]) \oplus (\Q[\mu^{\pm 1}])^m, 
\end{align*}
for some $m\geq 0$, where $\tilde M_K$ is the universal Abelian cover of the knot complement. This already implies that the Alexander polynomial can be obtained from linearized knot contact homology. Proposition \ref{KCH1 generator}(1) implies that $m=0$ for every knot $K$. 
\end{remark}

%
%

We will now begin a discussion that will lead up the the proofs of Propositions \ref{full V_K 2d} and \ref{KCH1 generator}.

\subsection{A 1-parameter family of generators} \label{sssec:1-parameter generators}
We begin by recalling some results from \cite{EkholmNgHigherGenus}. Consider the 1-parameter family of augmentations $\epsilon_{L_K}$ induced by $L_{K}$. This family is parametrized by $\lambda\in\C^{\ast}$, which can be thought of as the monodromy around $K$ of a flat connection on a trivial complex line bundle over $L_{K}$ (or over its deformation retract $K$).  
Let $KCC_*^{\epsilon_{L_K,\lambda}}(K;\C)$ denote the chain complex of knot contact homology, linearized with respect to this augmentation for a specific value of $\lambda\in\C^{\ast}$. We will obtain a topological model for $KCH_*^{\epsilon_{L_K,\lambda}}(K;\C)$. 

Let $\mathcal P_K$ denote the space of paths in $\R^3$ that start and end in $K$, and let $\mathcal P_K^0\subset \mathcal P_K$ denote the space of constant paths. There is a local coefficient system over $\mathcal P_K$ associated to the concatenation of group homomorphisms 
\begin{equation} \label{local systems}
\pi_1(\mathcal P_K) \to \pi_1(K)\times \pi_1(K) \to \pi_1(K) \stackrel{\lambda}{\to} \C^*,
\end{equation}
where $\pi_1(\mathcal P_K) = \pi_1(\mathcal P_K, p)$ for some choice of basepoint $p\in \pi_1(\mathcal P_K^0)$. The first map in \eqref{local systems} is induced by evaluation at the endpoints of a path, the second map is $(g_1,g_2)\mapsto g_1 g_2^{-1}$ in the group $\pi_1(K)$, and the last map is induced by the monodromy $\lambda$ of the flat connection on $K$. Let $C_*^{\lambda}(\mathcal P_K;\C)$ and $C_*^{\lambda}(\mathcal P_K^0;\C)$ be the corresponding singular chain complexes with local coefficients, and let $C_*^{\lambda}(\mathcal P_K,\mathcal P_K^0;\C) = C_*^{\lambda}(\mathcal P_K;\C)/ C_*^{\lambda}(\mathcal P_K^0;\C)$ be the relative chain complex. 


	Let 
	\begin{equation} \label{Theta}
		\Theta^{\lambda}\colon KCC_*^{\epsilon_{L_K,\lambda}}(K;\C) \to C_*^{\lambda}(\mathcal P_K,\mathcal P_K^0;\C),
	\end{equation}
	be the linear map defined as follows. If $c$ is a Reeb chord then let $\mathcal{N}(c,L_K\cup \R^{3})$ denote a moduli space of holomorphic maps $u\colon D\to T^{\ast} \R^{3}$, where $D$ is the unit disk with three boundary punctures at $1,i,-1$. We require $u$ to be asymptotic to $c$ at $i$, and to map the boundary segments between $i$ and $-1$ and between $1$ and $i$ to $L_{K}$. Also, $u$ maps the boundary segment between $-1$ and $1$ to $\R^{3}$, and it can be extended continuously at $1$ and $-1$. See Figure \ref{bdryN1_fig}. Restricting the maps $u$ to the boundary segment between $-1$ and $1$, $\mathcal{N}(c,L_K\cup \R^{3})$ parametrizes a family of paths in $\mathcal{P}_{K}$ and we define $\Theta^{\lambda}(c)$ as the chain parametrized by $\mathcal{N}(c,L_K\cup \R^{3})$. The following result appears as \cite[Lemma 4.5]{EkholmNgHigherGenus} (see also \cite[Theorem C]{CiliebakLatschevStringTopology} for a closed-string version of this result). We reproduce the argument for the reader's convenience.
\begin{lemma} \label{lemma: Theta isomorphism}
	 The map $\Theta^{\lambda}$ is a chain map and a quasi-isomorphism. 
\end{lemma}

\begin{proof}
	To see that $\Theta^{\lambda}$ is a chain map, we consider the codimension 1 strata of the boundary of $\mathcal{N}(c,L_{K}\cup\R^{3})$, as illustrated in Figure \ref{bdryN1_fig}. There are two such strata, corresponding to:
\begin{itemize}
\item breaking at a Reeb chord (see $\NNb$ in the figure);
\item the boundary arc in $\R^{3}$ shrinking to a constant path (see $\NNc$ in the figure). 
\end{itemize}
Gluing and compactness results presenting this as the codimension 1 boundary follow from \cite[Section 10.1]{ECL}. 

\begin{figure}
	\begin{center}
		\def\svgwidth{.6\textwidth}
		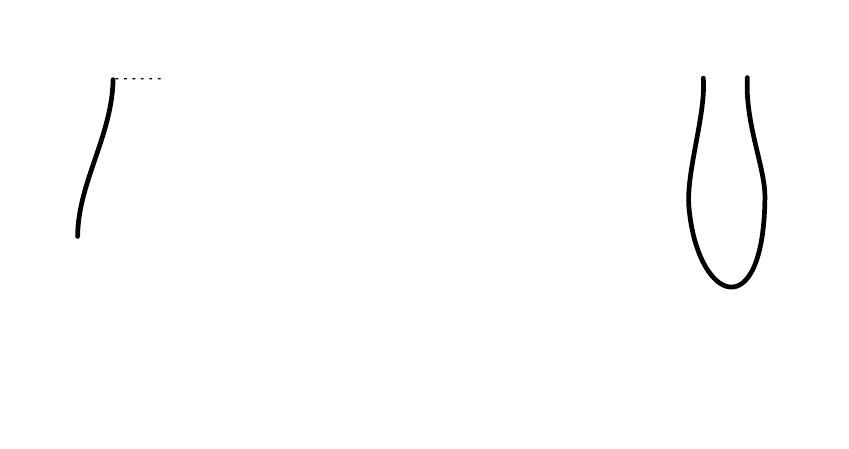
	\end{center}
	\caption{The boundary of $\NNa$. We abbreviate $C_*^\lambda(\mathcal P_K,\mathcal P_K^0;\mathbb C)$ as $C_*^\lambda$.}
	\label{bdryN1_fig}
\end{figure}	

	To see that $\Theta^{\lambda}$ is a quasi-isomorphism, we use the action filtration and a geodesic Morse model for $C_*^{\lambda}(\mathcal P_K,\mathcal P_K^0;\C)$. The generators of the two complexes can be identified and the map $\Theta^{\lambda}$ is upper triangular with $\pm1$ on the diagonal (coming from `trivial strips', with respect to the energy filtration). Details can be found in \cite{Asplund}, where the corresponding result is \cite[Theorem 4.1]{Asplund}; the fact that the map respects action follows from \cite[Proposition 4.7]{Asplund} and the upper triangular property follows from \cite[Theorem 4.12]{Asplund}.     
\end{proof}  

The previous result implies that we can compute $KCH_{\ast}^{\epsilon_{L_{K},\lambda}}(K;\C)$ by computing $H_*^{\lambda}(\mathcal P_K,\mathcal P_K^0;\C)$. 

\begin{lemma} \label{lem:KCHlambda}
For every $\lambda\in \mathbb C^*$
\begin{equation}\label{eq:basichomology}
KCH_k^{\epsilon_{L_K,\lambda}}(K;\C) \cong H_k^{\lambda}(\mathcal P_K,\mathcal P_K^0;\C) \cong H_k^{\lambda}(K\times K,\Delta;\C) \cong \begin{cases}
\C & \text{, if } k= 1,2 \\
0 & \text{, otherwise}
\end{cases}.
\end{equation}
\end{lemma}
\begin{proof}
Lemma \ref{lemma: Theta isomorphism} implies that the first isomorphism in \eqref{eq:basichomology} can be given by the map induced by $\Theta^\lambda$ on homology. 

For the second isomorphism, observe that the fiber of the evaluation at endpoints map $\mathcal{P}_{K}\to K\times K$ is the based loop space of $\R^{3}$, which is contractible. This map restricts to a bijective correspondence between the space of constant paths $\mathcal{P}_{K}^0 \subset \mathcal{P}_{K}$ and the diagonal $\Delta \subset K\times K$. Hence, the endpoints evaluation map induces an isomorphism between the long exact sequences of homology (with local coefficients specified by $\lambda \in \C^*$ as above) of the pairs $(\mathcal{P}_{K},\mathcal{P}_{K}^0)$ and $(K\times K, \Delta)$. This yields the second isomorphism in the statement.

The third isomorphism comes from the homology long exact sequence of the pair $(K\times K,\Delta)$.
\end{proof}

\begin{remark}
It is interesting to observe that 
$$
H_\ast^{\lambda}(K\times K;\C) \cong 
\begin{cases}
H_\ast (T^2;\C) & \text{ if } \lambda = 1 \\
0 & \text{ if } \lambda \neq 1
\end{cases},
$$
but that the ranks of the $H_k^{\lambda}(\Delta;\C)$ and of the $H_k^{\lambda}(K\times K,\Delta;\C)$ do not depend on $\lambda$. 
\end{remark}

We define another chain map 
\[ 
\Phi^{\lambda}\colon KCC_{\ast}^{\epsilon_{L_{K},\lambda}}(K;\C) \ \to \ C_{\ast-1}^{\lambda}(\mathcal{P}_{K}^{0};\C) \cong C_{\ast-1}^{\lambda}(K;\C)
\]
related to the above, as follows. We identify the space $\mathcal P_K^0$ of constant paths in $K$ with $K$ itself, and use the elements of the moduli space $\NNc$ in Figure \ref{bdryN1_fig} to define $\Phi^{\lambda}(c)$. More explicitly, $\Phi^{\lambda}(c)$ counts holomorphic disks $u$ in $T^*\R^3$ with boundary in $L_K$, one boundary puncture asymptotic to the Reeb chord $c$ of $\Lambda_{K}$, and one boundary puncture mapping to $K$. We fix a generic point $\xi\in K$ and, if $|c|=2$, then the disks $u$ contributing to $\Phi^{\lambda}(c)$ are those that send the boundary marked point to $\xi$. If $|c|\notin \{1,2\}$, then we set $\Phi^{\lambda}(c)=0$. (This map is essentially $d'_{\epsilon_0}$ in \cite[Lemma 4.6]{EkholmNgHigherGenus}.)


In the next result, $\delta\colon H_k^\lambda(\mathcal P_K,\mathcal P_K^0;\mathbb C) \to H_{k-1}^\lambda(\mathcal P_K^0;\mathbb C)  \cong H_{k-1}^{\lambda}(K;\C)$ denotes the connecting homomorphism in the homology long exact sequence of the pair $(\mathcal P_K,\mathcal P_K^0)$, with local coefficients specified by $\lambda\in \mathbb C^*$ as above.

\begin{lemma} \label{lem:PhiTheta}
$\left[\Phi^{\lambda}\right] = \delta\circ \left[\Theta^\lambda\right]$, where we use square brackets to denote the induced maps on homology.
\end{lemma}

\begin{proof}
By \eqref{eq:basichomology}, we only need to consider the cases where $k=1$ or $2$. Let $y\in KCC_{k}^{\epsilon_{L_{K},\lambda}}(K;\C)$ be a cycle. We will consider the boundary components in Figure \ref{bdryN1_fig}, for chords $c$ that contribute to $y$. Since $y$ is a cycle, the total contribution from curves in the moduli space $\mathcal{N}(c,\boldsymbol{a}, L_K\cup\R^3)$ in the right-hand side of Figure \ref{bdryN1_fig} vanishes. Curves in $\NNc$ are used to compute $[\Phi^{\lambda}(y)]$. Since the connecting homomorphism $\delta$ of a relative cycle is given by its boundary, the lemma follows.  
\end{proof}


For every $\lambda \in \mathbb C^*$ and every $k$, the differential $\partial_k^{\lambda}\coloneq\partial_k^{\epsilon_{L_{K},\lambda}}$ is a map from $KCC_k^\lambda(K;\mathbb C)$ to $KCC_{k-1}^\lambda(K;\mathbb C)$. The vector space $KCC_k^\lambda(K;\mathbb C)$ is $\mathbb C^{m_k}$, where $m_k$ is the number of chords of degree $k$. Since the $\partial_k^{\lambda}$ and $\Theta^\lambda$ depend analytically on $\lambda$ (as Laurent polynomials), we can define an analytic map $y\colon U \to \mathbb C^{m_1}$, where $U$ is some open neighborhood of 1 in $\mathbb C$ (with respect to the standard metric topology), such that:
\begin{itemize}
\item $\partial_1^{\lambda}(y(\lambda)) = 0$ for every $\lambda\in U$ and 
\item $[\Theta^\lambda(y(\lambda))] = 1 \in \mathbb C\cong H^\lambda_1(\mathcal P_K,\mathcal P_K^0;\mathbb C)$.
\end{itemize}


\begin{lemma} \label{lem:Phi(y)}
Denoting the class of a point in $H^\lambda_0(K;\mathbb C)$ by $\xi_{0}$, we have 
$$\left[\Phi^{\lambda}(y(\lambda))\right]=(1-\lambda)\, \xi_{0}.$$  
\end{lemma}
\begin{proof}
Using Lemma \ref{lem:PhiTheta} and the connecting homomorphism $\delta\colon H_1^\lambda(K\times K,\Delta;\mathbb C)\to H_0^\lambda(\Delta;\mathbb C)$, we have
$$
\left[\Phi^{\lambda}(y(\lambda))\right] = \delta\left(\left[\Theta^{\lambda}(y(\lambda))\right] \right) = \delta(1) = (1-\lambda)\, \xi_0.
$$
The term $1-\lambda$ in the last identity comes from the local coefficient system used to twist the homology differentials. 
\end{proof}

\begin{definition} \label{function G}
Let
\begin{equation} \label{def G}
G(\lambda,\mu,Q) \coloneq \partial (y(\lambda)) |_{\boldsymbol{a} \mapsto \epsilon_{L_{K},\lambda}(\boldsymbol{a})},
\end{equation}
where $\partial$ denotes the differential in the DGA $\mathcal A_K$, viewed as a Laurent polynomial in $(\lambda,\mu,Q)$ and a polynomial in the degree zero Reeb chords $a_{i}$, and where we substitute each degree zero Reeb chord $a_{i}$ with $\epsilon_{L_{K},\lambda}(a_{i})$ (which is a function of $\lambda$). 
\end{definition}

Since $y(\lambda)$ depends analytically on $\lambda$ near $\lambda=1$, we have that $G(\lambda,\mu,Q)$ is an analytic function of $(\lambda,\mu,Q)$ in a neighborhood of $(1,1,1)$. The next results are about partial derivatives of $G$. 

\begin{lemma} \label{Phi(y)=d mu G}
Denoting the class of a point in $H^\lambda_0(K;\mathbb C)$ by $\xi_{0}$, we have 
$$
\left[\Phi^{\lambda}(y(\lambda))\right] = (\partial_\mu G )|_{\mu = Q = 1} \, \xi_{0}.
$$ 
\end{lemma}
\begin{proof}
Pick a non-compact surface $\sigma_K \subset L_K$, with a single boundary component going once around $K$ and with boundary at infinity in the curve $x$. The surface $\sigma_K$ gives a Borel--Moore homology between $K$ and the curve $x$, and we call it a bounding chain for $K$. Let $c$ be a Reeb chord of degree 1. Consider the moduli space of holomorphic disks in $T^*\R^3$ with boundary on $L_{K}$, with one boundary puncture asymptotic to $c$ and one boundary intersection with $\sigma_K$, with orientation induced by the orientation of the original moduli space, the orientation of the boundary of the disk, and the orientations of the bounding chain and $L_K$. This is represented on the left in Figure \ref{bdry0_fig}. The moduli space is 1-dimensional, with a natural compactification whose elements are represented in Figure \ref{bdry0_fig}.
 
\begin{figure}
  \begin{center}
    \def\svgwidth{.8\textwidth}
    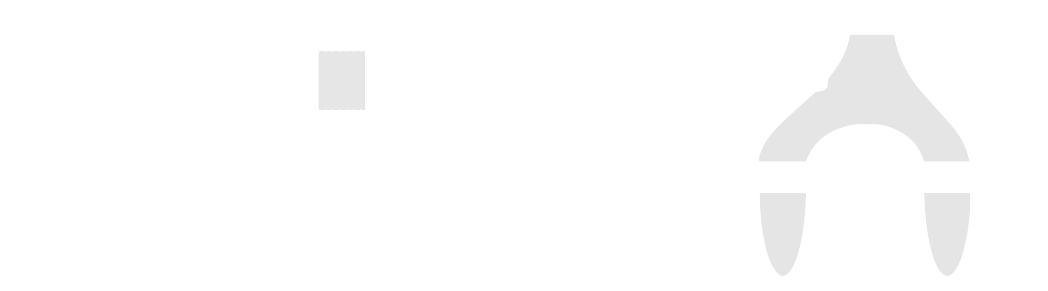
  \end{center}
  \caption{The boundary of the moduli space in the proof of Lemma \ref{Phi(y)=d mu G}. 
In the configuration on the right, the punctured disk at the top could have an arbitrary number of negative punctures capped by disks with boundary on $L_K$.}
  \label{bdry0_fig}
\end{figure}

Since $y(\lambda)$ is a linear combination of Reeb chords of degree 1, by taking $c$ in Figure \ref{bdry0_fig} to be each of these Reeb chords we get an equation with three terms (one for each boundary component in the figure). Since $y$ is a cycle, the first term vanishes (because it involves the linearized differential). The second term is $\Phi^\lambda(y(\lambda))$ (recall the definition of $\Phi^\lambda$ above, in terms of the moduli spaces $\NNc$ illustrated in Figure \ref{bdryN1_fig}).  To determine the third term, recall from Section \ref{choices and coeffs} that the $\mu$-powers in the differential of $\mathcal A_K$ are given by counting intersections of the boundaries of holomorphic disks with a longitude curve $x$ in $\Lambda_K$. Keeping in mind Remark \ref{rmk:derivatives} and the fact that $\mu = e^p$, we see that third term is 
$$(\partial_p G )|_{\mu = Q = 1} \,\xi_{0} = (\mu \, \partial_\mu G )|_{\mu = Q = 1} \, \xi_{0}= (\partial_\mu G )|_{\mu = Q = 1} \, \xi_{0}.$$ 
The lemma now follows.
\end{proof}





\begin{corollary} \label{d2G}
We have	
$\partial_{\mu} G(\lambda,1,1)=1-\lambda$. 
\end{corollary}
\begin{proof}
This follows by combining Lemmas \ref{lem:Phi(y)} and \ref{Phi(y)=d mu G}. 
\end{proof}


Next, we get more information on the function $G$ by using the results in Section \ref{sec:main proof} about counts of holomorphic annuli in $T^{\ast}\R^{3}$ between $\R^{3}$ and $L_{K}^{\delta}$. 
Recall the power series $\An_{L_{K}}(\lambda)$ from Section \ref{models}, the identifications of variables for $L=L_{K}^{\delta}$: $\alpha=x$ and $\beta=p$, as well as $\lambda=e^{x}$, $\mu=e^{p}$,  and $Q=e^{t}$.

\begin{figure}
	\begin{center}
		\def\svgwidth{.6\textwidth}
		\hspace{1cm}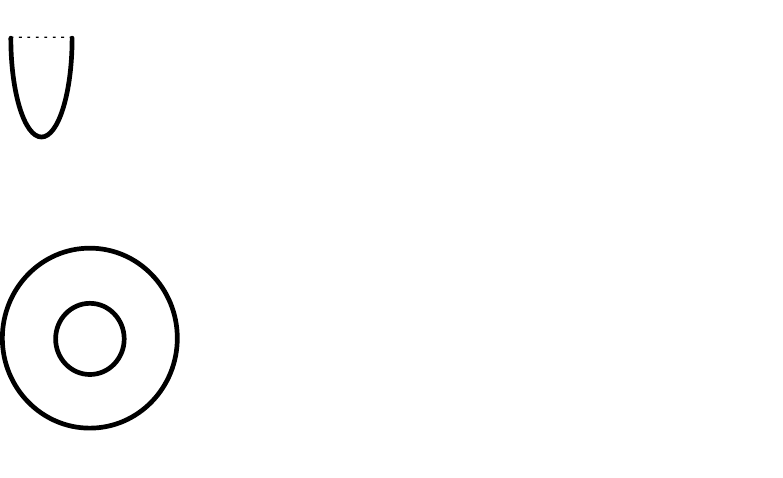
	\end{center}
	\caption{The moduli spaces in \eqref{eq:AlexGfibered5}.}
	\label{cobordismL_fig}
\end{figure}

\begin{proposition} \label{Alex G fibered}
	The following equation holds for any knot $K$:
	\begin{equation} \label{dAnL/dx}
	 \left(\partial_{p}G(e^{x},1,1)\right)\cdot\left(\partial_{x}\An_{L_{K}}(e^{x})\right)
	+ \left(\partial_{t}G(e^{x},1,1)\right) = 0.
	\end{equation}
Equivalently, we can write 
	\begin{equation*} 
	 \left(\partial_{\mu}G(\lambda,1,1)\right)\cdot\left(\lambda \, \partial_{\lambda}\An_{L_{K}}(\lambda)\right)
	+ \left(\partial_{Q}G(\lambda,1,1)\right) = 0.
	\end{equation*}
\end{proposition}
\begin{proof}
	We will use Lemma \ref{l:boundaryannuliy}, with $L=\Ld$. Since $L_{K}^{\delta}\cap\R^{3}=\varnothing$, the last two moduli spaces in Equation \eqref{eq:intpointsy} are empty. Combining Lemmas \ref{l:boundaryannuliy}, \ref{l:boundingchainann} and \ref{l:4chaincapL}, it follows that 
	\begin{equation}\label{eq:AlexGfibered5}
	\mathcal{M}^\infty(c,\sigma) \cup \mathcal{M}^{\infty}(c,\partial^{\infty}V) \cup \mathcal{M}^{\infty}(c,\tau),
	\end{equation}
see Figure \ref{cobordismL_fig}, 
together with the moduli spaces of the curves that contribute to the linearized differential of the degree 1 chord $c$ (for example, $\mathcal{M}_{\mathrm{an}}(c,\boldsymbol{a})$ in Lemma \ref{l:boundaryannuliy} and in Figure \ref{bdryL1_fig}) forms the oriented boundary of a 1-dimensional manifold.

Recall that $y(\lambda)$ is a linear combination of degree 1 chords, and that it is a linearized homology cycle for every $\lambda$. We will consider \eqref{eq:AlexGfibered5} for all the chords $c$ that are involved in $y(\lambda)$, and express algebraically the counts of elements in each moduli space. First, observe that since $y(\lambda)$ is a cycle, contributions from moduli spaces of curve of its linearized differential cancel out and we get an algebraic identity by considering only the spaces in \eqref{eq:AlexGfibered5}.

Let $c$ be a degree 1 chord appearing in $y(\lambda)$. Elements in the moduli space $\mathcal{M}^{\infty}(c,\sigma)$ consist of pairs $(u_1, u_2)$, where $u_1$ is an annulus in $T^*\R^3$ with one boundary component on $\R^3$ and another on $\Ld$, and $u_2$ is a disk in $\R\times ST^*\R^3$ with one positive puncture asymptotic to $c$ and a boundary marked point intersecting the bounding chain $\sigma_{u_1}$. Using the calculation of the $\lambda$- and $\mu$-powers in $\mathcal A_K$ explained in Section \ref{choices and coeffs} and Remark \ref{rmk:derivatives}, the count of elements in $\mathcal{M}^{\infty}(c,\sigma)$ is given by 
	\begin{equation} \label{q: partial derivative}
	\left(\left(\partial_{p} (\partial c)\right)\Big|_{\begin{subarray}{l} \boldsymbol{a} \mapsto \epsilon_{L_K,\lambda}(\boldsymbol{a}) \\ \mu=Q=1\end{subarray}}\right) \cdot\left(\partial_{x}\An_{L_{K}}(e^{x})\right).
	\end{equation}
 (Here in $\partial_p(\partial c)$, $\partial_p$ means taking a partial derivative with respect to the variable $p$, whereas $\partial c$ means taking the differential in the dg-algbera $\mathcal A_K$.) It is useful to observe that the coefficient of an annulus $u$ in $\partial_{x}\An_{L_{K}}(e^{x})$ is (up to sign) given by $d(u)/m(u)$ (degree over multiplicity). This is an integer number,  specifically the intersection number of the boundary of $u$ with a surface Poincar\'e-dual to the longitude curve $x$. Taking the linear combination of the expressions above, given by the coefficients of $y(\lambda)$ as a linear combination of chords $c$, we get 
	\begin{equation}\label{eq:AlexGfibered3}
	\left(\partial_{p}G(e^{x},1,1)\right)\cdot\left(\partial_{x}\An_{L_{K}}(e^{x})\right).
	\end{equation}

The union $\mathcal{M}^{\infty}(c,\partial^{\infty}V)\cup \mathcal{M}^{\infty}(c,\tau)$ contains disks in $\R\times ST^*\R^3$ with an interior puncture in $\partial^{\infty}V$ or with a boundary puncture in $\tau$. By Lemma \ref{compute Q}, one can choose $\tau$ such that the count of such disks can be written as
	\begin{equation*}
	\left(\partial_{t} (\partial c)\right)\Big|_{\begin{subarray}{l} \boldsymbol{a} \mapsto \epsilon_{L_K,\lambda}(\boldsymbol{a}) \\ \mu=Q=1\end{subarray}},
	\end{equation*}
see Remark \ref{rmk:derivatives}. Taking a linear combination as before, using the coefficients of $y(\lambda)$, we get
	\begin{equation}\label{eq:AlexGfibered6} 
\partial_{t}G(e^{x},1,1).
	\end{equation}
Combining \eqref{eq:AlexGfibered3} and \eqref{eq:AlexGfibered6} with the fact that $y(\lambda)$ is a cycle yields the result.
\end{proof}



\begin{remark}
Recall the non-uniqueness of $\tau$, see Remark \ref{r:nonuniquenessoftau}. Picking a $\tau$ differing from the one above by $k\cdot p$ would change the equation \eqref{dAnL/dx} to 
\[ 
\left(\partial_{p}G(e^{x},1,1)\right)\cdot\left(\partial_{x}\An_{L_{K}}(e^{x})\right)
+ \left(\partial_{t}G(e^{x},1,1)\right) + k\left(\partial_{p}G(e^{x},1,1)\right) = 0,
\]
corresponding to a change of variables $t\mapsto t+kp$.
\end{remark}

\begin{proposition} \label{partial mu G partial Q G}
$\lambda \, (\partial_\mu G)(\lambda,1,1) = (\lambda-1) (\partial_Q G)(\lambda,1,1)$.
\end{proposition}
\begin{proof}
According to Corollary \ref{AnL}, 
$$
\An_{L_{K}}(e^x) = \sum_{k>0} \frac{e^{kx}}{k} = - \log(1-e^x),
$$
so
$$
\left(\partial_{x} \An_{L_{K}}(e^x)\right) = \frac{e^x}{1-e^x} = \frac{\lambda}{1-\lambda}. 
$$
This combined with Proposition \ref{Alex G fibered} gives the result. 
\end{proof}

\begin{corollary} \label{d Q G}
$(\partial_Q G)(1,1,1) \neq 0$.
\end{corollary}
\begin{proof}
Differentiate the equality in Proposition \ref{partial mu G partial Q G} with respect to $\lambda$: 
$$
\frac{\partial G}{\partial \mu}(\lambda,1,1) + \lambda \frac{\partial^2 G}{\partial \lambda \partial \mu}(\lambda,1,1) = \frac{\partial G}{\partial Q}(\lambda,1,1) + (\lambda-1)\frac{\partial^2 G}{\partial \lambda \partial Q}(\lambda,1,1). 
$$
Evaluate at $\lambda = 1$ and use Corollary \ref{d2G}. 
\end{proof}

\subsection{Full linearized homology}

Recall the definition of the full linearized knot contact homology in Section \ref{sec:aug var poly}. In the proof of the next result, we will use Corollary \ref{d Q G} from Section \ref{sssec:1-parameter generators}.

\begin{lemma} \label{hat KCH MK}
There is an open neighborhood $U \subset \tilde V_K$ of $\boldsymbol{\epsilon_0} \in (\C^*)^3\times \C^{\boldsymbol{a}}$ (with respect to the standard metric topology) such that, for every $\epsilon\in U$, 
$$
\widehat{KCH}_k^\epsilon(K;\C)\cong
\begin{cases}
 \C & \text{if $k=2$} \\
 \C^2 & \text{if $k=0$} \\
 0 & \text{otherwise}
\end{cases}.
$$
\end{lemma}
\begin{proof}
Denote by $\epsilon_0$ the augmentation that gives the point $\boldsymbol{\epsilon_0} \in \tilde V_K$. 
In Lemma \ref{lem:KCHlambda}, it was shown that 
\begin{equation} \label{eq:KCH^epsilon0}
KCH_*^{\epsilon_0}(K;\C) \cong \begin{cases}
\C & \text{if $k = 1$ or 2} \\
 0 & \text{otherwise}
\end{cases}
\end{equation}

For each integer $k\geq 0$, denote by $m_k$ the dimension of $KCC^\epsilon_k(K;\C)$. This dimension is the number of Reeb chords of degree $k$, and is independent of the augmentation $\epsilon$. 
The vector space $\widehat{KCC}^{\epsilon}_*$ can be identified with $KCC^{\epsilon}_*$, with additional generators $\lambda, \mu , Q$ in degree 0. 
Note that for each $k\geq 2$ we have 
$\widehat{\partial}_k^{\epsilon_0} = \partial_k^{\epsilon_0}$, so 
$$
\widehat{KCH}_k^{\epsilon_0}(K;\C) \cong KCH_k^{\epsilon_0}(K;\C) \cong \begin{cases}
\C & \text{ if } k=2 \\
0 & \text{ if } k>2 
\end{cases} .
$$

Next, we will show that
\begin{equation} \label{claim 1}
\dim(\im (\widehat{\partial}_1^{\epsilon_0})) = \dim(\im (\partial_1^{\epsilon_0})) + 1.
\end{equation}
Before proving this note that, since $KCH_0^{\epsilon_0}(K;\C)=0$, equation \eqref{claim 1} implies that $\dim(\im (\widehat{\partial}_1^{\epsilon_0})) =\dim \widehat{KCC_0} - 2$, hence $\widehat{KCH}_0^{\epsilon_0}(K;\C)\cong \C^2$. The fact that $\widehat{KCH}_1^{\epsilon_0}(K;\C) \cong 0$ then follows from the Euler characteristic computation
$$
\chi\left(\widehat{KCH}_*^{\epsilon_0}(K;\C)\right)= \chi\left(KCH_*^{\epsilon_0}(K;\C)\right) + 3 = 3.
$$

Consider \eqref{claim 1}. Observe that $\codim \big(\im(\partial_2^{\epsilon_0}) \subset \ker(\partial_1^{\epsilon_0})\big)=1$ and that, using $y(\lambda)$ as defined before Lemma \ref{lem:Phi(y)}, $y(1)$ is in the complement $\ker(\partial_1^{\epsilon_0}) \setminus \im(\partial_2^{\epsilon_0})$. Denote $y(1)$ by $\beta_0$.  Pick a basis for a complementary subspace to  
$\ker(\partial_1^{\epsilon_0})$ in $\C^{m_1}$, and denote its elements as $\beta_j$, $1\leq j \leq m_0$ (note that $m_1 = \dim(\ker(\partial_1^{\epsilon_0})) + m_0$, since ${KCH}_0^{\epsilon_0}(K;\C)=0$). 
Define new variables 
$l=\lambda - {\epsilon_0}(\lambda) = \lambda-1$, $m= \mu-{\epsilon_0}(\mu) = \mu-1$ and $q=Q-{\epsilon_0}(Q) = Q-1$. We can then write the dg-algbera differential of the $\beta_j$ as  
$$
\partial \beta_j = \xi_j^0(l,m,q) + \sum_{i= 1}^{m_0} \xi_{ji}^1(l,m,q) a_i + \mathcal{O}(\boldsymbol{a}^2)
$$
where the $\xi_{j}^0$ and $\xi_{ji}^1$ are power series satisfying $\xi_j^0(0,0,0) = 0$ for all $0\leq j \leq m_0$ and  $\xi_{0i}^1(0,0,0)=0$ for all $1\leq i \leq m_0$. Our assumptions on the $\beta_j$ also imply that the square matrix 
$$
\Big( \xi_{ji}^1 (0,0,0) \Big)_{1\leq j, i \leq m_0}
$$
is invertible. Now, 
\begin{equation} \label{eq:d hat beta 0}
\widehat{\partial}_1^{\epsilon_0}\left( \beta_0\right) = \left(\frac{\partial\xi_0^0}{\partial l}(0,0,0)\right) l + \left(\frac{\partial\xi_0^0}{\partial m}(0,0,0)\right) m + \left(\frac{\partial\xi_0^0}{\partial q}(0,0,0)\right) q. 
\end{equation}
Corollary \ref{d Q G}, implies that the $q$-derivative in the expression 
above is non-zero.
The fact that 
\begin{equation*} 
\widehat{\partial}_1^{\epsilon_0}\left( \beta_j\right) = \left(\frac{\partial\xi_j^0}{\partial l}(0,0,0)\right) l + \left(\frac{\partial\xi_j^0}{\partial m}(0,0,0)\right) m + \left(\frac{\partial\xi_j^0}{\partial q}(0,0,0)\right) q + \partial_1^{\epsilon_0}(\beta_j) 
\end{equation*}
for all $1\leq j\leq m_0$ and the non-vanishing of $\widehat{\partial}_1^{\epsilon_0}\left( \beta_0\right)$ imply \eqref{claim 1}. %
\footnote{An anonymous referee pointed out the alternative way of presenting the calculation of $\widehat{KCH}_k^{\epsilon_0}(K;\C)$. The short exact sequence of chain complexes $0\to \C^3 \to \widehat{KCH}_*^{\epsilon_0}(K;\C) \to {KCH}_*^{\epsilon_0}(K;\C) \to 0$ combined with \eqref{eq:KCH^epsilon0} yields the long exact sequence 
$$
0 \to \widehat{KCH}_1^{\epsilon_0}(K;\C) \to \C \to \C^3 \to \widehat{KCH}_0^{\epsilon_0}(K;\C) \to 0.
$$
The result follows from the non-triviality of the map $\C\to \C^3$, which is implied by the non-vanishing of \eqref{eq:d hat beta 0}.}

To finish the proof of the lemma, we show that if the lemma holds for some $\epsilon_1\in \tilde V_K$, then it also holds for all other augmentations in a neighborhood of $\epsilon_1$. Given a Reeb chord $c$, we can think of 
$\widehat{\partial}^\epsilon(c)$ as the result of applying $\epsilon$ to all but one occurences of $\lambda$, $\mu$, $Q$, 
or Reeb chords in each term in $\partial c$. This is clearly continuous in $\epsilon$. By assumption, the rank of $\widehat{\partial}_2^{\epsilon_1}$ is $m_1-m_0-1$. The continuity of $\widehat{\partial}^\epsilon$ guarantees that 
the rank of $\widehat{\partial}_2^{\epsilon}$ is also at least $m_1-m_0-1$ near $\epsilon_1$, which immediately implies the upper semicontinuity of $\dim \widehat{KCH}_2^\epsilon(K;\C)$ at $\epsilon_1$. 
Similarly, the rank of $\widehat{\partial}_1^{\epsilon_1}$ is $m_0+1$ at $\epsilon_1$, hence it is at least as much for nearby $\epsilon$. This implies the upper semicontinuity of $\dim \widehat{KCH}_0^\epsilon(K;\C)$. The lemma follows.
\end{proof}

\begin{remark}
The last part of the proof of Lemma \ref{hat KCH MK} concerns the upper semicontinuity of $\dim \widehat{KCH}_0^\epsilon(K;\C)$ and of $\dim \widehat{KCH}_2^\epsilon(K;\C)$ at an $\epsilon_1$ where the lemma holds. If neither dimension can increase near $\epsilon_1$, then using Euler characteristic $=3$,
$$
\dim\widehat{KCH}_1^\epsilon(K;\C) = \dim\widehat{KCH}_0^\epsilon(K;\C) + \dim\widehat{KCH}_2^\epsilon(K;\C) - 3 \leq 2 + 1 - 3 = 0
$$ 
and $\dim\widehat{KCH}_1^\epsilon(K;\C) =0$. Similarly, the dimensions of $\widehat{KCH}_0^\epsilon(K;\C)$ and $\widehat{KCH}_2^\epsilon(K;\C)$ must also be constant near $\epsilon_1$. 

A more general approach to this would be to show that $\widehat{KCH}_*^\epsilon(K;\C)$ is a coherent sheaf on $\tilde V_K$, then the upper semicontinuity of the dimensions of the cohomology stalks 
would give us the upper semicontinuity of $\widehat{KCH}_k^\epsilon(K;\C)$ for all $k$.   
\end{remark}

Lemma \ref{hat KCH MK} has the following consequence. We use notation as in its proof: $\boldsymbol{a}=(a_{i})$ and $\boldsymbol{b}=(b_{j})$ are the degree zero and degree one Reeb chords of $\Lambda_{K}$, respectively, $\beta_{0}$ is a non-zero element in $\ker(\partial_{1}^{\epsilon_{0}})$ in the complement of $\im(\partial_{2}^{\epsilon_{0}})$, and $\beta_{j}\in\C^{\boldsymbol{b}}$, $j=1,\dots,m_0$, is a basis of a subspace of $\C^{\boldsymbol{b}}$ complementary to $\ker(\partial_{1}^{\epsilon_{0}})$.  Let $\tilde W_{K}\subset (\C^{\ast})^{3}\times\C^{\boldsymbol{a}}\cong (\C^{\ast})^{3}\times\C^{m_0}$ be the zero locus of the polynomials $\partial\beta_{j}$, $j=0,\dots,m_0$. By definition, $\tilde V_K \subset \tilde W_K$.

\begin{corollary} \label{partial beta i = 0 smooth}
 The solution set $\tilde W_K$ is a complex manifold of dimension 2 in an open neighborhood of $\boldsymbol{\epsilon_0}\in (\C^*)^3\times \C^{\boldsymbol{a}}$ (with respect to the standard metric topology).
\end{corollary}
\begin{proof}
 The polynomials $\partial \beta_j$ define a map $(\C^{\ast})^{3}\times\C^{\boldsymbol{a}} \to \C^{m_0+1}$. The differential of this map at $\boldsymbol{\epsilon_0}$ can be identified with $\widehat{\partial}_1^{\epsilon_0}$, and was shown in the proof of Lemma \ref{hat KCH MK} to have full rank. Indeed, the matrix of partial derivatives in the variables $q=Q-1$ and $\boldsymbol{a}$ is invertible. The implicit function theorem then implies that $\tilde W_K$ is smooth near $\boldsymbol{\epsilon_0}$, and locally parametrized by $l=\lambda-1$ and $m=\mu-1$.  
\end{proof}


\subsubsection{Proof of Proposition \ref{KCH1 generator} (1)}
The statement for the augmentation $\epsilon_0$ giving $\boldsymbol{\epsilon_0}\in \tilde V_K$ follows from Lemma \ref{lem:KCHlambda}. The upper semicontinuity of the ranks of cohomology groups implies that, for $\epsilon$ near $\epsilon_0$, $KCH_*^\epsilon(K;\C)$ either is as stated or vanishes. The argument at the end of the proof of Lemma \ref{hat KCH MK} can be adapted to show this semicontinuity. 
Assume now that $KCH_*^\epsilon(K;\C)$ vanishes for $\epsilon$ arbitrarily close to $\epsilon_0$. Then, since $\widehat \partial_2^\epsilon = \partial_2^\epsilon$, $\widehat{KCH}_2^\epsilon(K;\C) = 0$, which contradicts Lemma \ref{hat KCH MK}. The result follows. 
\qed

\subsubsection{Proof of Proposition \ref{full V_K 2d}}
 It is sufficient to prove both statements for a neighborhood of $\boldsymbol{\epsilon_0} \in \tilde V_K$ in the standard metric topology, since the singular locus of an affine variety is Zariski-closed. 
 Corollary \ref{partial beta i = 0 smooth} implies that the vanishing locus of the polynomials $\partial \beta_j$ is a smooth complex variety $\tilde W_K$ of dimension 2 near $\boldsymbol{\epsilon_0}$. We show that $\tilde V_K$ coincides with $\tilde W_K$ near $\boldsymbol{\epsilon_0}$. 
 
 Pick a collection  $\{c_i\}_{1\leq i \leq m_1-m_0-1}$ of Reeb chords of degree 2, such that the $\gamma_i \coloneq \partial_2^{\epsilon_0}(c_i)$ are linearly independent. These exist, by Proposition \ref{KCH1 generator} (1). Note that the collection $\{\beta_j,\gamma_i\}$, with $0\leq j \leq m_0$ and $1\leq i \leq m_1-m_0-1$, is a basis of $\C^{\boldsymbol{b}}=KCC_1^\epsilon(K;\C)$ for any $\epsilon$. Hence, $\tilde V_K$ is the vanishing locus of all the $\partial \beta_j$ and all the $\partial \gamma_i$. 
 
For all $1\leq k \leq m_1-m_0-1$, we have 
\[ 
\partial c_k = \alpha_k(\lambda,\mu,Q) \, \gamma_k + \sum_i \mathcal O(\boldsymbol{\tilde a}) \, \gamma_i + \sum_j \mathcal O(\boldsymbol{\tilde a}) \, \beta_j,
\]
where the coefficients $\mathcal O(\boldsymbol{\tilde a})$ depend on $(\lambda,\mu,Q,\boldsymbol{a})$ and have order at least one in the variables $\tilde a_i \coloneq a_i-{\epsilon_0}(a_i)$, and $\alpha_k(1,1,1) = 1$. Since $\partial^2=0$ and there are no chords of negative degree, 
\begin{equation}\label{partialpartial ej}
\alpha_k \, \partial \gamma_k = \sum_i \mathcal O(\boldsymbol{\tilde a}) \, \partial \gamma_i + \sum_j \mathcal O(\boldsymbol{\tilde a}) \, \partial \beta_j.
\end{equation}

Assume that, for some $k$, $\partial \gamma_k$ does not vanish near $\boldsymbol{\epsilon_0}$ in $\tilde W_K$. This will lead us to a contradiction. We know that $\tilde W_K$ is 2-dimensional and parametrized by $\lambda-1$ and $\mu-1$ near $\boldsymbol{\epsilon_0}$. The restriction of $\partial \gamma_k$ to $\tilde W_K$ can be written as a power series in $\lambda-1$ and $\mu-1$ near $\boldsymbol{\epsilon_0}$. By assumption, this series is non-zero, and has a minimal order in $\lambda-1$ and $\mu-1$. Pick the $k$ for which this minimal order is the smallest, and denote this order by $s$. The variables $\tilde a_i$ can also be expressed as power series in $\lambda-1$ and $\mu-1$ (near $\boldsymbol{\epsilon_0}$ in $\tilde W_K$), and we know that the constant terms in these series vanish. Now, restricting equation \eqref{partialpartial ej} to $\tilde W_K$, and using the fact that $\alpha_k(1,1,1) \neq 0$ and that the $\partial \beta_j$ vanish along $\tilde W_K$, we conclude that the right side of the equation has order greater than $s$, which is a contradiction. 

We conclude that the $\partial \gamma_k$ vanish in a neighborhood of $\boldsymbol{\epsilon_0}$ in $\tilde W_K$, and hence $\tilde W_K \subset \tilde V_K$ in that neighborhood. We already know that $\tilde V_K \subset \tilde W_K$. Hence the two coincide and the result follows. 
\qed

\begin{remark}
Proposition \ref{full V_K 2d} implies that $\tilde V_K$ is smooth and 2-dimensional near $\boldsymbol{\epsilon_0}$. According to Lemma \ref{KCH0 tangent}, the cotangent space to $\tilde V_K$ at a point $\boldsymbol{\epsilon}$ coming from an augmentation $\epsilon$ is isomorphic to $\widehat{KCH}_0^\epsilon(K,\C)$. Thus this vector space is 2-dimensional for $\boldsymbol{\epsilon}$ in a neighborhood of $\boldsymbol{\epsilon_0}$, according to Lemma \ref{hat KCH MK}. 
\end{remark}


\subsection{A 2-parameter family of generators}

By Proposition \ref{full V_K 2d}, we can parametrize $\tilde V_K$ by coordinates $\lambda, \mu$ near $\boldsymbol{\epsilon_0}$. Let $\epsilon_{\lambda,\mu}$ denote the augmentation at $(\lambda,\mu)$. Consider an analytic family $y(\lambda,\mu)$ of cycles in $KCC_{1}^{\epsilon_{\lambda,\mu}}(K;\C)$ generating $KCH_1^{\epsilon_{\lambda,\mu}}(K;\C)$ (by Proposition \ref{KCH1 generator}(1) this homology group is 1-dimensional, possibly after restricting to a smaller open set). To see why such an analytic family exists we argue as for $y(\lambda)$ in Definition \ref{function G}: we have analytic maps 
\begin{align*}
U &\to \mathrm{Fl}_n(k,k+1) \\
(\lambda,\mu) &\mapsto \left(\im \partial_2^{\epsilon_{\lambda,\mu}},\ker \partial_1^{\epsilon_{\lambda,\mu}}\right)
\end{align*}
where $U$ is an open subset of $\C^2$ and $\mathrm{Fl}_n(k,k+1)$ is the partial flag variety of $k$ and $k+1$ dimensional vector subspaces of $\C^n$. We can thus think of $\ker \partial_1^{\epsilon_{\lambda,\mu}}$ as defining a holomorphic vector bundle $E_1$ over $U$, with a holomorphic subbundle $E_2$ given by $\im \partial_2^{\epsilon_{\lambda,\mu}}$. The analytic function $y(\lambda,\mu)$ can be defined locally by taking a section of $E_1$ that is transverse to $E_2$. After multiplication of $y(\lambda,\mu)$ by an analytic function, we can assume that $y(\lambda,1)=y(\lambda)$, where $y(\lambda)$ is as in Definition \ref{function G}. In a manner similar to that definition, we have the following. 

\begin{definition} \label{function F}
Let 
$$
  F(\lambda, \mu, Q) \coloneq \partial (y(\lambda,\mu))|_{\boldsymbol{a} \mapsto \epsilon_{\lambda,\mu}(\boldsymbol{a})},
$$ 
where $\partial$ is the DGA-differential on $\mathcal A_K$. 
\end{definition}

This is the function that appears in Proposition \ref{KCH1 generator} (2).

\subsubsection{Proof of Proposition \ref{KCH1 generator} (2)} Our goal is to show that $(\partial_Q F)(1,1,1) \neq 0$ (which is \eqref{partial Q F neq 0}) and that $(\partial_\lambda F)(1,\mu,1)$ is a non-constant function of $\mu$ (which is \eqref{partial lambda F neq 0}). 

We will do this by showing that
\begin{equation} \label{d Q (F-G)}
(\partial_Q F)(\lambda,1,Q) = (\partial_Q G)(\lambda,1,Q)
\end{equation}
and 
\begin{equation} \label{d mu (F-G)}
(\partial_\mu F)(\lambda,1,1) = (\partial_\mu G)(\lambda,1,1). 
\end{equation}
Before proving these two claims we explain how they imply the result. 
Recall that Corollary \ref{d Q G} says that $(\partial_Q G)(1,1,1) \neq 0$, and \eqref{d Q (F-G)} then gives $(\partial_Q F)(1,1,1) \neq 0$. 

Corollary \ref{d2G} says that $(\partial_\mu G)(\lambda,1,1) = 1-\lambda$, and \eqref{d mu (F-G)} then gives $(\partial_\mu F)(\lambda,1,1) = 1-\lambda$ which implies that $(\partial_\lambda \partial_\mu F)(1,1,1) \neq 0$. Therefore, the analytic function of $\mu$ given by $(\partial_\lambda F)(1,\mu,1)$ is non-constant.

We prove \eqref{d Q (F-G)} and \eqref{d mu (F-G)}. Consider the following Taylor expansions in $\mu-1$:
\begin{align*}
y(\lambda,\mu) &= y(\lambda) + (\mu-1) z(\lambda) + \mathcal O((\mu-1)^2) \\
\epsilon_{\lambda,\mu} &= \epsilon_\lambda + (\mu-1) \eta_\lambda + \mathcal O((\mu-1)^2)
\end{align*}
for suitable holomorphic functions $z$ and $\eta$. We can then write, 
\begin{align*}
F(\lambda,\mu,Q) &= G(\lambda, \mu, Q) + (\mu-1) \left( (\partial(y(\lambda)))|_{\boldsymbol{a} \mapsto (\epsilon_\lambda,\eta_\lambda)(\boldsymbol{a})} + (\partial(z(\lambda)))|_{\boldsymbol{a} \mapsto \epsilon_\lambda(\boldsymbol{a})} \right)\\ 
&+ \mathcal O((\mu-1)^2) 
\end{align*}
where $x_{\boldsymbol{a} \mapsto (\epsilon_\lambda,\eta_\lambda)(\boldsymbol{a})}$ is the sum of all possible ways of applying $\epsilon_\lambda$ to all but one of the $a_i$ in $x$, and applying $\eta_\lambda$ to that chosen $a_i$ in $x$. Equation \eqref{d Q (F-G)} follows immediately. On the other hand,  
\begin{align*}
\frac{\partial F}{\partial \mu}(\lambda,1,1) &= \frac{\partial G}{\partial \mu}(\lambda,1,1) + \left( (\partial(y(\lambda)))|_{\boldsymbol{a} \mapsto (\epsilon_\lambda,\eta_\lambda)(\boldsymbol{a})} + (\partial(z(\lambda)))|_{\boldsymbol{a} \mapsto \epsilon_\lambda(\underline a)} \right)|_{\mu=Q=1} = \\ 
&= \frac{\partial G}{\partial \mu}(\lambda,1,1) + \eta_\lambda \big( \partial_1^{\epsilon_\lambda} (y(\lambda))\big) + \epsilon_\lambda(\partial(z(\lambda))).
\end{align*}
The claim now follows from the fact that the two last summands vanish. The first is zero because $y(\lambda)$ is a cycle in the chain complex linearized by $\epsilon_\lambda$. The second vanishes since $\epsilon_\lambda \circ \partial= 0$. 
\qed


 
\subsubsection{Partial derivatives of $F$}

Consider the curves in Figure \ref{E_matrices_fig}. 
Recall that the $\xi_i$ and $\eta_i$ are the intersection points in $\Md\cap \R^3$ of indices 1 and 2, respectively. Let $D$ be a matrix whose entry $D_{ij}$ is given by the count of elements in $\MMag$, that we denote as $|\MMag|$. Recall that we chose arcs connecting $\xi_i$ to $\eta_j$ to a base point and collapsed a disk neighborhood $U$ of these arcs to a point. That way we close up the boundary curves of the disks and we can write the matrix entries $D_{ij}$ with coefficients in $\C[\mu^{\pm}]$. 
 
Write the family of generators $y(\lambda,\mu)$ as a linear combination $\alpha_1(\lambda,\mu) \, c_1 + \ldots +\alpha_m(\lambda,\mu) \, c_m$ of degree 1 Reeb chords $c_i$, with coefficients given by analytic functions $\alpha_i$. Let $E_\xi$ be a matrix whose entry $(E_{\xi})_{ij}$ is given by 
$$
\alpha_1(\lambda,\mu) \, |\mathcal M(c_1,\xi_i,\xi_j)| + \ldots +\alpha_m(\lambda,\mu) \, |\mathcal M(c_m,\xi_i,\xi_j)|.
$$
Similarly, let $E_\eta$ be a matrix with entry $(E_{\eta})_{ij}$ given by 
$$
\alpha_1(\lambda,\mu) \, |\mathcal M(c_1,\eta_i,\eta_j)| + \ldots +\alpha_m(\lambda,\mu) \, |\mathcal M(c_m,\eta_i,\eta_j)|.
$$



\begin{figure}
  \begin{center}
    \def\svgwidth{.75\textwidth}
    \hspace{1cm}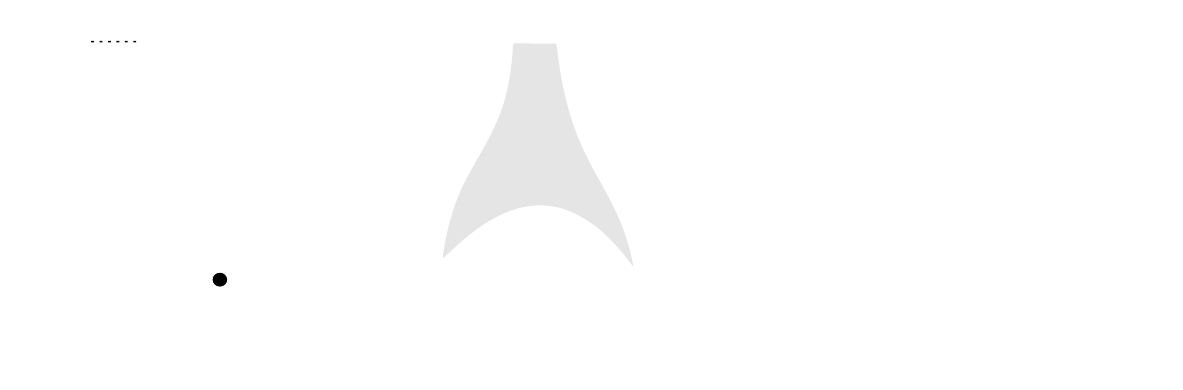
  \end{center}
  \caption{The moduli spaces $\MMae$, $\MMaf$ and $\MMag$}
  \label{E_matrices_fig}
\end{figure}


\begin{theorem} \label{Alex F general}
Taking $\lambda = e^x$, $\mu = e^p$ and $Q = e^t$, we get 
\begin{equation} \label{dAn/dp+}
 \left(\partial_{x} F(1,e^{p},1)\right)\left[\partial_{p}\An_{M_K}(e^p) + \Tr\left(D^{-1}(\partial_{p} D)(e^p)\right)\right] + \left(\partial_t F(1,e^{p},1)\right) = 0,  
\end{equation}
where $\Tr$ denotes the trace of a matrix. It follows that 
\begin{equation} \label{Alex from F}
\Delta_K(e^{p}) = (1-e^{p}) \exp\left(\int -\left(\frac{\partial_t F}{\partial_x F}\right) \Big|_{(1,e^p,1)} \, dp \right).
\end{equation}
\end{theorem}
\begin{proof}
The proof is similar to that of Proposition \ref{Alex G fibered}, but with $L=\Md$ instead of $\Ld$. In this case, the last two terms in \eqref{eq:intpointsy}
cannot be disregarded, and the corresponding count is given by the traces of the matrices $E_\xi$ and $E_\eta$, respectively. The analogue of equation \eqref{dAnL/dx} in Proposition \ref{Alex G fibered} is now
 \begin{equation} \label{d An/dp=dtF/dxF+Tr}
  (\partial_x F(1,e^p,1)) \cdot (\partial_{p}\An_{M_K}(e^p)) + (\partial_t F(1,e^p,1)) + \Tr E_\xi + \Tr E_\eta = 0.
 \end{equation}

To prove \eqref{dAn/dp+}, we will use Lemmas \ref{l:boundaryoftriangles} and \ref{l:brokentriangles} to express $\Tr E_\xi + \Tr E_\eta$ in \eqref{d An/dp=dtF/dxF+Tr} differently. Figures \ref{bdry6_fig} and \ref{bdry7_fig} are helpful to visualize the moduli spaces that are relevant for those lemmas. By combining these lemmas, we find that the union  
	\begin{align} \label{eq:moduli trace E}
	&\MMv \cup \MMw \cup \MMac \cup 
 \notag \\
	&\qquad\qquad \cup \MMy \cup \MMad.
	\end{align}
constitutes the boundary of a 1-dimensional manifold. 
%

Recall that we wrote $y(\lambda,\mu) = \sum_{i} \alpha_i(\lambda,\mu) c_i$. For each index 1 Reeb chord of $c_i$, multiply $\alpha_i$ by the count of elements in the moduli spaces in \eqref{eq:moduli trace E}, with $c = c_i$. The sum of all such contributions must vanish, since \eqref{eq:moduli trace E} is a boundary for all $c$. Let us analyze the terms in \eqref{eq:moduli trace E} more carefully, in reverse order. The fact that $y(\lambda,\mu)$ is a cycle implies that the contributions coming from the last two terms in \eqref{eq:moduli trace E} must vanish on their own. 



We are left with contributions from the first three terms in \eqref{eq:moduli trace E}, which give 
 \begin{equation*} 
  E_\xi \cdot D+ D\cdot E_\eta =
(\partial_x F(1,e^p,1)) \cdot (\partial_{p} D),
 \end{equation*}
where signs are induced from the boundary orientation of the oriented 1-dimensional moduli space.

The term on the right comes from the $\MMac$, using again the description of the $\lambda$- and $\mu$-powers in $\mathcal A_K$ in Section \ref{choices and coeffs}. The appearance of the partial derivatives is justified in Remark \ref{rmk:derivatives}, as was argued for formula \eqref{q: partial derivative}. 
Multiplying the last equation by $D^{-1}$ and taking the trace, we get
\begin{equation} \label{Tr=}
 \Tr E_\xi + \Tr E_\eta = (\partial_x F(1,e^p,1)) \cdot \Tr \left(D^{-1}(\partial_{p} D) \right).
\end{equation}
Equation \eqref{dAn/dp+} now follows from combining equations \eqref{d An/dp=dtF/dxF+Tr} and \eqref{Tr=}. 

To get \eqref{Alex from F} from \eqref{dAn/dp+}, first note that Proposition \ref{partial lambda F neq 0}(2) implies that the denominator in the integrand is not identically zero. Then, using integral signs to denote antiderivatives as explained in Remark \ref{rmk:antiderivatives}
\begin{align*}
(1-e^{p}) &\exp\left(\int -\left(\frac{\partial_t F}{\partial_x F}\right)\Big|_{(1,e^p,1)} \, dp \right)= \\ 
&= (1-e^{p}) \exp\left(\int \partial_{p}\An_{M_K}(e^p) dp + \int \Tr\left(D^{-1} (\partial_{p}D)\right) dp \right) = \\
&= (1-e^{p}) \exp\left(\An_{M_K}(e^{p}) \right) \cdot \exp \left( \int \partial_{p} (\log \det D(e^p)) \, dp \right) = \\
&= (1-e^{p}) \exp\left(\An_{M_K}(e^{p}) \right) \cdot \det D(e^{p}) = \\
&= \Delta_K(e^{p}).
\end{align*}
In the last identity, we used Theorem \ref{T:Alex zeta tau}.
\end{proof}



\begin{remark}
If $K$ is fibered, then $\Md$ can be made to not intersect the zero section in $T^*\R^3$, as explained in \cite{AENV}. In that case, the proof of Theorem \ref{Alex F general} would be simpler, since one would not need to take into account the Floer strips counted by the $D_{ij}$. 
\end{remark}

\begin{remark}
In the proof of Theorem \ref{Alex F general} we used Theorem \ref{T:Alex zeta tau}, which is a holomorphic curve reformulation of Equation \eqref{alex} where flow loops get replaced with annuli and flow lines get replaced with strips. One should be careful that the choices made to specify the $\mu$-powers for flow lines in $\tau_{\rm Morse}$ and for strips in $\tau_{\rm str}$ are compatible. The latter depend on the choices of paths $\gamma_{\xi_i}, \gamma_{\eta_j}$ to the basepoint and the corresponding disk $U$. Changing those paths leads to multiplication of the matrix $D$ on the left and on the right by diagonal matrices with diagonal entries which are powers of $\mu$. The effect of that on Theorem \ref{T:Alex zeta tau} is to multiply $\tau_{\rm str} = \det (D(\mu))$ by some power of $\mu$, which is compatible with the fact that $\Delta_K(\mu)$ is defined up to multiplication by such powers.  
\end{remark}

\subsection{Proof of Theorem \ref{T:Alex Aug}} \label{sec:proof of main thm}
To prove equation \eqref{Alex from Aug 2} in Theorem \ref{T:Alex Aug}, we need to show that we can replace $F$ with $\Aug_K$ in equation \eqref{Alex from F} in Theorem \ref{Alex F general}, 
if $\partial_\lambda \Aug_K|_{(\lambda,Q) = (1,1)} \neq 0$. Recall that, according to Proposition \ref{full V_K 2d}, $\tilde V_K$ can be parametrized by $(\lambda,\mu)$ near $\boldsymbol{\epsilon_0}$. We denote this parametrization $\epsilon_{\lambda,\mu}$ as before. Write $Q(\lambda,\mu)$ for the $Q$-component of $\epsilon_{\lambda,\mu}$ and consider the graph of $Q(\lambda,\mu)$:
$$
V\coloneq \{(\lambda,\mu,Q(\lambda,\mu))\}\subset (\C^*)^3.
$$
Observe that $V$ is the image under $\pi$ of $\tilde V_K$ near $\boldsymbol{\epsilon_0}$. 
Note that 
$$
F(\lambda,\mu,Q(\lambda,\mu)) = \epsilon_{\lambda,\mu} (\partial (y(\lambda, \mu))) = 0
$$
(because $\epsilon_{\lambda,\mu}$ is an augmentation), so $F|_V \equiv 0$. Also, writing $\boldsymbol{1} = (1,1,1)$, Equation \eqref{partial Q F neq 0} implies that $(\nabla F)|_{\boldsymbol{1}} \neq 0$, so $V$ is locally cut out by $F$ near $\boldsymbol{1}$. 

We next show that $V \subset V_K$. If $\Aug_K$ were the zero polynomial for some knot $K$ (which is conjecturally never the case, as recalled in Remark \ref{NgConjecture}), then $V_K = (\C^*)^3$ and $V\subset V_{K}$. On the other hand, if $\Aug_K$ is not identically zero, then $V_K$ is 2-dimensional. But since $\tilde V_K$ is 2-dimensional and parametrized by $(\lambda,\mu)$ near the smooth point $\boldsymbol{\epsilon_0}$, the image under $\pi$ of the irreducible component of $\tilde V_K$ through $\boldsymbol{\epsilon_0}$ is also 2-dimensional, so it is contained in $V_K$. This implies that $V\subset V_K$. 

The inclusion $V\subset V_{K}$ then implies that along $V$ we have $d\Aug_K = h \cdot (dF)$ for some holomorphic function $h$ defined near $\boldsymbol{1}$. Therefore, if $\partial_\lambda \Aug_K\neq 0$ then $h\neq 0$ and we get
$$
 \frac{\partial_Q \Aug_K}{\partial_\lambda \Aug_K}  =  \frac{h\cdot (\partial_Q F)}{h \cdot (\partial_\lambda F)}  =  \frac{\partial_Q F}{\partial_\lambda F}.
$$
Theorem \ref{T:Alex Aug} then follows from Theorem \ref{Alex F general}.\qed

\begin{remark} \label{Alex from branch}
According to formula \eqref{Alex from F}, one can obtain $\Delta_K$ from the function $F$ for every knot $K$. The proof of Theorem \ref{T:Alex Aug} above shows that $F$ cuts out a smooth surface $V\subset V_K$ near $\boldsymbol 1$. If $V_K$ is 2-dimensional (as is conjectured for every $K$), then $V$ is locally a branch of $V_K$. Thus, even when formula \eqref{Alex from Aug 2} does not hold, we can say that it makes sense along a suitable branch of the augmentation variety. 
\end{remark}

\begin{remark} \label{d mu F = 0}
It is a general fact that $(\partial_\mu F)|_{\lambda = Q = 1} \equiv 0$. This is because
$$
\partial_{\mu} F(1,\mu,1) = \partial_{\mu}\big(\epsilon_{1,\mu} (\partial (y(1, \mu))) \big) = 0, 
$$
since $\epsilon_{1,\mu} \circ \partial = 0$. 
The inclusion $V\subset V_{K}$ then implies that $(\partial_\mu \Aug_K)|_{\lambda = Q = 1} \equiv 0$ for every knot $K$. 
\end{remark}

\subsection{Independence of choices}

\label{sec: independence of choices}

Writing the dg-algebra $\mathcal A_K$ with coefficients in the ring $R=\C[\lambda^{\pm 1},\mu^{\pm1},Q^{\pm 1}]$ required 
making some choices, as discussed in Section \ref{choices and coeffs}. We now study the (in)dependence of equation \eqref{Alex from F}, and hence of \eqref{Alex from Aug 2}, on these choices. 

\begin{itemize}
 \item {\bf Capping half-disks for Reeb chords:}
Two choices of capping half-disks for a Reeb chord $c$ differ by an element of $\pi_2(ST^*\R^3,\Lambda_K) \cong H_2(ST^*\R^3,\Lambda_K;\Z)$. If $c$ is of degree different than 1, then $F$ is not altered by this change in capping disks. If $c$ is of degree 1 and is a term in the cycle $y$ in Definition \ref{function F}, then $y$ can be suitably altered so that $F$ (and hence also formula \eqref{Alex from F}) does not change. 
\item {\bf Meridian in $\Lambda_K$:} A change of meridian curve for $K$ (which we also called a change of framing \eqref{change framing}) would yield a change of variables of the form 
 $(\lambda,\mu,Q) \mapsto (\lambda, \lambda^{k}\mu,Q)$, for some $k\in \Z$. 
Such a change would not affect the 
 numerator on the integrand in \eqref{Alex from F}. The denominator would change to 
 $$
 \epsilon_{M_K}(\partial_\lambda F + k  \lambda^{k-1} \mu \partial_\mu F).
 $$
 But $\epsilon_{M_K}(\partial_\mu F)=0$, as we saw in Remark \ref{d mu F = 0}, so the integrand in \eqref{Alex from F} remains unchanged.  
Note also that changing the orientation of $K$ correponds to 
$(\lambda,\mu,Q) \mapsto (\lambda^{-1},\mu^{- 1},Q)$ in $F$ (the direction of $p$ changes, to preserve $x\cdot p = 1$). This changes $\Delta_K(\mu)$ to $\Delta_K(\mu^{-1})$. It is well-known that the Alexander polynomial is invariant under this operation (see for instance \cite{SeifertAlexander,MilnorDuality}), which is compatible with the fact that $\Delta_K(\mu)$ does not depend on the orientation of $K$. 
\item {\bf Capping disks for longitude and meridian:}  
Recall that this is equivalent to a choice of splitting of the short exact sequence \eqref{splitting ses}. A change in that splitting corresponds to a change of variables 
$(\lambda,\mu,Q) \mapsto (\lambda Q^l,\mu Q^m,Q)$, for some $l,m\in \Z$. The integrand in \eqref{Alex from F} changes to 
$$
\epsilon_{M_K}\left(-l- m \, \frac{\partial_{\mu} F}{\partial_\lambda F}-\frac{\partial_Q F}{\partial_\lambda F}\right) = -l-\epsilon_{M_K}\left(\frac{\partial_Q F}{\partial_\lambda F}\right),
$$
where we used again the fact that $\epsilon_{M_K}(\partial_\mu F)=0$. 
This change has the effect of multiplying $\Delta_K(\mu)$ by $\mu^{-l}$, 
and it is compatible with the fact that the Alexander polynomial is defined up to multiplication by a power of $\mu$. 
\end{itemize}

\section{Examples}\label{sec Ex}

\subsection{Computing $\Aug_K$} \label{elimination}
In many examples, the augmentation polynomial $\Aug_K$ can be computed using elimination theory, as follows. Recall that $R=\C[\lambda^{\pm1},\mu^{\pm1},Q^{\pm1}]$, $\tilde R=R[\boldsymbol{a}]$ where $\boldsymbol a = (a_i)$ is the ordered list of degree 0 Reeb chords, and $\tilde I_K$ is the ideal generated by the differentials of degree 1 Reeb chords. Recall that $\tilde V_K = V(\tilde I_K)\subset (\C^*)^3\times \C^{\boldsymbol a}$ is the full augmentation variety and that $V_K\subset (\C^*)^3$ is the augmentation variety. If $I_K \coloneq R\cap \tilde I_K$, then $V_K \subset \mathrm{Cl}(\pi(\tilde V_K)) \subset V(I_K)$, where $\mathrm{Cl}$ denotes Zariski-closure. Hence, $I(V_K) \supset I(V(I_K)) = \sqrt{I_K}$, by Hilbert's Nullstellensatz. Therefore, if $I_K \neq \{0\}$ then $\Aug_K$ cannot be $0$. 

Corollary \ref{Aug exists} implies that $V_K\subset (\C^{\ast})^3$ is either 2- or 3-dimensional. Hence, $V_K\subset V(I_K)$ is contained in the union of maximal-dimensional irreducible components of $V(I_K)$. Denote this union by $X_K \subset (\C^*)^3$. Since the dimension of $X_K$ is 2 or 3, 
each irreducible component of $X_K$ is the zero set of some irreducible polynomial, 
and hence $I(X_K)$ is a principal ideal. If $f$ is a generator of this ideal, then $\Aug_K$ must divide $f$. If we further knew that $V_K = V(I_K)$ (or equivalently, that $I(V_K) = \sqrt{I_K}$), then we could conclude that $\Aug_K = f$. 

The polynomial $f$ can in principle be explicitly computed as follows. From a Gr\"obner basis for $\tilde I_K$, one can obtain a Gr\"obner basis for $I_K$. 
Then one can find generators for $\sqrt{I_K}$ 
and obtain a minimal decomposition of $\sqrt{I_K}$. 
A generator $f$ of $I(X_K)$ can be obtained from such a decomposition. 

In the case when we know that $I_K$ is a principal ideal, then there is a potentially simpler computational approach to determining $f$ using the following fact: $I_K$ is a principal ideal if and only if its reduced Gr\"obner basis has a single element $g$ (which generates $I_K$). A generator $f$ of $\sqrt{I_K}$ would be given by a \emph{reduction} of $g$, 
(if $g=u g_1^{a_1}\ldots g_r^{a_r}$ is a factorization into irreducible polynomials, then a reduction of $g$ is $g_1 \ldots g_r$). 



\subsection{The trefoil}
We represent the trefoil by the 2-strand braid $\sigma_1^{-3}$ (augmentation polynomials of a knot and its mirror are related by a change of variables $(\lambda,\mu,Q)\mapsto (\lambda Q^{-1},\mu^{-1},Q^{-1})$ for certain capping disks, etc). 
In the flow tree model, its Legendrian DGA is generated in degree 0 by orbits $a_{1,2}, a_{2,1}$, in degree 1 by 
$b_{1,2}, b_{2,1}, c_{1,1}$, $c_{1,2}, c_{2,1}, c_{2,2}$ and in degree 2 by $e_{1,1}, e_{1,2}, e_{2,1}, e_{2,2}$. The 
differential in degree 1 is
\begin{align*}
\begin{cases}
\partial c_{1,1} &= \lambda \mu^{-2}-\lambda \mu^{-3} - (2Q - \mu) a_{1,2} - Q a_{1,2}^2 a_{2,1} \\ 
\partial c_{1,2} &= Q-\mu+\mu a_{1,2} + Q a_{1,2} a_{2,1} \\
\partial c_{2,1} &= Q-\mu + \lambda \mu^{-2} a_{2,1} + Q a_{1,2} a_{2,1} \\
\partial c_{2,2} &= \mu - 1 - Q a_{2,1} + \mu a_{1,2}a_{2,1} \\
\partial b_{1,2} &= \lambda^{-1} \mu^3 a_{1,2} - a_{2,1} \\
\partial b_{2,1} &= - a_{1,2} + \lambda \mu^{-3} a_{2,1}
\end{cases}
\end{align*}
The differentials of four of these chords are explained in detail in \cite[Section 7.2]{EkholmNgHigherGenus}, and the differentials of the other two chords are obtained in a similar manner. Note that we applied the transformation 
$\lambda \mapsto \lambda \mu^{-3}$ to the formulas in \cite{EkholmNgHigherGenus}, since in our conventions the $x$ curve must be null-homologous in 
$\R^3\setminus K$. On chords of degree 0, we get 
$$
\epsilon_{M_K}(a_{1,2}) = \mu^{-2}(\mu-1), \qquad \epsilon_{M_K}(a_{2,1}) = \mu(\mu-1).
$$

One can then compute the $\epsilon_{M_K}$-linearized differential, and see that its kernel in degree 1 is spanned by 
\begin{align*}
y_1 &= \mu^2(2-\mu)\, c_{2,1} + \mu(2\mu-1) \, c_{2,2} +(1-\mu^2) \, b_{1,2}\\
y_2 &= \mu^{4}(2\mu -1)\, c_{1,1} + \mu^2(3-4\mu + 2\mu^2) \, c_{2,1} + (1-\mu+\mu^2)\, b_{1,2} \\
y_3 &=  c_{1,2} - c_{2,1} - \lambda\mu^{-2} b_{1,2} \\
y_4 &= \lambda\mu^{-3} b_{1,2} + b_{2,1} 
\end{align*}

Writing $F_1 = (\partial y_1)_{\boldsymbol{a} \mapsto \epsilon_{M_K}(\boldsymbol{a})}$, we can compute
\begin{align*}
(\partial_\lambda F_1)|_{(\lambda,Q)=(1,1)} &= \mu(\mu-1)(\mu^2-\mu +1)\\
(\partial_Q F_1)|_{(\lambda,Q)=(1,1)} &= \mu(2-4\mu+6\mu^2 -3\mu^3)
\end{align*}
and the right side of formula \eqref{Alex from F} is $\mu^2(1-\mu+\mu^2)$, which is the Alexander polynomial of the trefoil knot (up to the usual ambiguity of a power of $\mu$). 
Using $y_2$ instead of $y_1$ would also yield the Alexander polynomial of the trefoil. If we used $y_3$ or $y_4$, then the integrand on the right side of \eqref{Alex from F} would be $\frac{0}{0}$, so these cycles would not be suitable for computing the Alexander polynomial.

From the DGA differential of chords of degree 1, one can also obtain the augmentation polynomial of the right-handed trefoil, which is 
$$
\Aug_K(\lambda, \mu, Q)=\lambda^2(\mu - 1) + \lambda(\mu^4-\mu^3 Q + 2 \mu^2 Q^2 - 2 \mu^2 Q - \mu Q^2 + Q^2) + (\mu^3 Q^4-\mu^4 Q^3).
$$
See again \cite[Section 7.2]{EkholmNgHigherGenus} for details. Formula \eqref{Alex from Aug 2} can then also be verified for the trefoil. 

\subsection{Other examples}
One can find a {\em Mathematica} notebook with augmentation polynomials of many knots, up to 10 crossings, on Lenhard Ng's webpage. 
Equation \eqref{Alex from Aug 2} gives the Alexander polynomial for all of these examples, {\em except} the ${8_{20}}$ knot and a connected sum of a left-handed trefoil with a right-handed trefoil. In both of these cases, $\partial_x \Aug_K|_{(\lambda,Q)=(1,1)} = \partial_Q \Aug_K|_{(\lambda,Q)=(1,1)} = 0$. Recall from Remark \ref{Alex from branch} that for a general knot with 2-dimensional augmentation variety, one can obtain the Alexander polynomial from a suitable branch of the $V_K$. In the case of the connected sum of a left-handed and a right-handed trefoil, one can understand geometrically the origin of different branches containing the line $\{(1,\mu,1)\}$. 

For this, recall that the augmentation variety restricted to $Q=1$ behaves well under connected sums, as explained in \cite[Proposition 5.8]{NgFramed}%
\footnote{We are assuming that $V_K|_{Q=1}$ is equal to the 2-variable augmentation variety (the Zariski closure in the $(\lambda,\mu)$-plane of the points $(\lambda_0,\mu_0,1)$ which can be lifted to an augmentation of $\mathcal A_K$), which is true in these examples but not known for a general knot.} 
: 
\begin{equation} \label{V sum}
V_{K_1\#K_2}|_{Q=1} = \{(\lambda_1 \lambda_2,\mu) | (\lambda_i,\mu,1)\in V_{K_i} \text{ for } i=1,2\}
\end{equation}
The 2-variable augmentation polynomial for the left-handed trefoil is 
$$
(\lambda-1)(\mu-1)(1+\lambda\mu^3).
$$
We have three branches: $\{\lambda = 1\}$ and $\{\mu=1\}$, corresponding to the fillings $M_K$ and $L_K$, respectively, as well as $\{\lambda = - \mu^{-3}\}$. The 2-variable augmentation variety for the right-handed trefoil is 
$$
(\lambda-1)(\mu-1)(\lambda+\mu^3).
$$
The branches are now $\{\lambda = 1\}$, $\{\mu=1\}$ and $\{\lambda = - \mu^{3}\}$. Equation \eqref{V sum} shows that the third branches of the 2-variable augmentation varieties of the left- and right-handed trefoils combine to give a new branch covering the line $\lambda = 1$ in the 2-variable augmentation variety of the connected sum. Allowing $Q\neq 1$, we get a new branch of $V_K$ containing the line $(1,\mu,1)$, as wanted.

\section{Disk potentials, SFT-stretching, and a deformation of the Alexander polynomial} \label{sec:SFT-stretch}
In this section we discuss connections between the study undertaken in this paper and more physically inspired treatments of knot theory, in particular open topological strings and Gromov--Witten theory.
We will see that our results suggest a natural $Q$-deformation, $Q=e^t$, of the Alexander polynomial, and that this is related to disk potentials and Floer torsion. We will also observe that the higher genus counterparts give a geometric framework for  invariants introduced in \cite{GPV,GPPV,GM}, as outlined in \cite{EGGKPS}. Our discussion in this section does not contain full proofs of all the transversality and gluing theorems needed. The perturbation scheme for counting bare curves, see \cite{EkholmShende}, suffices for this purpose, but we will not give a detailed treatment here. Nevertheless, the results discussed give important information about the SFT-stretched limit and serve as a starting point for connections to physically inspired invariants in low-dimensional topology. 

We will use the following result which is a consequence of Lemma \ref{LK MK augment}. Here a generalized holomorphic disks is a tree of finitely many holomorphic disks with boundary on a Lagrangian, where the edges correspond to intersections with suitably chosen bounding chains. For details, see \cite{AENV, EkholmNgHigherGenus}. 

\begin{lemma}\label{2.3 iii}
With notation as in Lemma \ref{LK MK augment} the following holds. For each Reeb chord $a$ of degree 0, the space of generalized holomorphic disks on $L^{\delta}$ with a positive puncture asymptotic to $a$ is canonically isomorphic to the corresponding space for $L^{0}$ (which is simply the space of ordinary holomorphic disks). It follows in particular that $L^{\delta}$ and $L^{0}$ define identical augmentations.  
\end{lemma}

\begin{proof}
This follows from Lemma \ref{LK MK augment}: generalized holomorphic disks are combinations of one disk with positive puncture and several disks without punctures. Since there are no disks without punctures and the disks with punctures are naturally identified as $\delta$ varies between 0 and $\delta_0$, the statement follows. 
\end{proof}

\subsection{Basic disk potentials and SFT-stretching}

Given \eqref{Alex from F}, we could define a $Q$-deformation, $Q=e^{t}$, of the Alexander polynomial by taking
\begin{align} \label{deform Alex}
\Delta_K(e^p,e^t) &= (1-e^p) \exp \left( \int - \left( \frac{\partial_t F}{\partial_x F} \right) \Big|_{\{x=0\}} dp \right) = \\ \notag
&= (1-e^p) \exp \left( \int - \left( \frac{\partial_t \Aug_K}{\partial_x Aug_K} \right) \Big|_{\{x=0\}} dp \right),
\end{align}
where we can write the second equality when $\partial_x \Aug_K$ does not vanish. We will see how this deformation is related to disk potentials and Floer torsion. 

Consider first the disk potential for $L=L_{K}^{\delta}$, for $\delta>0$. Here $L\cap\R^{3}=\varnothing$ and we consider SFT-stretching along an $\epsilon$-sphere bundle $S_{\epsilon}T^{\ast}\R^{3}$ for $\epsilon$ small compared to $\delta$, as in \cite{EkholmShende}. Under such stretching all holomorphic curves with boundary on $L\cup\R^{3}$ leave $\R^{3}$, since there are no closed Reeb orbits in $ST^{\ast}\R^{3}$. Thus, after stretching the curves represent relative homology classes in $H_{2}(T^{\ast}\R^{3}\setminus\R^{3},L)$, which is generated by $x$ and $t$ (recall that the meridian $p$ is a boundary in $\Ld$). In particular, counting generalized holomorphic disks with boundary on $L$ in a sufficiently stretched almost complex structure gives a Gromov--Witten disk potential 
\[ 
W_{K}^0({x},t) = \sum_{n>0} C_{n}(e^t)e^{nx},
\] 
where $C_{n}(e^t)$ is the count of curves with boundary in the class of $n$ times the generator of $H_1(L)$. We then have the following local parameterization of a branch of the augmentation variety:
\begin{equation} \label{p=dW/dx}
p=\frac{\partial W_{K}^0}{\partial x},
\end{equation}
see \cite{AENV} and \cite{EkICM}. (One can also think of the SFT-stretched curve counts as taking place in the resolved conifold, after conifold transition.) 

We next show that we can define a disk potential in a similiar way also for $L=M_{K}^{\delta}$, although $\Md$ may intersect $\R^{3}$. To this end, we first slightly deform $M_{K}^{\delta}$ so that its intersection with some neighborhood of the 0-section coincides with a finite number of cotangent fibers. We then SFT-stretch in a smaller $\epsilon$-neighborhood. In this case the situation is less straightforward. However, the geometry in the negative end after stretching is easy to understand. Using the flat metric we have Reeb chords $\gamma_{ij}$ between any two distinct fibers $F_i$, $F_j$ and no other Reeb chords. Furthermore, the Lagrangian $L$ in the negative end consists of fibers $F_{i}$.

\begin{lemma}
The dimension of any moduli space of holomorphic curves with boundary on $F=\bigcup F_{i}$ and $m$ positive punctures equals $m$.
\end{lemma}
  
\begin{proof}
The Conley--Zehnder indices of the Reeb chords are 0. The dimension formula then follows from \cite{ECL}.	
\end{proof}

Using this formula we find that $M_{K}^{\delta}$ defines a basic disk potential that is invariant under deformations.

\begin{figure}
  \begin{center}
    \def\svgwidth{.2\textwidth}
\begingroup%
  \makeatletter%
  \providecommand\color[2][]{%
    \errmessage{(Inkscape) Color is used for the text in Inkscape, but the package 'color.sty' is not loaded}%
    \renewcommand\color[2][]{}%
  }%
  \providecommand\transparent[1]{%
    \errmessage{(Inkscape) Transparency is used (non-zero) for the text in Inkscape, but the package 'transparent.sty' is not loaded}%
    \renewcommand\transparent[1]{}%
  }%
  \providecommand\rotatebox[2]{#2}%
  \newcommand*\fsize{\dimexpr\f@size pt\relax}%
  \newcommand*\lineheight[1]{\fontsize{\fsize}{#1\fsize}\selectfont}%
  \ifx\svgwidth\undefined%
    \setlength{\unitlength}{120.30396242bp}%
    \ifx\svgscale\undefined%
      \relax%
    \else%
      \setlength{\unitlength}{\unitlength * \real{\svgscale}}%
    \fi%
  \else%
    \setlength{\unitlength}{\svgwidth}%
  \fi%
  \global\let\svgwidth\undefined%
  \global\let\svgscale\undefined%
  \makeatother%
  \begin{picture}(1,0.58548281)%
    \lineheight{1}%
    \setlength\tabcolsep{0pt}%
    \put(0,0){\includegraphics[width=\unitlength,page=1]{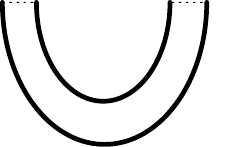}}%
    \put(0.21142912,0.4053761){\color[rgb]{0,0,0}\makebox(0,0)[lt]{\lineheight{0}\smash{\begin{tabular}[t]{l}$F_i$\end{tabular}}}}%
    \put(0.69769473,0.03132492){\color[rgb]{0,0,0}\makebox(0,0)[lt]{\lineheight{0}\smash{\begin{tabular}[t]{l}$F_j$\end{tabular}}}}%
    \put(0,0){\includegraphics[width=\unitlength,page=2]{strip.pdf}}%
  \end{picture}%
\endgroup%

  \end{center}
  \caption{Strip with two positive ends and boundary on two cotangent fibers $F_i$ and $F_j$}
  \label{strip_fig}
\end{figure}

\begin{proposition}\label{l:diskpotentialMK}
For a sufficiently stretched almost complex structure, all holomorphic disks with boundary on $M_{K}^{\delta}$ lie outside the $\epsilon$-neighborhood. Define the corresponding basic disk potential 
\[ 
U_{K}^0(p,t)= \sum_{n>0} B_{n}(e^t)e^{np},
\]	
counting generalized holomorphic disks. This $U_{K}^0$ is invariant under deformations and 
\begin{equation} \label{x=dU/dp}
x=\frac{\partial U_{K}^0}{\partial p}
\end{equation}
gives a local branch of the augmentation variety.
\end{proposition}
\begin{proof}
We need to show that under stretching there are no broken disks with lower level in the negative end. The minimal dimension of such a lower level is 2, corresponding to a strip with two positive ends, see Figure \ref{strip_fig}. Since it came from a rigid disk, the strip is capped at an upper level by two punctured disks (the cappings might in principle consist of more complicated configurations, which would also contain components of negative index), and the sum of their expected dimensions is $-2$, hence they do not exist generically. It follows that no disk can lie in the neighborhood and that the count is invariant in 1-parameter families (degenerations that do not involve the negative end are treated as usual, see \cite{AENV}). The last statement then follows exactly as in \cite[Section 6.9]{AENV}.	
\end{proof}	

\begin{remark}
The basic potential function $U_K^0$ was constructed in \cite{AENV} for fibered knots. Proposition \ref{l:diskpotentialMK} extends the construction of $U_K^0$ to an arbitrary knot.  
\end{remark}

Recall now the integral in \eqref{Alex from F}. As we observe in Remark \ref{Alex from branch}, the function $F$ locally cuts out a branch of $V_K$, along which we can write  
\[ 
\int- \frac{\partial_{t}F}{\partial_{x} F} 
dp = \int \partial_{t}x(t,p) 
dp.
\]
If this branch was cut out by \eqref{x=dU/dp}, then
\begin{equation} \label{eq:Alex U}
 \int \partial_{t}x(t,p) 
dp  = \int \partial_{t}\partial_{p} U_{K}^{0}  
dp =  \partial_{t} U_{K}^{0},  
\end{equation}
hence in this case the $Q$-deformation of the Alexander polynomial \eqref{deform Alex} could be written as   
\[ 
\Delta_K(e^p,e^t) = (1-e^p) \exp\left(\partial_{t}U_{K}^{0}\right).
\]
We will see in Section \ref{sec:defAlexfiber} that there is strong evidence that this formula holds for fibered knots, but not for every knot. For general $K$, one must also take into account disks with negative punctures at the Reeb chords $\gamma_{ij}$ mentioned above, and the disk potential for $M_K$ is not unique.

\begin{remark}
As we just saw, if the two branches of $V_K$ coincide, we have 
\[ 
\Delta_K(e^p) = (1-e^p) \lim_{t\to 0}\exp\left(\partial_{t}U_{K}^{0}\big|_{t=0}\right).
\]
One might wonder how much information about $U_K^0(p,t)$ is known by $\Delta_K(\mu)$. Replacing $M_K$ and $U_K^0$ with $L_K$ and $W_K^0$ suggests that much information is lost in the limit. Using equations \eqref{d Q (F-G)} and \eqref{d mu (F-G)} and 
Proposition \ref{partial mu G partial Q G}, the result of exchanging $x$ and $p$ in Theorem \ref{T:Alex Aug} is
\begin{align}
(1-e^x) &\exp \left( \int - \left(\frac{\partial_t F}{\partial_p F}\right) \Big|_{(e^x,1,1)} dx \right) = (1-e^x) \exp \left( \int \frac{e^x}{1-e^x} dx \right) = 1
\end{align}
for every knot $K$, and this expression remembers nothing about the disk potential $W_K^0$, which is indeed very different for different knots. 
\end{remark}

\subsection{Invariance of Floer torsion}
A first sign that input from extra negative punctures at the Reeb chords $\gamma_{ij}$ may be necessary comes from the Floer cohomological torsion of $CF^*(\R^3,\Md)$, defined as the product 
\begin{equation}\label{eq:Floertorsion}
\tau_{K}(e^{p},e^{t})=\zeta_{\rm{an}}(e^{p},e^{t})\cdot \tau_{\rm{str}}(e^{p},e^{t}).
\end{equation}
Here, $\zeta_{\rm{an}}(e^{p},e^{t})$ is a count of all disconnected generalized holomorphic annuli from $\R^3$ to $\Md$ in $T^{\ast}\R^3$. Since there are no closed Reeb orbits in $ST^*\R^3$, generalized annuli are (trees of) disks with an intersection with the 4-chain $V$ from \eqref{4 chain} in the SFT-stretched limit, hence 
\begin{equation} \label{eq:zeta U}
\zeta_{\rm{an}}(e^{p},e^{t})=\exp(\partial_{t} U_{K}^{0}(p,t)).
\end{equation}

On the other hand, $\tau_{\rm{str}}(e^{p},e^{t})=\det(D)$, where $D$ is the Floer differential in the Floer cohomology complex $CF^{\ast}(\R^3,\Md)$ with Novikov coefficients. This differential counts two-level curves in the stretched limit, with the lower level being a trivial strip over a straight line. We show in Proposition \ref{prp:invtauK} that $\tau_{K}$ is invariant under deformations. 
As discussed in Section \ref{sec:defAlexnonfiber} the study of $\tau_K$ leads to a more general notion of disk potential for $M_{K}$, counting also disks with negative punctures at the Reeb chords $\gamma_{ij}$. We denote such generalized disk potentials by $U_{K}^{0,\epsilon}$. 

\begin{remark}
To compute the powers of $e^t$ in counts of disks and annuli, we need to define an intersection number with the $4$-chain $V$. To force such intersections into the interior of the curves, we must specify a way to push the boundaries of the disks in $\R^{3}$ off of $\R^{3}$ in $T^{\ast}\R^{3}$, compare \cite[Section 2]{EkholmNgHigherGenus}. Also, the 4-chain intersects $M_{K}$ so there will be contributions to $e^{t}$ from linking with this locus, and defining them needs additional choices of capping paths. We will not go into the details of making all these choices here. 
\end{remark}

\begin{remark}
In view of \cite{EkholmNgHigherGenus,EkICM}, it is natural to consider the all genus counterpart $U_{K}^{\epsilon}(p,t,g_{s})$ of the disk potential:
\[ 
U_{K}^{\epsilon}(p,t,g_{s})=g_{s}^{-1}U_{K}^{0,\epsilon}(p,t)+U_{K}^{1,\epsilon}(p,t)+\dots+ g_{s}^{-r}U_{K}^{r+1,\epsilon}(p,t)+\dots,
\]
where $U_{K}^{r+1,\epsilon}$ counts generalized holomorphic curves of Euler characteristic $\chi=-r$, possibly with negative punctures.
The Lagrangians $M_{K}$ and $L_{K}$ share ideal contact boundary $\Lambda_{K}$. The Legendrian SFT of $\Lambda_{K}$ determines an operator $\widehat{\mathcal{A}}_{K}(e^{\hat p},e^{\hat x},e^{t},e^{g_{s}})$, where the operator $e^{\hat p}$ acts as multiplication by $e^{p}$ and $e^{\hat x}$ acts as $e^{-g_{s}\partial_{p}}$. In particular, $e^{\hat p}e^{\hat x}=e^{g_{s}}e^{\hat x}e^{\hat p}$. This is a quantization of the augmentation polynomial $\Aug_{K}$ and, through holomorphic curve counting for $L_{K}$, $\widehat{\mathcal{A}}_{K}(e^{\hat p},e^{\hat x},e^{t},e^{g_{s}})=0$ is the recursion relation for the colored HOMFLY-polynomial. In view of this, one expects from SFT that $\widehat{\mathcal{A}}_{K}\Psi_{K}=0$, where $\Psi_{K}=\exp\left(U_{K}^{\epsilon}\right)$. This observation together with Theorem \ref{T:Alex Aug}, which in this context gives the semi-classical limit of the un-normalized expectation value $\exp\left(Ng_{s}\partial_{t}\right)\Psi_{K}(p)$, lead to a large $N$ version, or $Q$-deformation, of the invariant $\widehat{Z}$, \cite{GPV,GPPV,GM}, for knot complements, see \cite{EGGKPS}.
\end{remark}

Our next result will be used to show that the basic disk potential $U_K^0$ gives in general a different branch of the augmentation variety than that associated to the Alexander polynomial. We will write $m$ for a meridian curve representing the generator $p$ of $H_1(M_K)$ and $\lambda_{\rm std} = \sum_i p_i dq_i$ for the standard Liouville form in $T^*\R^3$. 	

\begin{proposition}\label{prp:invtauK}
The torsion $\tau_{K}$ is invariant under deformations (changing $J$, Lagrangian isotopies of $\Md$, etc.) for which $\int_{m}\lambda_{\rm std}$ remains positive. 
\end{proposition}

\begin{proof}
We want to show that $\tau_{K}(p,t)$ is invariant under deformations.
Consider a generic 1-parameter family of geometric data with a perturbation scheme in which output punctures in several level holomorphic disks buildings are time-ordered. Using the positivity hypothesis, we can show the invariance of $\tau_{K}$ by studying codimension one phenomena associated to the presence of intersection points. Bifurcations of Floer strips are associated to four types of curves (see Figure \ref{bdry8_fig}): 
\begin{itemize}
	\item A disk of virtual dimension $(-1)$ in $\mathcal{M}(\xi_{i},\xi_{j})$, $\xi_{i}\ne \xi_{j}$.
	\item A disk of virtual dimension $(-1)$ in $\mathcal{M}(\eta_{i},\eta_{j})$, $\eta_{i}\ne \eta_{j}$.
	\item A disk of virtual dimension $(-1)$ in $\mathcal{M}(\xi_{i},\xi_{i})$.
	\item A disk of virtual dimension $(-1)$ in $\mathcal{M}(\eta_{i},\eta_{i})$. 
\end{itemize} 
We note that the last two $(-1)$ curves also arise as the limits of pinched holomorphic annuli, see Figure \ref{bdry9_fig}. 
\begin{figure}
	\begin{center}
		\def\svgwidth{.85\textwidth}
		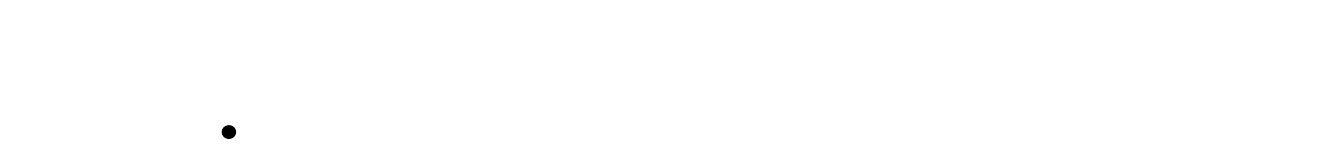
	\end{center}
	\caption{Splitting of disks of virtual dimension $-1$ from Floer strips}
	\label{bdry8_fig}
\end{figure}

\begin{figure}
	\begin{center}
		\def\svgwidth{.6\textwidth}
		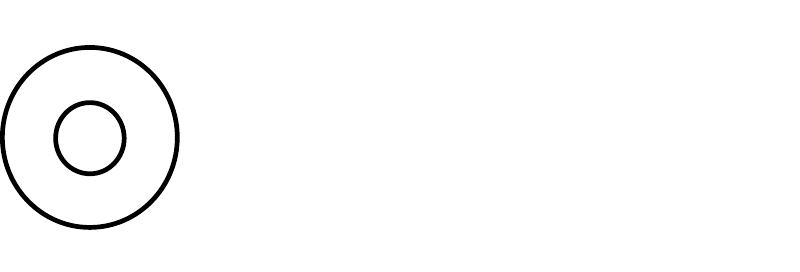
	\end{center}
	\caption{Pinching of holomorphic annuli into virtual dimension $-1$ strips}
	\label{bdry9_fig}
\end{figure}

To see that $\tau_{K}$ remains invariant under the first two $(-1)$-instances, we note that it is well-known how they affect the differential $D(e^{p},e^{t})$: by row or column operations. This leaves $\det(D)$ fixed and hence $\tau_{K}$ does not change.

Consider next a $(-1)$-disk $\delta$ in $\mathcal{M}(\xi_{i},\xi_{i})$. Here the differential changes by adding this disk to $\xi$. This means that the determinant is multiplied by $(1+\delta)$. On the other hand, the contribution to $\tau_{K}$ from the annulus that pinched giving $\delta$ is multiplication by $(1+\delta)$ as well (consider all disconnected curve configurations without and with this annulus). Thus the factor $1+\delta$ just moves between the annulus count and the determinant, leaving $\tau_{K}$ unchanged.
\end{proof}
 
\begin{remark}
Note that the deformation invariance of $U_K^0$ in Proposition \ref{l:diskpotentialMK} and of $\tau_K$ in Proposition \ref{prp:invtauK}, combined with \eqref{eq:zeta U}, imply that in \eqref{eq:Floertorsion} both factors are invariant for sufficiently SFT-stretched almost complex structures. Another way to see this is to observe that a $(-1)$-disk where they would change would have to limit to a Reeb chord connecting a fiber to itself, and there are no such Reeb chords in $ST^{\ast}\R^{3}$. 
\end{remark}

\subsection{Deformation of the Alexander polynomial for fibered knots}\label{sec:defAlexfiber}
Let $K$ be a fibered knot. Then there is a function $f\colon M_{K} \to S^1$ without critical points, and $df$ can be used to shift $M_{K}$ off of $\R^{3}$ in $T^{\ast}\R^{3}$. This means that the Floer torsion of $M_{K}$ can be expressed as
\[ 
\tau_{K}=\zeta_{\rm{an}}(e^{p},e^{t}),
\]  
since there are no strips. 

\begin{proposition}\label{prp:limitfibered}
If $K$ is a fibered knot and $U_{K}^{0}$ denotes the basic disk potential of $M_{K}$ then 
\begin{equation}\label{eq:defAlexfiber}
\Delta_{K}(e^{p}) = (1-e^{p}) \lim_{t\to 0}\,\exp\left(\partial_{t} U_{K}^{0}(p,t)\right).
\end{equation}
\end{proposition}

\begin{proof}
Combining \eqref{eq:zeta U} with Proposition \ref{prp:invtauK}, Lemma \ref{annuli loops}, and \eqref{alex}, we have
\[ 
\lim_{t\to 0} \exp(\partial_{t} U_{K}^{0}(p,t)) = \lim_{t\to 0}\zeta_{\rm{an}}(e^{p},e^{t}) = \zeta_{\rm{loop}}(e^{p})=(1-e^{p})^{-1}\Delta_{K}(e^{p}),
\]
as wanted.
\end{proof}

This result, combined with \eqref{eq:Alex U} above, implies that the branch of the augmentation variety cut out by $U_K^0$ in \eqref{x=dU/dp} is tangent along the line $(x,t)=(0,0)$ to the branch cut out by the function $F$ in Remark \ref{Alex from branch}. If these branches coincided, then we would conclude the identity of $Q$-deformations of the Alexander polynomial
$$
(1-e^p) \tau_K(e^p,e^t) = \Delta(e^p,e^t),
$$
where the left side is defined in \eqref{eq:Floertorsion} and the right side is defined in \eqref{deform Alex}. It seems likely that this may hold for every knot $K$.

\subsection{Deformation of the Alexander polynomial for non-fibered knots}\label{sec:defAlexnonfiber}
In this section we consider the geometry of deformations of the Alexander polynomial for non-fibered knots. This is more complicated than the case of fibered knots and involves counts of curves with negative boundary punctures, which is not yet properly understood. First ideas in this direction are presented in \cite[Section 4]{EGGKPS}. 

If $K'$ is fibered then the leading coefficient of $\Delta_{K'}(e^{p})$ equals $\pm 1$. Consider a non-fibered knot $K$ with this leading coefficient not equal to $1$. Our first observation is that Proposition \ref{prp:limitfibered}, where $U_{K}^{0}(p,t)$ is the basic disk potential, does \emph{not} hold for $K$. The reason is that, for action reasons, any non-constant holomorphic disk with boundary in $M_{K}$ contributes to the $e^{kp}$-coefficient in $U_{K}^{0}$ where $k>0$. Therefore the leading coefficient in the right-hand side of \eqref{eq:defAlexfiber} equals $1$, whereas the leading coefficient in the left-hand side is not equal to $1$. 

Since $(1-e^p)\tau_K(e^p,e^t)$ from \eqref{eq:Floertorsion} gives a $Q$-deformation of the Alexander polynomial also in the non-fibered case, it is natural to expect that this deformation can be expressed in terms of a disk potential for general $K$. Applying SFT-stretching to \eqref{eq:Floertorsion}, one expects corrections to $U_K^0$ to come from disks with negative punctures at the Reeb chords $\alpha_{ij}$ of grading 0, connecting cotangent fibers over the points $\xi_{i}$ and $\eta_{j}$. More precisely, given a collection $\epsilon$ of functions $\alpha_{ij}(t,p)$ for all Reeb chords $\alpha_{ij}$ connecting $\xi_{i}$ to $\eta_{j}$, we define the disk potential $U_{K}^{0,\epsilon}(t,p)$ as the count of generalized disks in a sufficiently SFT-stretched almost complex structure on $T^{\ast}\R^{3}$ with boundary on $M_{K}$ and negative punctures at chords $\alpha_{ij}$, where we write a factor $\alpha_{ij}(t,p)$ for each negative puncture where the disk is asymptotic to $\alpha_{ij}$. We then expect the following.

\begin{conjecture}\label{conj1}
For any knot $K$ there is a collection of functions $\epsilon = \{ \alpha_{ij}(t,p) \}$ such that the following two properties hold
\begin{enumerate}
\item The disk potential $U_{K}^{0,\epsilon}(p,t)$ is invariant under deformations and 
$x=\partial U^{0,\epsilon}_{K}/\partial p$ cuts out a branch of the augmentation variety.
\item The disk potential $U_{K}^{0,\epsilon}(p,t)$ recovers $\Delta_{K}(e^{p})$ as $t\to 0$. More precisely,
\[
\lim_{t\to 0}\,\exp\left(\partial_{t} U_{K}^{0,\epsilon}(p,t)\right)=(1-e^{p})^{-1}\Delta_{K}(e^{p}).
\]
\end{enumerate}
\end{conjecture}
Condition $(1)$ in Conjecture \ref{conj1} requires $\epsilon$ to be an augmentation of a dg-algebra generated by all the Reeb chords connecting cotangent fibers at the critical points $\xi_{i}$ and $\eta_{j}$. For more details on this, we refer to \cite[Section 4]{EGGKPS}.  

\bibliographystyle{alpha}
\bibliography{biblio}

\newcommand{\etalchar}[1]{$^{#1}$}
\begin{thebibliography}{EGG{\etalchar{+}}22}

\bibitem[AD14]{AudinDamian}
Mich\`ele Audin and Mihai Damian.
\newblock {\em Morse theory and {F}loer homology}.
\newblock Universitext. Springer, London; EDP Sciences, Les Ulis, 2014.
\newblock Translated from the 2010 French original by Reinie Ern\'{e}.

\bibitem[AENV14]{AENV}
Mina Aganagic, Tobias Ekholm, Lenhard Ng, and Cumrun Vafa.
\newblock Topological strings, {D}-model, and knot contact homology.
\newblock {\em Adv. Theor. Math. Phys.}, 18(4):827--956, 2014.

\bibitem[Asp21]{Asplund}
Johan Asplund.
\newblock Fiber {F}loer cohomology and conormal stops.
\newblock {\em J. Symplectic Geom.}, 19(4):777--864, 2021.

\bibitem[BC07]{BiranCorneaQuantum}
Paul Biran and Octav Cornea.
\newblock Quantum structures for {L}agrangian submanifolds, 2007.
\newblock arXiv:0708.4221.

\bibitem[BEH{\etalchar{+}}03]{BEHWZ}
Fr\'ed\'eric Bourgeois, Yakov Eliashberg, Helmut Hofer, Kris Wysocki, and
  Eduard Zehnder.
\newblock Compactness results in {S}ymplectic {F}ield {T}heory.
\newblock {\em Geom. Topol.}, 7(3):799--888, 2003.

\bibitem[CEL10]{ECL}
Kai Cieliebak, Tobias Ekholm, and Janko Latschev.
\newblock Compactness for holomorphic curves with switching {L}agrangian
  boundary conditions.
\newblock {\em J. Symplectic Geom.}, 8(3):267--298, 2010.

\bibitem[CELN17]{CELN}
Kai Cieliebak, Tobias Ekholm, Janko Latschev, and Lenhard Ng.
\newblock Knot contact homology, string topology, and the cord algebra.
\newblock {\em J. \'{E}c. polytech. Math.}, 4:661--780, 2017.

\bibitem[CL09]{CiliebakLatschevStringTopology}
Kai Cieliebak and Janko Latschev.
\newblock The role of string topology in symplectic field theory.
\newblock In {\em New perspectives and challenges in symplectic field theory},
  volume~49 of {\em CRM Proc. Lecture Notes}, pages 113--146. Amer. Math. Soc.,
  Providence, RI, 2009.

\bibitem[Cor17]{Cornwell}
Christopher~R. Cornwell.
\newblock K{CH} representations, augmentations, and {$A$}-polynomials.
\newblock {\em J. Symplectic Geom.}, 15(4):983--1017, 2017.

\bibitem[DRET22]{ED-RT}
Georgios Dimitroglou~Rizell, Tobias Ekholm, and Dmitry Tonkonog.
\newblock Refined disk potentials for immersed {L}agrangian surfaces.
\newblock {\em J. Differential Geom.}, 121(3):459--539, 2022.

\bibitem[EENS13]{KCH}
Tobias Ekholm, John Etnyre, Lenhard Ng, and Michael~G. Sullivan.
\newblock Knot contact homology.
\newblock {\em Geom. Topol.}, 17(2):975--1112, 2013.

\bibitem[EES05]{EkholmEtnyreSullivanOrientations}
Tobias Ekholm, John Etnyre, and Michael Sullivan.
\newblock Orientations in {L}egendrian contact homology and exact {L}agrangian
  immersions.
\newblock {\em Internat. J. Math.}, 16(5):453--532, 2005.

\bibitem[EGG{\etalchar{+}}22]{EGGKPS}
Tobias Ekholm, Angus Gruen, Sergei Gukov, Piotr Kucharski, Sunghyuk Park, and
  Piotr Su\l{}kowski.
\newblock {$\widehat{Z}$} at large {$N$}: from curve counts to quantum
  modularity.
\newblock {\em Comm. Math. Phys.}, 396(1):143--186, 2022.

\bibitem[Ekh07]{EkholmFlowTrees}
Tobias Ekholm.
\newblock Morse flow trees and {L}egendrian contact homology in 1-jet spaces.
\newblock {\em Geom. Topol.}, 11:1083--1224, 2007.

\bibitem[Ekh18]{EkICM}
Tobias Ekholm.
\newblock Knot contact homology and open {G}romov-{W}itten theory.
\newblock In {\em Proceedings of the {I}nternational {C}ongress of
  {M}athematicians---{R}io de {J}aneiro 2018. {V}ol. {II}. {I}nvited lectures},
  pages 1063--1086. World Sci. Publ., Hackensack, NJ, 2018.

\bibitem[EKL20a]{EkholmKucharskiLonghi2}
Tobias Ekholm, Piotr Kucharski, and Pietro Longhi.
\newblock Multi-cover skeins, quivers, and 3d $\mathcal{N}=2$ dualities.
\newblock {\em J. High Energ. Phys.}, 2020, 2 2020.

\bibitem[EKL20b]{EkholmKucharskiLonghi1}
Tobias Ekholm, Piotr Kucharski, and Pietro Longhi.
\newblock Physics and geometry of knots-quivers correspondence.
\newblock {\em Comm. Math. Phys.}, 379(2):361--415, 2020.

\bibitem[EL23]{EkholmLekili2}
Tobias Ekholm and Yank\i{} Lekili.
\newblock Duality between {L}agrangian and {L}egendrian invariants.
\newblock {\em Geom. Topol.}, 27(6):2049--2179, 2023.

\bibitem[EN20]{EkholmNgHigherGenus}
Tobias Ekholm and Lenhard Ng.
\newblock Higher genus knot contact homology and recursion for colored
  {HOMFLY}-{PT} polynomials.
\newblock {\em Adv. Theor. Math. Phys.}, 24(8):2067--2145, 2020.

\bibitem[ES14]{EkholmSmithmathann}
Tobias Ekholm and Ivan Smith.
\newblock Exact {L}agrangian immersions with one double point revisited.
\newblock {\em Math. Ann.}, 358(1-2):195--240, 2014.

\bibitem[ES16]{EkholmSmithjams}
Tobias Ekholm and Ivan Smith.
\newblock Exact {L}agrangian immersions with a single double point.
\newblock {\em J. Amer. Math. Soc.}, 29(1):1--59, 2016.

\bibitem[ES19]{EkholmShende}
Tobias Ekholm and Vivek Shende.
\newblock Skeins on {B}ranes, 2019.
\newblock arXiv:1901.08027.

\bibitem[FOOO09]{FOOO}
Kenji Fukaya, Yong-Geun Oh, Hiroshi Ohta, and Kaoru Ono.
\newblock {\em Lagrangian intersection {F}loer theory: anomaly and
  obstruction}, volume~46 of {\em AMS/IP Studies in Advanced Mathematics}.
\newblock American Mathematical Society/International Press, 2009.

\bibitem[Fra08]{FrauenfelderGromovCpctness}
Urs Frauenfelder.
\newblock Gromov convergence of pseudoholomorphic disks.
\newblock {\em J. fixed point theory appl.}, 215(3):215--271, 2008.

\bibitem[Fri83]{FriedHomological}
David Fried.
\newblock Homological identities for closed orbits.
\newblock {\em Invent. Math.}, 71(2):419--442, 1983.

\bibitem[GM21]{GM}
Sergei Gukov and Ciprian Manolescu.
\newblock A two-variable series for knot complements.
\newblock {\em Quantum Topol.}, 12(1):1--109, 2021.

\bibitem[GPPV20]{GPPV}
Sergei Gukov, Du~Pei, Pavel Putrov, and Cumrun Vafa.
\newblock B{PS} spectra and 3-manifold invariants.
\newblock {\em J. Knot Theory Ramifications}, 29(2):2040003, 85, 2020.

\bibitem[GPV17]{GPV}
Sergei Gukov, Pavel Putrov, and Cumrun Vafa.
\newblock Fivebranes and 3-manifold homology.
\newblock {\em J. High Energy Phys.}, (7):071, front matter+80, 2017.

\bibitem[HL99]{HutchingsLee}
Michael Hutchings and Yi-Jen Lee.
\newblock Circle-valued {M}orse theory and {R}eidemeister torsion.
\newblock {\em Geom. Topol.}, 3:369--396, 1999.

\bibitem[Hut02]{HutchingsReidemeister}
Michael Hutchings.
\newblock Reidemeister torsion in generalized {M}orse theory.
\newblock {\em Forum Math.}, 14(2):209--244, 2002.

\bibitem[Hut14]{HutchingsNotes}
Michael Hutchings.
\newblock Lecture notes on embedded contact homology.
\newblock In {\em Contact and symplectic topology}, volume~26 of {\em Bolyai
  Soc. Math. Stud.}, pages 389--484. J\'anos Bolyai Math. Soc., Budapest, 2014.

\bibitem[Lan02]{LangAlgebra}
Serge Lang.
\newblock {\em Algebra}, volume 211 of {\em Graduate Texts in Mathematics}.
\newblock Springer-Verlag, New York, third edition, 2002.

\bibitem[Lee03]{LeeTorsion}
Yi-Jen Lee.
\newblock Non-contractible periodic orbits, {G}romov invariants, and
  {F}loer-theoretic torsions, 2003.
\newblock arXiv:math/0308185.

\bibitem[Mar02]{MarkTorsion}
Thomas Mark.
\newblock Torsion, {TQFT}, and {S}eiberg--{W}itten invariants of $3$-manifolds.
\newblock {\em Geom. Topol.}, 6(1):27--58, 2002.

\bibitem[Mil62]{MilnorDuality}
John Milnor.
\newblock A duality theorem for {R}eidemeister torsion.
\newblock {\em Ann. of Math. (2)}, 76:137--147, 1962.

\bibitem[Mil65]{MilnorHcobordism}
John Milnor.
\newblock {\em Lectures on the {$h$}-cobordism theorem}.
\newblock Princeton University Press, Princeton, N.J., 1965.
\newblock Notes by L. Siebenmann and J. Sondow.

\bibitem[MS04]{McDuffSalamon2}
Dusa McDuff and Dietmar Salamon.
\newblock {\em $J$-holomorphic curves and symplectic topology}, volume~52 of
  {\em American Mathematical Society Colloquium Publications}.
\newblock American Mathematical Society, Providence, RI, 2004.

\bibitem[MW18]{MakWu}
Cheuk~Yu Mak and Weiwei Wu.
\newblock Dehn twist exact sequences through {L}agrangian cobordism.
\newblock {\em Compos. Math.}, 154(12):2485--2533, 2018.

\bibitem[Ng08]{NgFramed}
Lenhard Ng.
\newblock Framed knot contact homology.
\newblock {\em Duke Math. J.}, 141(2):365--406, 2008.

\bibitem[Ng14]{NgTopological}
Lenhard Ng.
\newblock A topological introduction to knot contact homology.
\newblock In {\em Contact and symplectic topology}, volume~26 of {\em Bolyai
  Soc. Math. Stud.}, pages 485--530. J\'anos Bolyai Math. Soc., Budapest, 2014.

\bibitem[PdM82]{MeloPalis}
Jacob Palis, Jr. and Welington de~Melo.
\newblock {\em Geometric theory of dynamical systems. An introduction}.
\newblock Springer-Verlag, New York-Berlin, 1982.
\newblock Translated from the Portuguese by A. K. Manning.

\bibitem[Sei35]{SeifertAlexander}
Herbert Seifert.
\newblock {\"U}ber das {G}eschlecht von {K}noten.
\newblock {\em Math. Ann.}, 110(1):571--592, 1935.

\bibitem[Spa16]{SpanoGromov}
Gilberto Spano.
\newblock Twisted {G}romov and {L}efschetz invariants associated with bundles,
  2016.
\newblock arXiv:1605.00624.

\bibitem[Spa17]{SpanoCategorification}
Gilberto Spano.
\newblock A categorification of the {A}lexander polynomial in embedded contact
  homology.
\newblock {\em Algebr. Geom. Topol.}, 17(4):2081--2124, 2017.

\end{thebibliography}
\end{document}